\documentclass[reqno,12pt]{amsart}
\usepackage{amssymb, amsthm, amsmath}
\usepackage{version}
\usepackage{fancybox}
\usepackage{bm}
\usepackage{color}
\usepackage{mathrsfs}
\newtheorem{definition}{\sc Definition}[section]
\newtheorem{theorem}[definition]{\sc Theorem}
\newtheorem{lemma}[definition]{\sc Lemma}

\newtheorem{proposition}[definition]{\sc Proposition}
\newtheorem{corollary}[definition]{\sc Corollary}
\newtheorem{remark}[definition]{\sc Remark}

\makeatletter

\@addtoreset{equation}{section}

\renewcommand{\d}{\text{\rm d}}
\newcommand{\vep}{\varepsilon}

\newcommand{\sgn}{{\rm sgn}}
\newcommand{\M}{\mathfrak{M}}

\newcommand{\HeinLvier}{H^1_0(\Omega) \cap L^4(\Omega)}
\newcommand{\HzweiLsechs}{H^2(\Omega) \cap L^6(\Omega)}

\newcommand{\iI}{I_{[0,\infty)}}
\newcommand{\iII}{I_{[u_0(x),\infty)}}
\newcommand{\iIIo}{I_{[u_0(x_0),\infty)}}
\newcommand{\iIIv}{I_{[v_0(x),\infty)}}
\newcommand{\iIIn}{I_{[u_{0,n}(x),\infty)}}

\newcommand{\lam}{\kappa}
\newcommand{\rr}{c_r}

\DeclareMathOperator*{\esssup}{ess\,sup}

\begin{document}

%\begin{frontmatter}

\title[]{Allen-Cahn equation with strong irreversibility}

\author{Goro Akagi}
\address[Goro Akagi]{Mathematical Institute, Tohoku University, Aoba, Sendai 980-8578 Japan ; Helmholtz Zentrum M\"unchen, Institut f\"ur Computational Biology, Ingolst\"adter Landstra\ss e 1, 85764 Neunerberg, Germany ; Technische Universität M\"unchen, Zentrum Mathematik, Boltzmannstra\ss e 3, D-85748 Garching bei M\"unchen, Germany.}
\email{akagi@m.tohoku.ac.jp}

\author{Messoud Efendiev}
\address[Messoud Efendiev]{Helmholtz Zentrum M\"unchen, Institut f\"ur Computational Biology, Ingolst\"adter Landstra\ss e 1, 85764 Neunerberg, Germany}
\email{messoud.efendiyev@helmholtz-muenchen.de}

\date{\today}

\thanks{The first author is supported by JSPS KAKENHI Grant Numbers 16H03946, 16K05199, 17H01095 and by the Alexander von Humboldt Foundation and by the Carl Friedrich von Siemens Foundation. He would also like to acknowledge the kind hospitality of the Helmholtz Zentrum M\"unchen and the Technical University of Munich during his stay in Munich.}

 \begin{abstract}
This paper is concerned with a fully nonlinear variant of the Allen-Cahn equation with strong irreversibility, where each solution is constrained to be non-decreasing in time. Main purposes of the paper are to prove the well-posedness, smoothing effect and comparison principle, to provide an equivalent reformulation of the equation as a parabolic obstacle problem and to reveal long-time behaviors of solutions. More precisely, by deriving \emph{partial} energy-dissipation estimates, a global attractor is constructed in a metric setting, and it is also proved that each solution $u(x,t)$ converges to a solution of an elliptic obstacle problem as $t \to +\infty$.
 \end{abstract}



\maketitle

\tableofcontents

\pagestyle{myheadings}

\section{Introduction}\label{S:I}

Evolution equations along with \emph{strong irreversibility} often appear in Damage Mechanics to describe a unidirectional evolution of damaging phenomena. For instance, damage accumulation and crack propagation exhibit strong irreversibility, since the  degree  of damage never decreases spontaneously. Therefore, to describe such phenomena as a phase field model, one may need to take into account the strong irreversibility (or \emph{unidirectionality}) of evolution. On the other hand, (spatial) \emph{propagation} of damage is described in terms of diffusion (type) processes. However, these two effects, namely unidirectionality of evolution and  diffusive nature,  often conflict each other. Such a conflict of two different effects may produce significant features of damaging phenomena. A few ways have been proposed to describe %the strong irreversibility (or unidirectionality) as well as  propagativity  of
damaging phenomena in view of such two effects; above all, one often employs parabolic PDEs with the \emph{positive-part} \emph{function} $(\,\cdot\,)_+ := \max \{\,\cdot\,,0\} \geq 0$ (or a \emph{negative-part} one). The simplest example reads,
\begin{equation}\label{irHeat}
u_t = \left( \Delta u \right)_+ \ \mbox{ in } \Omega \times (0,\infty), 
\end{equation}
which is a classical problem (see e.g.~\cite{K-S80}) and also still revisited by many authors (see e.g.~\cite{GiSa94,GiGoSa94} and also~\cite{AkKi13,Liu14}). Further physical backgrounds of such irreversible models will be briefly reviewed in \S \ref{S:irr}. From view points of mathematical analysis, equations such as \eqref{irHeat} are classified as \emph{fully nonlinear} PDEs, and hence, the lack of gradient structure gives rise to difficulties and particularly prevents us to reveal \emph{dissipation structure} driven by the diffusion term. Indeed, in view of the irreversible feature of the equation, dissipative behaviors of solutions may be partially inhibited. On the other hand, dissipative behaviors may also occur like a classical diffusion equation, unless they violate the strong irreversibility (see also Remark \ref{R:diss} below).

In this paper, we are concerned with the Allen-Cahn equation with strongly irreversibility,
\begin{equation}\label{iAC}
u_t = \Big( \Delta u - W'(u) \Big)_+ \ \mbox{ in } \ \Omega \times (0,\infty),
\end{equation}
where $W'(u) = u^3 - \lam u$ (with $\lam > 0$) is the derivative of a double-well potential $W(u)$ simply given by
\begin{equation}\label{potential}
W(u) := \frac 1 4 u^4 - \frac \lam 2 u^2
\end{equation}
and where $(\,\cdot\,)_+$ stands for the positive-part function and $\Omega$ is a smooth bounded domain of $\mathbb R^N$. In order to fix an idea, we shall use the simplest form \eqref{potential}; however, the most of arguments throughout the present paper can be extended to more general double-well potential functions (on the other hand, the dimensional restriction in (v) of Theorem \ref{T:ex} relies on the cubic growth). Equation \eqref{iAC} is a strongly irreversible version of the celebrated Allen-Cahn equation,
\begin{equation}\label{AC}
 u_t = \Delta u - W'(u) \ \mbox{ in } \ \Omega \times (0,\infty),
\end{equation}
which has been well studied and is known for a phase-separation model driven by the combination of double-well potential and diffusion term. Moreover, \eqref{iAC} also appears in a special setting of a phase field model describing crack-propagation (see Remark \ref{R:iAC}).

As is already pointed out, \eqref{iAC} is classified as a fully nonlinear parabolic equation, which is formulated in a general form $u_t = F(D^2 u)$ with a nonlinear function $F$ and the Hessian matrix $D^2 u$. Here we shall reformulate the equation as a generalized gradient flow (of subdifferential type), which is fitter to distributional frameworks and energy techniques. By applying the (multivalued) inverse mapping $\alpha(\,\cdot\,)$ of $(\,\cdot\,)_+$ to both sides, \eqref{iAC} is reduced to 
$$
\alpha(u_t) \ni \Delta u - W'(u) \ \mbox{ in } \ \Omega \times (0,\infty).
$$
The inverse mapping $\alpha$ of $(\,\cdot\,)_+$ can be decomposed as follows:
\begin{equation}\label{subid}
\alpha(s) = s + \partial \iI(s), \quad \partial \iI(s) = \begin{cases}
				   0 &\mbox{ if } \ s > 0\\
				   (-\infty,0] &\mbox{ if } \ s = 0\\
				   \emptyset &\mbox{ if } \ s < 0
				  \end{cases}
				  \quad \mbox{ for } \ s \in \mathbb R,
\end{equation}
where $\partial \iI$ stands for the subdifferential of the indicator function $\iI$ over the half-line $[0,+\infty)$. In the present paper, we shall particularly consider the Cauchy-Dirichlet problem for \eqref{iAC}, which is hereafter denoted by (P) and equivalently given as
\begin{alignat}{4}
 u_t + \eta - \Delta u + W'(u) = 0, \quad \eta &\in \partial \iI (u_t) \quad &\mbox{ in }& \ \Omega \times (0,\infty),\label{pde}\\
 u&= 0 \quad &\mbox{ on }& \ \partial \Omega \times (0,\infty),\label{bc}\\
 u&= u_0 \quad &\mbox{ in }& \ \Omega.\label{ic}
\end{alignat}
Furthermore, comparing \eqref{pde} with \eqref{iAC}, one can immediately find the relation,
\begin{equation}\label{mu}
 \eta = - \Big( \Delta u - W'(u) \Big)_-,
\end{equation}
where  $(\,\cdot\,)_-$ stands for the negative part function, i.e., $(s)_- := \max \{-s,0\} \geq 0$.  To be precise, such a doubly-nonlinear reformulation including the relation \eqref{mu} is justified in a strong formulation, e.g., under the frame over $L^2(\Omega)$, where equations hold in a pointwise sense; on the other hand, in a weaker formulation such as $H^{-1}$-framework, it is more delicate to verify the equivalence of two equations as well as \eqref{mu}.

Behaviors and properties of solutions to (P) can be imagined from the form of equations \eqref{iAC} and \eqref{pde}. For instance, each solution $u(x,t)$ of (P) behaves like that of the classical Allen-Cahn equation \eqref{AC} at $(x,t)$ where $\Delta u - W'(u)$ is positive. Otherwise, $u(x,t)$ never evolves. Therefore one may expect that smoothing effect and energy-dissipation \emph{partially} occur, but not everywhere. On the other hand, it is not easy to give a proof for such conjectures. Indeed, even existence and uniqueness of solutions have not yet been fully studied due to the severe nonlinearity of \eqref{iAC} and \eqref{pde}. Moreover, to the best of authors' knowledge, such partial effects of smoothing and energy-dissipation have never been studied so far. Different from classical Allen-Cahn equations such as \eqref{AC}, due to the defect of the (full) energy-dissipation structure, (P) has no absorbing set, and hence, no global attractor in any $L^p$-spaces. Indeed, from the non-decrease of $u(x,t)$ in time, i.e., $u(x,t) \geq u(x,s)$ a.e.~in $\Omega$ if $t \geq s$,  one cannot expect any dissipation estimates for the $L^p$-norm $\|u(\cdot,t)\|_{L^p(\Omega)}$, provided that $u_0 \geq 0$. On the other hand, due to the presence of a gradient structure lying inside of $(\,\cdot\,)_+$ in \eqref{iAC}, (P) shares a common Lyapunov energy with \eqref{AC},
$$
E(w) := \dfrac 1 2 \int_\Omega |\nabla w(x)|^2 \, \d x + \int_\Omega W(w(x)) \, \d x,
$$
which decreases along the evolution of solutions $u = u(x,t)$ to (P) as well as of those to \eqref{AC}. So one may expect that a partial energy-dissipation occurs (more precisely, a (quantitative) dissipative estimate for $E(u(t))$ holds in a proper sense) and it enables us to construct an absorbing set and a global attractor for (P) under a non-standard setting. However, it is unclear in which setting one can find out a partial energy-dissipation structure of (P) and establish quantitative dissipative estimates enough for a construction of a global attractor.

As we shall see in \S \ref{S:ref}, the Cauchy-Dirichlet problem (P) (equivalently, \eqref{iAC}, \eqref{bc}, \eqref{ic}) can be equivalently rewritten as an obstacle problem of parabolic type,
\begin{align*}
   u \geq u_0, \quad u_t - \Delta u + u^3 - \lam u \geq 0 \quad \mbox{ in } \ \Omega \times (0,\infty),\\
   \left( u - u_0 \right) \left(u_t - \Delta u + u^3 - \lam u \right) = 0 \quad \mbox{ in } \ \Omega \times (0,\infty),\\
   u|_{\partial \Omega} = 0, \quad u|_{t = 0} = u_0,
\end{align*}
whose obstacle function coincides with the initial datum. Such parabolic obstacle problems whose obstacle functions coincide with initial data are also studied in the context of (American) option evaluation (see~\cite{LauSal09,CafFig13} and references therein). This reformulation will play a key role to discuss long-time behaviors of solutions as well as to investigate qualitative properties, e.g., comparison principle and uniqueness (or selection principle), of solutions to (P) under milder assumptions.

The strongly irreversible evolution also exhibits a stronger dependence on initial state, compared to a classical Allen-Cahn equation. For example, solutions of \eqref{irHeat} are constrained to be not less than initial data.  Such a stronger initial-state-dependence of evolution can be found out more explicitly in the parabolic obstacle problem above. Indeed, the evolution law (= the obstacle problem) explicitly depends on initial data. Moreover, as will be illustrated below, due to the strong irreversibility, one cannot expect the existence of global attractors in a usual sense for dynamical systems (DS for short) generated by such strongly irreversible equations. Furthermore, related issues of DS (e.g., convergence to equilibria and Lyapunov stability of equilibria) must be also affected by such a strong dependence of DS on initial states. Therefore, it would be interesting to reveal the whole picture of such a peculiar dynamics. %In particular, it seems worth discussing how to extract characteristic behavior (e.g., energy-dissipation and emergence of global attractors) driven by the diffusion term against the strong irreversibility.

Main purposes of the present paper are to prove the well-posedness of (P) in an $L^2$-framework and to investigate qualitative and quantitative properties (e.g., comparison principle, smoothing effect, energy-dissipation estimates) and long-time behaviors of solutions. In particular, we shall focus on how to extract an energy-dissipation structure of \eqref{pde} beyond the obstacle arising from the strong irreversibility, and moreover, we shall discuss in which setting (e.g., phase space) one can construct a global attractor for the DS generated by (P).

In Section \ref{S:irr}, we briefly review several previous studies on strongly irreversible evolution equations (such as \eqref{iAC} and \eqref{pde}) arising from Damage Mechanics and so on.
Section \ref{S:wp} is devoted to discussing the well-posedness and a smoothing effect for (P) and providing a proof for the uniqueness and continuous dependence of solutions on initial data.
In Section \ref{S:e}, we arrange energy inequalities which will be used to prove a smoothing effect for (P) as well as to reveal long-time behaviors of solutions. In this section, one may also find out energy-dissipation structures concealed in the equation. Finally, we also give a sketch of proof for the smoothing effect, that is, the existence of solutions to (P) for a wider class of initial data. A detailed proof will be shown in Appendix \ref{S:Aex}.
In Section \ref{S:ref}, we equivalently reformulate (P) as a parabolic variational inequality of obstacle type. This fact also indicates the lack of classical regularity of solutions to (P); indeed, it is well known that solutions to (elliptic) obstacle problems are at most of class $C^{1,1}$ (see, e.g.,~\cite{Cafferelli}). The argument for justifying the reformulation is somewhat delicate and deeply related to the construction of solutions to (P); so in this section, we shall give only a formal argument, and a precise one will be given in Appendix \S \ref{A:reform}.
Moreover, in Section \ref{S:cp}, we shall discuss a comparison principle for the equation resulting from the reformulation. Furthermore, we shall obtain a uniform estimate for solutions to (P), and in particular, it will be verified that solutions of (P) enjoy a \emph{range-preserving property}, that is, if $u_0$ takes a value within a certain range, then so does $u(\cdot,t)$ for any $t > 0$. It is a fundamental requirement for phase-field models. Here the comparison principle is not directly proved for (P), since there arise some difficulties from the double nonlinearity in the $L^2$-framework (on the other hand, it can be directly proved for (P) under some additional assumptions).
Sections \ref{S:ps} and \ref{S:At} are devoted to constructing a global attractor for a DS generated by (P) in a proper sense. As mentioned above, no global attractor exists in any $L^p$-spaces, because of the strong irreversibility. Therefore, it is most crucial how to set up a phase set, which will be given by a metric space without linear and convex structures. % Roughly speaking, the set is supposed to be bounded in some direction in which the nondecreasing constraint is less effective (more precisely, the constraint set $[\,\cdot\, \geq u_0] = \{u \in L^2(\Omega) \colon u(x) \geq u_0(x) \ \mbox{ for a.e.~} x \in \Omega\}$ is unbounded for each initial datum $u_0$); on the other hand, the set is still unbounded in other directions in which energy dissipates.  
In Section \ref{S:conv}, we shall prove the convergence of each solution $u(x,t)$ for (P) as $t \to \infty$ and characterize the limit as a solution of an elliptic variational inequality of obstacle type. Also here, the reformulation exhibited in \S \ref{S:ref} will play a crucial role to characterize equilibria.
In Appendix \S \ref{S:Aex}, we particularly give a detailed proof for the existence of solutions for (P) along with rigorous derivations of energy inequalities, which will be derived in \S \ref{S:e} with formal arguments. This part would be of independent interest in view of studies on nonlinear evolution equations as well. In particular, it is noteworthy that a smoothing effect is proved for the doubly nonlinear evolution equation \eqref{pde}, although there are only few results on smoothing effects for doubly nonlinear evolution equations of the form $A(u') + B(u) = 0$. In Appendix \S \ref{A:reform}, the reformulation of (P), which is formally discussed in \S \ref{S:ref}, will be proved rigorously.\\

\noindent
{\bf Notation.} We denote by $\|\cdot\|_p$, $1 \leq p \leq \infty$ the $L^p(\Omega)$-norm, that is, $\|f\|_p := (\int_\Omega |f(x)|^p \, \d x)^{1/p}$ for $p \in [1,\infty)$ and $\|f\|_\infty := \esssup_{x \in \Omega} |f(x)|$. Denote also by $(\cdot,\cdot)$ the $L^2$-inner product, i.e., $(u,v) := \int_\Omega u(x) v(x) \, \d x$ for $u,v \in L^2(\Omega)$. For each normed space $X$ and $T > 0$, $C_w([0,T];X)$ denotes the space of weakly continuous functions on $[0,T]$ with values in $X$. We also simply write $u(t)$ instead of $u(\cdot,t)$, which is regarded as a function from $\Omega$ to $\mathbb R$, for each fixed $t \geq 0$. Here and henceforth, we use the same notation $\iI$ for the indicator function over the half-line $[0,\infty)$ as well as for that defined on $L^2(\Omega)$ over the closed convex set $K := \{u \in L^2(\Omega) \colon u \geq 0 \ \mbox{ a.e.~in } \Omega\}$, namely,
$$
\iI(u) = \begin{cases}
	  0 &\mbox{ if } \ u \in K,\\
	  \infty &\mbox{ otherwise }
	 \end{cases}
	 \quad \mbox{ for } \ u \in L^2(\Omega),
	 $$
	 if no confusion may arise. Moreover, let $\partial \iI$ also denote the subdifferential operator (precisely, $\partial_{\mathbb R} \iI$) in $\mathbb R$ (see~\eqref{subid}) as well as that (precisely, $\partial_{L^2} \iI$) in $L^2(\Omega)$ defined by
	 $$
	 \partial_{L^2} \iI(u) = \left\{ \eta \in L^2(\Omega) \colon (\eta, u - v) \geq 0 \ \mbox{ for all } \ v \in K \right\} \quad \mbox{ for } \ u \in K.
	 $$
Here, we note that these two notions of subdifferentials are equivalent each other in the following sense: for $u, \eta \in L^2(\Omega)$,
	 $$
	 \eta \in \partial_{L^2} \iI(u) \quad \mbox{ if and only if } \quad \eta(x) \in \partial_{\mathbb R} \iI(u(x)) \ \mbox{ a.e.~in } \Omega
	 $$
	 (see, e.g.,~\cite{HB1,HB3}).
We denote by $C$ a non-negative constant, which does not
depend on the elements of the corresponding space or set and may vary from
line to line.

\section{Evolution equations with strong irreversibility}\label{S:irr}

Evolution equations including the positive-part function such as \eqref{irHeat} and \eqref{iAC} have been studied in several papers and they play important roles particularly in Damage Mechanics. In this section, we briefly review some of those models and related nonlinear PDEs including the positive-part function.

\subsection{Quasi-static brittle fracture models}

Francfort and Marigo~\cite{FraMar98} proposed a quasi-static evolution of brittle fractures in elastic bodies based on Griffith's criterion (see also~\cite{DMasoToader} and~\cite{Francfort}). Let $\Omega \subset \mathbb R^3$ be an elastic body and let $\Gamma_n \subset \overline \Omega$ be a crack at time $t_n$. Then the crack $\Gamma_{n+1} \subset \overline\Omega$ and the displacement $\vec u_{n+1} : \Omega \setminus \Gamma_{n+1} \to \mathbb R^3$ at time $t_{n+1}$ are obtained as a minimizer of the elastic energy,
 $$
 \mathcal F(\vec u,\Gamma) = \underbrace{\int_{\Omega \setminus \Gamma} \mu |\vep(\vec u)|^2 + \lambda |\mathrm{tr}\,\vep(\vec u)|^2 \, \d x}_{\text{bulk energy}} + \underbrace{\mathcal H^2(\Gamma)}_{\text{surface energy}}
 $$
 among $\Gamma \subset \overline\Omega$ including $\Gamma_n$ and $\vec u : \Omega \setminus \Gamma \to \mathbb R^3$ satisfying a boundary condition $\vec u |_{\partial \Omega} = \vec g$ associated with the external load $\vec g$ on (some part of) the boundary. Here $\vep(\vec u)$ is the symmetric part of the gradient matrix of $\vec u$, $\lambda, \mu > 0$ and $\mathcal{H}^2$ denotes the two-dimensional Hausdorff measure. Furthermore, concerning the mode III (i.e., anti-planar shear) crack growth, the displacement vector $\vec u = \vec u(x)$ is reduced to a scalar-valued function $u = u(x)$ of class $SBV(\Omega)$ (see~\cite{FraLar03}). In order to perform numerical analysis of the mode III crack propagation, $\mathcal F$ is often regularized as the Ambrosio-Tortorelli energy (see~\cite{AmTo90,AmTo92}),
$$
\mathcal F_\vep(u,z) = \dfrac \mu 2 \int_\Omega (1 - z)^2 |\nabla u|^2 \, \d x
+ \int_\Omega f u \, \d x
+ \int_\Omega \gamma(x) \left( \dfrac{|\nabla z|^2}{2 \vep} + \vep V(z) \right) \, \d x,
$$
where $u$ and $z$ stand for the deformation of the material and a phase parameter describing the  degree  of crack (e.g., $z = 1$ means ``completely cracked'' configuration), respectively, $V(\cdot)$ is a potential function, $\vep > 0$ is a relaxation parameter (which is also related to the thickness of the diffuse interface) and $\mu$ is a positive constant and $\gamma(x)$ denotes the \emph{fracture toughness} of the material. It is proved in~\cite{AmTo90,AmTo92} that $\mathcal F_\vep$ converges to the Francfort-Marigo energy in the sense of $\Gamma$-convergence as $\vep \to 0$. Quasi-static dynamics of the approximated brittle fracture model is also studied by introducing a constrained minimization scheme associated with $\mathcal F_\vep$ (see~\cite{Giacom05}). Here we stress again that the evolution of the phase parameter $z(x,t)$ is supposed to be monotone (i.e., non-decreasing in time).

A couple of nonlinear evolution equations have been also proposed to describe (or approximate) quasi-static evolution of brittle fractures.  Above all, Kimura and Takaishi~\cite{K-T10,T-K09} developed a crack propagation model for numerical simulation. Their model is derived as a \emph{double} gradient flow (i.e., in both variables $(u,z)$) for $\mathcal F_\vep(u,z)$:
\begin{alignat*}{3}
\alpha_1 u_t & = \mu \mathrm{div} \left( (1-z)^2 \nabla u \right) + f(x,t) \quad &\mbox{ in }& \Omega \times (0,\infty),\\
\alpha_2 z_t & = \left( \vep \mathrm{div} (\gamma(x) \nabla z) - \dfrac{\gamma(x)}\vep V'(z) + \mu |\nabla u|^2 (1-z) \right)_+ \quad &\mbox{ in }& \Omega \times (0,\infty)
\end{alignat*}
where $\alpha_1$, $\alpha_2$ are positive constants (related to numerical efficiency), together with boundary and initial conditions (see also~\cite{BaMi14}). Here we remark that the second equation of the system above includes the positive-part function in the right-hand side, in order to reproduce the non-decreasing (in time) evolution of the phase parameter $z(x,t)$.

\begin{remark}\label{R:iAC}
 {\rm  
 Equation \eqref{iAC} can be derived as an extreme case of the quasi-static model for the regularized energy. More precisely, let $V(\cdot)$ be a double-well potential, $V(s) = s^4/4 - \lam s^2/2$, to confine the phase parameter into an interval (see~\cite{VoyMoz13}). %Then the damaged and undamaged states ($z = \pm 1$) are supposed to be (locally) stable (cf.~see~\cite{Giacom05}). Such a situation may occur in some eroded materials, where materials are basically stable in the undamated state; however, when the  degree  of damage is increasing and exceeds the some critical value, it turns to be easily breakable.
 Moreover, for (mathematical) simplicity, set $\gamma(x) \equiv 1$, $f \equiv 0$, $\vep = 0$ and take $\alpha_1 = 0$ and $\alpha_2 = 1$ to the double-gradient flow model. Testing the first equation by $u$ and integrating by parts, we see that
 $$
 (1 - z)^2|\nabla u|^2 = 0 \quad \mbox{ for a.e. } (x,t) \in \Omega \times (0,\infty),
 $$
 which means that either $(1-z)$ or $|\nabla u|$ is zero a.e.~in $\Omega \times (0,\infty)$. Hence the system is reduced to the single equation \eqref{iAC}.
 } 
\end{remark}

\subsection{Damage accumulation models}

Barenblatt and Prostokishin~\cite{BaPr93} proposed a damage accumulation model, which derives the following fully nonlinear parabolic PDE including the positive-part function:
$$
u_t = u^\alpha \left( u_{xx} + \lam u \right)_+ \ \mbox{ in } \ (a,b) \times (0,\infty)
$$
with parameters $\alpha > 1$, $\lam > 0$ to describe the evolution of damage factor, that is, an internal variable used in the Kachanov theory~\cite{Kachanov}. In this model, the positive-part function plays a role to impose the non-decreasing constraint on the evolution of the damage factor, since the time derivative of $u(x,t)$ is non-negative. Their model was mathematically studied by Bertsch and Bisegna in~\cite{BeBi97}, where the solvability of the initial-boundary value problem is proved in a classical framework and long-time behaviors of solutions are also investigated. In particular, it is proved that the \emph{regional} blow-up phenomena occur (i.e., the blow-up set of a solution is a subinterval of $(a,b)$; however, it is neither a point set nor the whole of the interval) under suitable assumptions on $\lambda$, $\alpha$ and the interval $(a,b)$ (see also~\cite{A14}). 

\subsection{Irreversible evolution equations governed by subdifferentials}

As is explained in \S \ref{S:I}, the strongly irreversible evolution can be also described in terms of the subdifferential operator $\partial \iI$ of the indicator function over the half-line. In what follows, we shall recall strongly irreversible evolution equations formulated in such a way. Let us start with a \emph{rate-independent} unidirectional flow along with the Ambrosio-Tortorelli energy (see Knees, Rossi and Zanini \cite{KRZ13,KRZ15} and references therein, e.g.,~\cite{EfMi06}). In~\cite{KRZ13,KRZ15}, they discussed the existence of solutions to the Cauchy problem for the rate-independent evolution equation,
$$
\partial \mathcal R(z_t) + \mathrm{D}_z \mathcal F_\vep(u,z) \ni 0, \quad 0 < t < T, \quad u = \arg \min_{v} \mathcal F_\vep(v,z),
$$
where $\mathrm{D}_z$ denotes a functional derivative (e.g., Fr\'echet derivative) of $\mathcal F_\vep$ with respect to the second variable $z$, with a $1$-positively homogeneous and unidirectional dissipation functional $\mathcal R$ given by
$$
%\mathcal R_1(\eta) = \begin{cases}
%		      \int_\Omega \kappa |\eta(x)| \, \d x &\mbox{ if } \ \eta \leq 0 \ \mbox{ a.e.~in } \Omega,\\
%		      \infty &\mbox{ otherwise,}
%		     \end{cases}
\mathcal R(\eta) = \int_\Omega \kappa |\eta(x)| \, \d x + \iI(\eta(x))
		     \quad \mbox{ for } \ \eta \in L^1(\Omega)
		     $$
for some $\kappa > 0$ (more precisely, in~\cite{KRZ13}, a modified Ambrosio-Tortorelli energy is treated). 
		   
\begin{comment}
and an energy functional $\mathcal E$ defined by
$$
\mathcal E(t,u,z) := \dfrac 1 2 a_s(z,z) + \int_\Omega f(z) \, \d x + \dfrac 1 2 \int_\Omega g(z) \mathbb C \vep(u+u_D(t)):\vep(u+u_D(t)) \, \d x - \langle \ell(t),u \rangle,
$$
where $a_s(\cdot,\cdot)$ is a bilinear form associated with a seminorm of $H^s(\Omega)$ for $s \in \{1,3/2\}$, $\vep(u)$ is the symmetrized strain tensor and $\mathbb C$ is an elasticity tensor. Due to the unidirectional constraint on the derivative $z'(t)$, the evolution of the phase parameter $z(t)$ must be non-increasing.
\end{comment}

Strongly irreversible evolution equations also appear in other topics. For instance, the following \emph{irreversible phase transition} model is proposed by Fr\'emond and studied in~\cite{Fr1,Fr2},
\begin{alignat*}{3}
&\theta_t - \theta \chi_t - \Delta \theta = \chi_t^2 \quad &\mbox{ in }& \ \Omega \times (0,\infty),\\
&\chi_t + \partial \iI (\chi_t) - \Delta \chi + \beta(\chi) \ni \theta - \theta_c \quad &\mbox{ in }& \ \Omega \times (0,\infty),
\end{alignat*}
where $\theta$ and $\chi$ denote the absolute temperature ($\theta_c$ is a transition temperature) and a phase parameter, respectively, and moreover, $\beta$ is a maximal monotone graph in $\mathbb R^2$. Due to the presence of the subdifferential term $\partial \iI(\chi_t)$, the evolution of $\chi$ is constrained to be non-decreasing. We refer the reader to~\cite{LSS02} and references therein for mathematical analysis of the model. Moreover, Aso and Kenmochi~\cite{AK05} (see also~\cite{AsFrK}) studied the existence of solutions for a quasivariational evolution inequality of reaction-diffusion type such as
\begin{alignat*}{4}
& \theta_t - \Delta \theta + k(\theta,w) = h(t,x) \quad &\mbox{ in }& \ \Omega \times (0,\infty),\\
& w_t + \partial I_{[g(\theta),\infty)}(w_t) - \Delta w + \ell(\theta,w) \ni q(t,x) \quad &\mbox{ in }& \ \Omega \times (0,\infty),
\end{alignat*}
where $k$ and $\ell$ are Lipschitz continuous functions in both variables, $g$ is a smooth nonnegative function and $h$ and $q$ are given functions in a suitable class. These systems are also reduced to \eqref{iAC} in an isothermal setting, i.e., $\theta = $ constant (with suitable assumptions). Furthermore, we also refer the reader to references~\cite{BoSc04,RR15} and references therein.

 Equations reviewed in this section have been studied mostly in view of well-posedness. On the other hand, qualitative and quantitative analysis on behaviors of solutions is still left to be open, since the equations are somewhat complicated and also have several different complexities. So the study on intrinsic phenomena arising from the strong irreversibility have not yet been fully pursued. In the present paper, we shall treat a simpler equation, \eqref{iAC}, but investigate various properties and behaviors of solutions as well as the well-posedness of (P) in order to find out intrinsic features of parabolic PDEs with the positive-part function. %As will be shown, equation \eqref{iAC} contains intrinsic features of strongly irreversible models mentioned above, even if it may not be an exact form of the original equation(s).

\begin{comment}
 As mentioned above, several contributions have been made for evolution equations with strong irreversibility. However, a common feature over the equations treated in these papers resides in the presence of a positive/negative-part function or equivalently the subdifferential operator $\partial \iI$ of an indicator function over the half-line. Particularly, the evolution of solution  strongly depends on the initial state, in contrast with many nonlinear dissipative evolutionary problems, where solutions enjoy dissipation or decay properties (in various norms and quantities) and the corresponding DS possesses a global attractor, i.e., a compact set attracting all trajectories. On the other hand, solutions for (P) are restricted to be non-decreasing, and therefore, at least any $L^p$-norms of solution never decrease nor dissipate as $t \to \infty$. Hence in some topologies, one may never expect the existence of such a \emph{compact} set attracting all trajectories. On the other hand, one may simultaneously expect some dissipative structure of (P) in a suitable topology arising from the parabolicity of the equation. One of main purposes of this research is to reveal such a dissipative structure of the equation hidden behind the strong dependence on initial data and to develop a method to measure the dissipativity.
\end{comment}
 
\section{Existence of $L^2$ solutions}\label{S:wp}

The $L^2(\Omega)$-solvability of (P) ($=$ \{\eqref{iAC}, \eqref{bc}, \eqref{ic}\}) can be ensured for smooth data by applying a general theory due to Barbu~\cite{Barbu75} and Arai~\cite{Arai}; more precisely, for any $u_0 \in H^2(\Omega) \cap H^1_0(\Omega) \cap L^6(\Omega)$, Problem (P) possesses at least one \emph{$L^2(\Omega)$-solution} $u = u(x,t)$ defined by
\begin{definition}\label{D:sol}
 A function $u \in C([0,\infty);L^2(\Omega))$ is said to be a \emph{solution} {\rm (}or an $L^2(\Omega)$-solution{\rm )} of {\rm (P)}, if the following conditions are all satisfied\/{\rm :}
 \begin{enumerate}
  \item $u$ belongs to $W^{1,2}(\delta,T;L^2(\Omega))$, $C([\delta,T];H^1_0(\Omega) \cap L^4(\Omega))$, $L^2(\delta,T;H^2(\Omega))$ and $L^6(\delta,T;L^6(\Omega))$ for any $0 < \delta < T < \infty$,
  \item there exists $\eta \in L^\infty(0,\infty;L^2(\Omega))$ such that
       \begin{equation}\label{EQ}
	u_t + \eta - \Delta u + u^3 - \lam u = 0, \quad \eta \in \partial \iI(u_t)
	 \ \mbox{ for a.e. } (x,t) \in \Omega \times (0,\infty)
       \end{equation}
	and $\eta =  - \big( \Delta u - u^3 + \lam u \big)_-  $ for a.e. $(x,t) \in \Omega \times (0,\infty)$. Hence $u$ also solves \eqref{iAC} a.e.~in $\Omega \times (0,\infty)$.
  \item $u(\cdot,0) = u_0$ a.e.~in $\Omega$.
 \end{enumerate}
\end{definition}

On the other hand, the uniqueness of solutions does not follow from general theories. Furthermore, by focusing on specific structures of the equation \eqref{pde}, we shall improve the result above on the $L^2(\Omega)$-solvability.  More precisely, we shall prove a \emph{smoothing effect} for (P), that is,  even if initial data belong to a closure of a set $D$ (of more regular functions), corresponding solutions belong to the set $D$ instantly.  To state more details, let us introduce a set
\begin{align*}
 D_r := \Big\{
 u \in H^2(\Omega) \cap H^1_0(\Omega) \cap L^6(\Omega)
 \colon
 \|(\Delta u - u^3 + \lam u)_-\|_2^2 \leq r
 \Big\}
\end{align*}
for each $r > 0$.  Here we stress that $D_r$ is an unbounded set. Indeed, let $z \in C^2(\Omega) \cap C(\overline\Omega)$ be the negative solution of the classical elliptic Allen-Cahn equation,
\begin{equation}\label{cAC-e}
-\Delta z + z^3 - \lam z = 0 \ \mbox{ in } \Omega, \quad z = 0 \ \mbox{ on } \partial \Omega.
\end{equation}
Then any multiple $w = cz$ of $z$ satisfies $\Delta w - w^3 + \lam w \geq 0$ a.e.~in $\Omega$, provided that $c \geq 1$. Then $w$ belongs to $D_r$, and therefore, $D_r$ is unbounded.  

Now, let us state a theorem on the well-posedness and smoothing effect.

\begin{theorem}[Well-posedness and smoothing effect]\label{T:ex}
 Let $r > 0$ be arbitrarily fixed.
 \begin{enumerate}
  \item[(i)] Let $u_0$ belong to the closure $\overline{D_r}^{L^2}$ of $D_r$ in $L^2(\Omega)$. Then {\rm (P)} admits a solution $u = u(x,t)$ satisfying
 \begin{align*}
   u \in L^2(0,T;H^1_0(\Omega)) \cap L^4(0,T;L^4(\Omega)),  \\
   t^{1/2} u_t \in L^2(0,T;L^2(\Omega)), \quad t u_t \in L^2(0,T; H^1_0 (\Omega)), \\
   t^{1/2} u \in L^\infty(0,T;H^1_0(\Omega)) \cap L^2(0,T;H^2(\Omega)),  \\
   t^{1/4} u \in L^\infty(0,T;L^4(\Omega)), \quad t^{1/6} u \in L^6(0,T;L^6(\Omega)), \\
   t^{1/3} u \in L^\infty(0,T;L^6(\Omega)), \quad tu \in L^\infty(0,T;H^2(\Omega)),\\
  u \in C_w((0,T];H^2(\Omega) \cap L^6(\Omega))  \cap C((0,T];H^1_0(\Omega) \cap L^4(\Omega)), \\
  u(t) \in D_r \ \mbox{ for all } \ t \in (0,T]
\end{align*}
for any $0 < T < \infty$.
  \item[(ii)] If $u_0$ belongs to the closure $\overline{D_r}^{H^1_0 \cap L^4}$ of $D_r$ in $H^1_0(\Omega) \cap L^4(\Omega)$, then it further holds that
 \begin{align*}
   u \in W^{1,2}(0,T;L^2(\Omega)) \cap L^2(0,T;H^2(\Omega)) \cap L^6(0,T;L^6(\Omega)),  \\
  u \in C([0,T];H^1_0(\Omega) \cap L^4(\Omega)), \quad
  t^{1/2} u_t \in L^2(0,T; H^1_0 (\Omega)),\\
   t^{1/2} u \in L^\infty(0,T;H^2(\Omega)), \quad t^{1/6} u \in L^\infty(0,T;L^6(\Omega)) 
 \end{align*}
 for any $0 < T < \infty$.
  \item[(iii)] If $u_0 \in H^2(\Omega) \cap H^1_0(\Omega) \cap L^6(\Omega)$, then $u \in C_w([0,T];H^2(\Omega) \cap L^6(\Omega))$ and $u_t \in L^2(0,T; H^1_0  (\Omega))$ for any $0<T<\infty$.
  \item[(iv)] Let $T > 0$ be fixed. For $N \leq 3$, $L^2(\Omega)$-solutions $u$ belonging to  the class
	      \begin{equation}\label{uniq-class}
		C([0,T];H^1_0(\Omega)) 
	      \end{equation}
	       are uniquely determined by initial data $u_0 \in H^1_0(\Omega)$ and they continuously depend on initial data $u_0$ in the following sense\/{\rm :} let $u_i$ be the unique solution of {\rm (P)} for the initial data $u_{0,i} \in H^1_0(\Omega)$ {\rm (}for $i = 1,2${\rm )} and set $w := u_1 - u_2$. Then there exists a constant $C > 0$ which depends only on $\sup_{t \in (0,T)}\|\nabla u_i(t)\|_2$ {\rm (}$i = 1,2${\rm )} such that
 \begin{align}\label{conti-dep}
 \|w(t)\|_2^2 + \|\nabla w(t)\|_2^2 \leq \left(\|w(0)\|_2^2 + \|\nabla w(0)\|_2^2\right) e^{Ct}
 \end{align}
	      for all $t \in [0,T]$.
  \item[(v)] For $N \leq 4$, $L^2(\Omega)$-solutions belonging to \eqref{uniq-class} and
	     \begin{equation}\label{uniq-class-2}
	      L^2(0,T;H^2(\Omega)) \cap L^6(0,T;L^6(\Omega))
	     \end{equation}
	     are uniquely determined by initial data $u_0 \in H^1_0(\Omega)$ and \eqref{conti-dep} holds true with a constant $C$ depending only on $\|u_i\|_{L^2(0,T;H^2(\Omega))}$ and $\|u_i\|_{L^6(0,T;L^6(\Omega))}$ {\rm (}$i = 1,2${\rm )}. 
  \item[(vi)] Furthermore, for general $N$, $L^2(\Omega)$-solutions belonging to $L^\infty(\Omega \times (0,T))$  as well as \eqref{uniq-class}  are uniquely determined  by initial data $u_0 \in H^1_0(\Omega)$  and they satisfy \eqref{conti-dep} with a constant $C$ which depends only on uniform bounds $\|u_i\|_{L^\infty(\Omega \times (0,T))}$ of solutions $u_i$ {\rm (}$i = 1,2${\rm )}.
 \end{enumerate}
\end{theorem}

\begin{remark}[Invariance of the set $D_r$]{\rm
 Thanks to (i), the set $D_r$ (and its closures) turns out to be invariant under the evolution generated by (P) (see also \eqref{e3} below). Hence $D_r$ will play a role of a phase space in order to investigate the dynamics of solutions to (P) (see \S \ref{S:ps} and \S \ref{S:At}).
}\end{remark}

\begin{remark}[Difference between $D_r$ and its closure]\label{R:se}
 {\rm
To observe how smoothing effect occurs (in Theorem \ref{T:ex}), let us consider the following two examples (with $N = 1$, $\Omega = (-1,1)$ and $\lam = 1$ for simplicity):
\begin{enumerate}
  \item[(i)] Set $u_0(x) = |x| - 1 \in H^1_0(-1,1) \setminus H^2(-1,1)$. Then define $u_{0,\vep} \in W^{2,\infty}(-1,1)$ by
	    $$
	    u_{0,\vep}(x) = \begin{cases}
			     |x| - 1 &\mbox{ if } |x| > \vep,\\
			     \frac 1 \vep \frac{x^2}2 + \frac \vep 2 - 1
			     &\mbox{ if } |x| \leq \vep
			    \end{cases}
	    $$
	    for $\vep > 0$. Then one observes that
	    $$
	    u_{0,\vep}'' - u_{0,\vep}^3 + u_{0,\vep}
	    = \begin{cases}
	       - u_{0,\vep}^3 + u_{0,\vep} < 0 &\mbox{ if } |x| > \vep,\\
	       \frac 1 \vep \underbrace{- u_{0,\vep}^3 + u_{0,\vep}}_{\text{close to zero}} > 0 &\mbox{ if } |x| \leq \vep
	      \end{cases}
	    $$
	    for $\vep > 0$ small enough. Therefore
	    $$
	    \| (u_{0,\vep}'' - u_{0,\vep}^3 + u_{0,\vep})_-\|_2^2
	    = \int_{|x| > \vep} (u_{0,\vep}^3 - u_{0,\vep})^2 \, \d x
	    \leq \|u_0^3 - u_0\|_2^2 =: r < + \infty.
	    $$
	     Moreover, one can check that $u_{0,\vep} \to u_0$ strongly in $H^1_0(-1,1)$. Hence $u_0$ belongs to the closure of $D_r$ in $H^1_0(-1,1)$. However, $u_0$ does not belong to $D_r$ ($\subset H^2(-1,1)$). On the other hand, by Theorem \ref{T:ex}, $u(x,t)$ belongs to (at least) $H^2(-1,1) \subset C^{1+\alpha}([-1,1])$ at any $t > 0$. Therefore the sharp edge of $u_0(x)$ at $x = 0$ instantly vanishes.
 \item[(ii)] Set $u_0(x) \equiv -1 \in L^2(-1,1) \setminus H^1_0(-1,1)$ (hence $u_0$ violates the homogeneous Dirichlet condition) and define approximated data by
	     $$
	     u_{0,\vep}(x) = \begin{cases}
			      -1 &\mbox{ if } |x| < 1 - \vep,\\
			      -1 + \frac 1 {\vep^2} (|x| - 1 + \vep)^2 &\mbox{ if } |x| \geq 1 - \vep.
			     \end{cases}
	     $$
	     Then $u_{0,\vep} \in H^2(-1,1) \cap H^1_0(-1,1)$ and $u_{0,\vep} \to u_0$ strongly in $L^2(-1,1)$ as $\vep \to 0$. Moreover, we observe that
	     $$
	     u_{0,\vep}'' - u_{0,\vep}^3 + u_{0,\vep}
	     = \begin{cases}
		0 &\mbox{ if } |x| < 1 - \vep,\\
		\frac 2 {\vep^2} - u_{0,\vep}^3 + u_{0,\vep} > 0 &\mbox{ if } |x| \geq 1 - \vep,
	       \end{cases}
	     $$
	     which yields $\|(u_{0,\vep}'' - u_{0,\vep}^3 + u_{0,\vep})_-\|_2^2 = 0$. Hence $u_0$ belongs to the closure of $D_r$ in $L^2(-1,1)$ (but $u_0 \not\in D_r$). Since the solution to (P) satisfies the boundary condition $u(\pm 1,t) = 0$ for any $t > 0$ by Theorem \ref{T:ex}, the values of $u(\pm 1, t)$ jump to $0$ from $-1$ at $t = 0$.
\end{enumerate}
 }
\end{remark}

\begin{proof}[Proof of {\rm (iv)}--{\rm (vi)}]
Before starting a proof for (iv), we remark that the uniqueness of solutions for (P) is not ensured by the abstract results (e.g., Arai~\cite{Arai}, Colli-Visintin~\cite{CV}, Colli~\cite{Colli}, Visintin~\cite{Visintin}). For instance, in~\cite{CV,Colli}, the uniqueness is proved for (abstract) doubly nonlinear equations, $A(u_t) + B(u) \ni 0$, provided that either $A$ or $B$ is linear. %We shall treat a nonlinear (monotone) operator (corresponding to a part of the operator $B$) as a perturbation.

 Fix $\delta > 0$ arbitrarily.  Let $u_i$ ($i = 1,2$) be two solutions for (P) belonging  to \eqref{uniq-class}  with initial data $u_{0,i} \in H^1_0(\Omega)$ ($i = 1,2$) and set $w := u_1 - u_2$. Then
$$
w_t + \eta_1 - \eta_2 - \Delta w + u_1^3 - u_2^3 = \lam w,
$$
where $\eta_i$ is a section of $\partial \iI(\partial_t u_i)$ for $i = 1,2$. Test both sides by $w_t$ and employ the monotonicity of $\partial \iI$ to find that
$$
\|w_t\|_2^2 + \dfrac 1 2 \dfrac \d {\d t} \|\nabla w\|_2^2
\leq \dfrac \lam 2 \dfrac \d {\d t} \|w\|_2^2 - \left( u_1^3 - u_2^3, w_t\right)
\leq \dfrac \lam 2 \dfrac \d {\d t} \|w\|_2^2 + C \|u_1^3 - u_2^3\|_2^2 + \dfrac 1 2 \|w_t\|_2^2
 $$
  for a.e.~$t \in (\delta,T)$ (see Definition \ref{D:sol}).  Here note that, for any $\vep > 0$, there exists a constant $C_\vep > 0$ such that
$$
\dfrac 1 2 \dfrac \d {\d t} \|w\|_2^2
= (w_t, w) \leq \vep \|w_t\|_2^2 + C_\vep \|w\|_2^2.
$$
Therefore choosing $\vep > 0$ small enough, one obtains
\begin{equation}\label{uniq-0}
\alpha \dfrac \d {\d t} \|w\|_2^2 + \dfrac 1 2 \dfrac \d {\d t} \|\nabla w\|_2^2
\leq C_\vep \|w\|_2^2 + C \|u_1^3 - u_2^3\|_2^2
\end{equation}
 for some $\alpha > 0$. In case $N \leq 3$, thanks to the Mean-Value Theorem and Sobolev's embedding $H^1_0(\Omega) \hookrightarrow L^6(\Omega)$, it follows that
\begin{equation}\label{loc-Lip}
\|u_1^3 - u_2^3\|_2^2 \leq C \left( \|\nabla u_1\|_2^4 + \|\nabla u_2\|_2^4 \right)\|\nabla w\|_2^2.
\end{equation}
Thus Gronwall's inequality yields
\begin{equation}\label{cdi}
\|w(t)\|_2^2 + \|\nabla w(t)\|_2^2 \leq \left(\|w( \delta )\|_2^2 + \|\nabla w( \delta )\|_2^2\right) e^{C_0  (t - \delta) } \quad \mbox{ for all } \ t \geq \delta,
\end{equation}
where $C_0$ is a constant depending only on  $\sup_{t \in (0,T)}\|\nabla u_i(t)\|_2$. From the fact that $u_i \in C([0,T];H^1_0(\Omega))$, one can pass to the limit as $\delta \to 0_+$ and obtain \eqref{cdi} with $\delta = 0$.  If $w(0) = 0$, then we conclude that $w \equiv 0$, i.e., $u_1 \equiv u_2$. This completes a proof of (iv).

 Concerning (v), for any $3 \leq N \leq 5$ (then $H^2(\Omega) \subset L^{2N}(\Omega)$), by Gagliardo-Nirenberg's inequality we infer that
 \begin{align*}
  \|u_1^3 - u_2^3\|_2^2 &\leq C \left( \|u_1\|_{2N}^4 + \|u_2\|_{2N}^4 \right) \|\nabla w\|_2^2\\
  &\leq C \left( \|u_1\|_{H^2(\Omega)}^{4\theta} \|u_1\|_6^{4(1-\theta)} +  \|u_2\|_{H^2(\Omega)}^{4\theta} \|u_2\|_6^{4(1-\theta)}\right) \|\nabla w\|_2^2,
 \end{align*}
 where $\theta$ is given by
 $$
 \dfrac 1 {2N} = \theta \left( \dfrac 1 2 - \dfrac 2 N\right) + \frac{1-\theta}6.
 $$
 Furthermore, assuming $N \leq 4$, one finds that
 $$
 2 \theta + \dfrac{2(1-\theta)}3 \leq 1,
 $$
 which yields
 $$
 \|u_i\|_{H^2(\Omega)}^{4\theta} \|u_i\|_6^{4(1-\theta)}
 %\leq (\|u_i\|_{H^2(\Omega)} + 1)^2 + (\|u_i\|_6 + 1)^6
 \in L^1(0,T)
 $$
 for $i = 1,2$. Therefore by Gronwall's inequality, we can obtain the desired conclusion. 
 
 To prove (vi), the argument above can be also generalized for general dimension $N$ by assuming the boundedness of solutions, i.e., $u_i \in L^\infty(Q)$ ($i=1,2$) with $Q = \Omega \times (0,T)$, and by replacing \eqref{loc-Lip} by
 $$
\|u_1^3 - u_2^3\|_2^2 \leq C \left( \|u_1\|_\infty^4 + \|u_2\|_\infty^4 \right)\|w\|_2^2.
 $$
 Therefore for general $N$, bounded solutions are uniquely determined by initial data and a similar inequality to \eqref{cdi} holds with a constant $C_0$ depending on $\|u_i\|_{L^\infty(Q)}$. Thus (vi) is proved.
\end{proof}

Before giving a (sketch of) proof for the existence part (i)--(iii) of Theorem \ref{T:ex}, we shall (formally) derive energy estimates in the next section. 

 \section{Energy inequalities and partial energy-dissipation estimates}\label{S:e}

In this section, we first collect key energy inequalities, which will play a crucial role later; in particular, we shall derive \emph{partial energy-dissipation estimates} and apply them to construct a global attractor in a peculiar setting (cf. as we mentioned in \S \ref{S:I}, due to the strong irreversibility, no absorbing set and no global attractor exist in any $L^p$-spaces). In order to derive (some of) them in an intuitive way, we here carry out formal arguments only. Secondly, we shall give a sketch of proof for the existence part of Theorem \ref{T:ex}. In Appendix \ref{S:Aex}, we shall give detailed proofs for the existence part and energy inequalities.

\bigskip
\noindent
{\bf Energy Inequality 1.} Test \eqref{pde} by $u_t$ and employ the relation $(\eta,u_t) = 0$ for any $\eta \in \partial \iI(u_t)$ to see that
\begin{equation}\label{e1}
 \|u_t\|_2^2 + \dfrac{\d}{\d t} E(u(t)) = 0 \quad \mbox{ a.e.~in } \ (0,\infty),
\end{equation}
where $E : H^1_0(\Omega) \cap L^4(\Omega) \to \mathbb R$ is an energy functional given by
$$
E(w) := \dfrac 1 2 \|\nabla w\|_2^2 + \dfrac 1 4 \|w\|_4^4 - \dfrac \lam 2 \|w\|_2^2 \quad \mbox{ for } \ w \in H^1_0(\Omega) \cap L^4(\Omega).
$$
Since $E$ is coercive, one can observe that
\begin{equation}\label{e1-1}
\int^T_0 \|u_t\|_2^2 \, \d t + \sup_{t \in [0,T]} \left( \|\nabla u\|_2^2 + \|u\|_4^4 \right) \leq C \left( E(u_0) + 1 \right)
\end{equation}
for any $T > 0$.  Hence for $u_0 \in H^1_0(\Omega) \cap L^4(\Omega)$, (if a solution exists, then) one can expect that $u \in W^{1,2}(0,T;L^2(\Omega)) \cap L^\infty(0,T;H^1_0(\Omega) \cap L^4(\Omega))$.  Multiplying \eqref{e1} by $t$, we also have
\begin{equation}\label{e1-2}
 t \|u_t\|_2^2 + \dfrac{\d}{\d t} \big( t E(u(t)) \big) = E(u(t)) \quad \mbox{ a.e.~in } \ (0,\infty).
\end{equation}

\bigskip
\noindent
{\bf Energy Inequality 2.} The following is a formal computation. Differentiate both sides of \eqref{pde} in $t$ and set $v = u_t$. Then we have
\begin{equation}\label{pde:v}
 v_t + \eta_t - \Delta v + 3 u^2 v = \lam v \ \mbox{ in } \ \Omega \times (0,\infty),
\end{equation}
where $\eta$ is a section of $\partial \iI(v)$. Test both sides by $v$. It follows that
$$
\dfrac 1 2 \dfrac \d {\d t} \|v\|_2^2 + \dfrac \d {\d t} \iI^*(\eta) + \|\nabla v\|_2^2
+ 3 \int_\Omega u^2 v^2 \, \d x = \lam \|v\|_2^2,
$$
where $\iI^*$ stands for the convex conjugate of $\iI$, i.e.,
$$
\iI^*(\sigma) = \sup_{s \in \mathbb R} \left(s \sigma - \iI(s)\right) = \sup_{s \geq 0} s \sigma = I_{(-\infty,0]}(\sigma).
 $$
 Here we used the fact $(\eta_t , v) = (\d / \d t) \iI^*(\eta)$ by the relation $v \in \partial \iI^*(\eta)$. Moreover, we note that $\iI^*(\eta) = 0$.

Now, for each potential function $V = V(x)$, let us denote by $\lambda_\Omega(V)$ the first eigenvalue of the Schr\"odinger operator $- \Delta + V(x)$ over $\Omega$ equipped with the homogeneous Dirichlet boundary condition. If $u_0 \geq 0$, then one observes that
$$
\|\nabla v\|_2^2 + 3 \int_\Omega u^2 v^2 \, \d x
\geq \|\nabla v\|_2^2 + 3 \int_\Omega u_0^2 v^2 \, \d x
\geq \lambda_\Omega (3 u_0^2)\|v\|_2^2.
$$
Hence
$$
\dfrac 1 2 \dfrac \d {\d t} \|v\|_2^2 %+ \dfrac \d {\d t} \iI^*(\eta)
+ \left( \lambda_\Omega(3u_0^2) - \lam \right) \|v\|_2^2 \leq 0.
$$
In addition, assuming $\lambda_\Omega (3 u_0^2) > \lam$, one can obtain the exponential decay estimate for $v = u_t$,
\begin{equation}\label{e2}
\|v(t)\|_2^2 \leq \|v_0\|_2^2 \exp \left( -2 ( \lambda_\Omega(3u_0^2) - \lam)t \right)
\end{equation}
for all $t > 0$. Here we note that $v_0$ corresponds to $(\Delta u_0 - u_0^3 + \lam u_0)_+$. Exponential decay estimate \eqref{e2} will be used in Corollary \ref{C:exp-conv} (see also Remark \ref{R:exp-conv} in \S \ref{S:conv}).

\bigskip
\noindent
{\bf Energy Inequality 3.}  The following argument will play a key role to overcome difficulties arsing from the doubly nonlinearity of \eqref{pde} and enable us to establish \emph{partial energy-dissipation estimates}.  Formally test \eqref{pde:v} by $\eta$ to find that
$$
 \dfrac \d {\d t} \iI(v) + \dfrac 1 2 \dfrac \d {\d t} \|\eta\|_2^2 + (-\Delta v, \eta) + 3 \int_\Omega u^2 v \eta \, \d x = \lam (v, \eta).
$$
Note that $v \eta \equiv 0$,  $\iI(v) = 0$ a.e.~in $\Omega \times (0,\infty)$ and $(-\Delta v, \eta) \geq 0$. It follows that
$$
%\dfrac \d {\d t} \iI(v) +
\dfrac 1 2 \dfrac \d {\d t} \|\eta\|_2^2 \leq 0,
$$
which implies that
\begin{equation}\label{e3}
\|\eta(t)\|_2^2 \leq \|\eta(s)\|_2^2 \quad \mbox{ for  a.e.  } \ 0 \leq s \leq t < \infty.
\end{equation}
Here we recall that $\eta(0) = \eta_0 :=  - (\Delta u_0 - u_0^3 + \lam u_0)_- $.
Likewise, multiplying \eqref{pde:v} by $|\eta|^{p-2}\eta \in \partial \iI(v)$, one can also derive
%\begin{equation}\label{e3-1}
$$
\|\eta(t)\|_p \leq \|\eta(s)\|_p \quad \mbox{ for  a.e.  } \ 0 \leq s \leq t < \infty,
$$
%\end{equation}
when $\eta(0) \in L^p(\Omega)$, for any $p \in (1,\infty)$, and hence,
$$
\|\eta(t)\|_\infty \leq \|\eta(s)\|_\infty \quad \mbox{ for  a.e.  } \ 0 \leq s \leq t < \infty,
$$
provided that $\eta(0) \in L^\infty(\Omega)$.

\bigskip
\noindent
{\bf Energy Inequality 4.} Testing \eqref{pde} by $u$, we have
$$
\dfrac 1 2 \dfrac \d {\d t} \|u\|_2^2 + \|\nabla u\|_2^2 + \|u\|_4^4
= \lam \|u\|_2^2 - (\eta, u)
\leq \lam \|u\|_2^2 + \|\eta\|_2 \|u\|_2.
$$
By H\"older and Young inequalities and \eqref{e3} with $s = 0$ and $\eta(0) = \eta_0$, we further derive that
\begin{equation}\label{e4}
\dfrac 1 2 \dfrac \d {\d t} \|u\|_2^2 + \|\nabla u\|_2^2 + \dfrac 1 2 \|u\|_4^4
\leq C_1 (1 + \|\eta_0\|_2^{4/3})
\end{equation}
for some constant $C_1 > 0$ depending only on $|\Omega|$ and $\lam$. Integration of both sides over $(0,T)$ yields
\begin{equation}\label{e4-0}
 \dfrac 1 2 \|u(T)\|_2^2 + \int^T_0 \left( \|\nabla u\|_2^2 + \dfrac 1 2 \|u\|_4^4 \right) \, \d t \leq C_1 T (1 + \|\eta_0\|_2^{4/3}) + \dfrac 1 2 \|u_0\|_2^2
\end{equation}
for any $T > 0$.  Thus one expects that $u \in L^2(0,T;H^1_0(\Omega)) \cap L^4(0,T;L^4(\Omega))$ for $u_0 \in \overline{D_r}^{L^2}$ (then $\|\eta_0\|_2 \leq r < \infty$).  Moreover, it follows from \eqref{e1-2} and \eqref{e4-0} that
\begin{equation}\label{e4-2}
 \int^T_0 t \|u_t\|_2^2 \, \d t + T E(u(T)) \leq \dfrac{C_1} 2 T \left( 1 + \|\eta_0\|_2^{4/3} \right) + \dfrac 1 4 \|u_0\|_2^2
\end{equation}
for any $T > 0$.  It also implies that $t^{1/2} u_t \in L^2(0,T;L^2(\Omega))$, $t^{1/2}u \in L^\infty(0,T;H^1_0(\Omega))$ and $t^{1/4}u \in L^\infty(0,T;L^4(\Omega))$ whenever $u_0 \in \overline{D_r}^{L^2}$.  

We next derive a \emph{partial energy-dissipation estimate}. Assume that $\|\eta_0\|_2^2 \leq r$ for some $r > 0$. Then combining \eqref{e4} with \eqref{e1}, one finds that
\begin{equation}\label{e4-1}
 \|u_t\|_2^2 + \dfrac{\d}{\d t} \phi(t) + 2 \lam \phi(t) \leq
  %\lam C_1 (1 + \|\eta_0\|_2^{4/3})
  \lam C_1 (1 + r^{2/3}) =: C_r,
\end{equation}
where $\phi : H^1_0(\Omega) \cap L^4(\Omega) \to \mathbb R$ is a functional given by 
$$
\phi(t) := \dfrac 1 2 \|\nabla u(t)\|_2^2 + \dfrac 1 4  \|u(t)\|_4^4.
$$
Therefore we conclude that
\begin{equation}\label{diss}
\phi(t) \leq \dfrac{C_r}{2\lam} + e^{-2\lam t} \left[ \phi(0) - \dfrac{C_r}{2\lam} \right] \quad \mbox{ for all } \ t \geq 0,
\end{equation}
which will play an important role to construct an absorbing set in \S \ref{S:ps}.  Here it is noteworthy that $C_r$ is independent of $u_0$ belonging to $\overline{D_r}^{H^1_0 \cap L^4}$; however, $C_r$ cannot be chosen uniformly for all $r > 0$. Hence \eqref{diss} can be regarded as a \emph{partial} energy-dissipation estimate.

\bigskip
\noindent
{\bf Energy Inequality 5.} Test \eqref{pde} by $-\Delta u + u^3 - \lam u$ to get
$$
\dfrac \d {\d t} E(u(t)) - \|\eta\|_2^2 + \|-\Delta u + u^3 - \lam u\|_2^2 = 0.
$$
Here we used the fact that $\eta =  - (\Delta u - u^3 + \lam u)_- $. Combining this with \eqref{e3} where $s = 0$ and $\eta(0) = \eta_0$, one has
\begin{equation}\label{e5}
\dfrac \d {\d t} E(u(t)) + \|-\Delta u + u^3 - \lam u\|_2^2 \leq \|\eta_0\|_2^2 \quad \mbox{ a.e.~in } (0,\infty),
\end{equation}
which implies
\begin{equation}\label{e5-1}
E(u(T)) + \int^T_0 \|-\Delta u + u^3 - \lam u\|_2^2 \, \d t \leq T \|\eta_0\|_2^2 + E(u_0)
\end{equation}
for any $T > 0$.  Thus we infer that $u \in L^2(0,T;H^2(\Omega)) \cap L^6(0,T;L^6(\Omega))$, provided that $u_0 \in \overline{D_r}^{H^1_0 \cap L^4}$.  Furthermore, it also follows from \eqref{e5} that
\begin{equation}\label{e5-2}
 T E(u(T)) + \int^T_0 t \|-\Delta u + u^3 - \lam u\|_2^2 \, \d t
  \leq \int^T_0 E(u(t)) \, \d t + \dfrac{T^2}2 \|\eta_0\|_2^2,
\end{equation}
 which along with \eqref{e4-0} implies $t^{1/2}u \in L^2(0,T;H^2(\Omega))$ and $t^{1/6} u \in L^6(0,T;L^6(\Omega))$ if $u_0 \in \overline{D_r}^{L^2}$. 

\bigskip
\noindent
{\bf Energy Inequality 6.} The following argument is also formal; indeed, the differentiability (in $t$) of $-\Delta u + u^3 - \lam u$ is not supposed in Definition \ref{D:sol}. Test \eqref{pde} by $(-\Delta u + u^3 - \lam u)_t$. Then we observe that
$$
\left(u_t, (-\Delta u + u^3 - \lam u)_t \right)
= \|\nabla u_t\|_2^2 + 3 \int_\Omega u^2 u_t^2 \, \d x - \lam \|u_t\|_2^2
$$
and
$$
\left(\eta, (-\Delta u + u^3 - \lam u)_t \right) \geq 0.
$$
Here we used that fact that $(\eta, - \Delta u_t) \geq 0$ and $\eta u_t \equiv 0$ a.e.~in $\Omega \times (0,\infty)$. Therefore
$$
\|\nabla u_t\|_2^2 + 3 \int_\Omega u^2 u_t^2 \, \d x
+ \dfrac 1 2 \dfrac \d {\d t} \|\Delta u - u^3 + \lam u\|_2^2
\leq \lam \|u_t\|_2^2.
$$
Furthermore, by \eqref{e1},
$$
\|\nabla u_t\|_2^2 + 3 \int_\Omega u^2 u_t^2 \, \d x
+ \dfrac 1 2 \dfrac \d {\d t} \|\Delta u - u^3 + \lam u\|_2^2
\leq \lam \|u_t\|_2^2 = - \lam \dfrac \d {\d t} E(u(t)),
$$
which can be rewritten as
\begin{equation}\label{e6}
\|\nabla u_t\|_2^2 + 3 \int_\Omega u^2 u_t^2 \, \d x
+ \dfrac \d {\d t} \left[
\dfrac 1 2 \|\Delta u - u^3 + \lam u\|_2^2 + \lam E(u(t))
\right] \leq 0
\end{equation}
for a.e.~$t > 0$. In particular,
\begin{align}
 \int^T_0 \|\nabla u_t\|_2^2 \, \d t + 3 \int^T_0 \int_\Omega u^2 u_t^2 \, \d x \, \d t + \dfrac 1 2 \|\Delta u(T) - u^3(T) + \lam u(T)\|_2^2 
 \nonumber \\
+ \lam E(u(T)) \leq \dfrac 1 2 \|\Delta u_0 - u_0^3 + \lam u_0\|_2^2 + \lam E(u_0)
 \label{e6-int}
\end{align}
for all $T > 0$.  Hence $u \in L^\infty(0,T;H^2(\Omega) \cap L^6(\Omega))$ and $u_t \in L^2(0,T;H^1(\Omega))$ if $u_0 \in H^2(\Omega) \cap H^1_0(\Omega) \cap L^6(\Omega)$. 
\begin{comment}% Ghidaglia type: but, non-smoothing effect -> not sufficient dissipation estimate ?
		  Moreover, combining \eqref{e6} with \eqref{e5}, we have
$$
\dfrac \d {\d t} \psi(t) + 2 \psi(t) \leq \|\eta_0\|_2^2 + 2(\lam + 1) E(u(t))
\leq \|\eta_0\|_2^2 + 2(\lam + 1) \phi(t)
%\dfrac \d {\d t} \psi(u(t)) +  \|\Delta u - u^3 + \lam u\|_2^2 \leq \|\eta_0\|_2^2
$$
with $\psi(t) := \frac 1 2 \|\Delta u - u^3 + \lam u\|_2^2 + (\lam+1) E(u(t))$. Thus by \eqref{diss},
$$
\psi(t) \leq e^{-2t} \psi(0) + C_1 e^{-2\lam t} \phi(0) + C_2 \quad \mbox{ for all } t > 0
$$
for some $C_1, C_2 > 0$ independent of $\phi(0)$ and $\psi(0)$. Since $E(\cdot)$ is bounded from below, there exists $C_3 > 0$ such that
\begin{equation}\label{diss2}
\dfrac 1 2 \|\Delta u(t) - u^3(t) + \lam u(t)\|_2^2 \leq e^{-2t} \psi(0) + C_1 e^{-2\lam t} \phi(0) + C_3 \quad \mbox{ for all } t > 0.
\end{equation}  \end{comment}

On the other hand, multiply \eqref{e6} by $t$ and compute as follows:
\begin{align}
t\|\nabla u_t\|_2^2 + 3 t \int_\Omega u^2 u_t^2 \, \d x
+ \dfrac \d {\d t} \left( t  \left[
\dfrac 1 2 \|\Delta u - u^3 + \lam u\|_2^2 + \lam E(u(t))
 \right] \right) \nonumber\\
 \leq \dfrac 1 2 \|\Delta u - u^3 + \lam u\|_2^2 + \lam E(u(t)).\label{e6-0}
\end{align}
Integrating both sides over $(0,T)$, we conclude that
 \begin{align*}
  \lefteqn{
 \int^T_0 t\|\nabla u_t\|_2^2 \,\d t
  + 3 \int^T_0 t \left( \int_\Omega u^2 u_t^2 \, \d x \right) \, \d t
  }\nonumber\\
  &\quad + T \left[
\dfrac 1 2 \|\Delta u(T) - u^3(T) + \lam u(T)\|_2^2 + \lam E(u(T))
 \right]\nonumber\\
 &\leq \dfrac 1 2 \int^T_0 \|\Delta u - u^3 + \lam u\|_2^2 \, \d t
 + \lam \int^T_0 E(u(t))\, \d t
 \end{align*}
for all $T > 0$. Combining it with \eqref{e5-1}, one can obtain an estimate exhibiting a smoothing effect,
   \begin{align}
    \lefteqn{
 \int^T_0 t\|\nabla u_t\|_2^2 \,\d t
 + 3 \int^T_0 t \left( \int_\Omega u^2 u_t^2 \, \d x \right) \, \d t
    }\nonumber\\
    &\quad + T \left[
\dfrac 1 2 \|\Delta u(T) - u^3(T) + \lam u(T)\|_2^2 + \lam E(u(T))
 \right]\nonumber\\
 &\leq \dfrac 1 2 \left( T \|\eta_0\|_2^2 + E(u_0) - E(u(T)) \right)
 + \lam \int^T_0 E(u(t))\, \d t \label{e6-1}
   \end{align}
  for all $T > 0$.  Hence one expects that $t^{1/2} u_t \in L^2(0,T;H^1(\Omega))$, $t^{1/2}u \in L^\infty(0,T;H^2(\Omega))$ and $t^{1/6} u \in L^\infty(0,T;L^6(\Omega))$ for $u_0 \in \overline{D_r}^{H^1_0 \cap L^4}$. Moreover, multiply \eqref{e6-0} by $t$ again. Then
\begin{align*}
t^2 \|\nabla u_t\|_2^2 + 3 t^2 \int_\Omega u^2 u_t^2 \, \d x
+ \dfrac \d {\d t} \left( t^2  \left[
\dfrac 1 2 \|\Delta u - u^3 + \lam u\|_2^2 + \lam E(u(t))
 \right] \right)\\
 \leq 2t \left(
 \dfrac 1 2 \|\Delta u - u^3 + \lam u\|_2^2 + \lam E(u(t))
 \right).
\end{align*}
Integrate both sides over $(0,T)$. Then it follows that
 \begin{align}
  \lefteqn{
 \int^T_0 t^2 \|\nabla u_t\|_2^2 \, \d t + 3 \int^T_0 t^2 \left( \int_\Omega u^2 u_t^2 \, \d x \right) \, \d t
}\nonumber\\
  &\quad + T^2  \left[
  \dfrac 1 2 \|\Delta u(T) - u^3(T) + \lam u(T)\|_2^2 + \lam E(u(T)) \right]
  \nonumber\\
  &\leq 2 \int^T_0 t \left(
 \dfrac 1 2 \|\Delta u - u^3 + \lam u\|_2^2 + \lam E(u(t))
 \right) \, \d t\nonumber\\
  &\stackrel{\eqref{e5-2}}\leq \int^T_0 E(u(t)) \, \d t + 2 \lam \int^T_0 t E(u(t)) \, \d t + \dfrac{T^2}2 \|\eta_0\|_2^2
  + \dfrac{\lam T}2 \|u(T)\|_2^2.
  %+ T \lam \left[ C_1 T (1 + \|\eta_0\|_2^{4/3}) + \dfrac 1 2 \|u_0\|_2^2 \right].
  \label{e6-3}
 \end{align}
By virtue of \eqref{e4-0}, we may obtain $t u_t \in L^2(0,T;H^1(\Omega))$, $t u \in L^\infty(0,T;H^2(\Omega))$ and $t^{1/3} u \in L^\infty(0,T;L^6(\Omega))$ for $u_0 \in \overline{D_r}^{L^2}$. 

  Let us also derive another partial energy-dissipation estimate. Inequality \eqref{e6-1} yields
 \begin{align}
  \lefteqn{
\dfrac 1 2 \|\Delta u(T) - u^3(T) + \lam u(T)\|_2^2 + \lam E(u(T))
}\nonumber\\
& \leq \dfrac 1 2 \left( \|\eta_0\|_2^2 + \dfrac 1 T E(u_0) - \dfrac 1 T E(u(T)) \right)
 + \dfrac \lam T \int^T_0 E(u(t))\, \d t
\label{A}
 \end{align}
for any $T > 0$.
  Due to the decrease of the energy $t \mapsto E(u(t))$ and the fact that $E(\cdot) \geq - M_0 := \inf_{w \in H^1_0(\Omega)} E(w) \leq 0$, it follows that
\begin{align}
 \lefteqn{
 \dfrac 1 2 \|\Delta u(t) - u^3(t) + \lam u(t)\|_2^2
 }\nonumber\\
 & \leq \lam M_0 + \dfrac 1 2 \left( \|\eta_0\|_2^2 + \dfrac 1 t E(u_0) + \dfrac 1 t M_0 \right)+ \dfrac \lam t \int^t_0 \phi(u(\tau)) \, \d \tau\nonumber\\
 &\stackrel{\eqref{diss}}\leq
 \lam M_0 + \dfrac 1 2 \left( r + \dfrac 1 t E(u_0) + \dfrac 1 t M_0 \right)
 + \dfrac{C_r}2 + \dfrac{\phi(0)}{2t}
 \quad \mbox{ for all } \ t > 0, \label{diss2}
\end{align}
which will be used to construct an absorbing set in \S \ref{S:ps}. On the other hand, one can also exhibit a dissipation structure in a more quantitative way.  Since $t \mapsto E(u(t))$ is non-increasing and $E(\cdot)$ is bounded from below, we deduce that
$$
E_\infty := \lim_{t \to \infty} E(u(t)) \geq - M_0,
$$
and therefore,
$$
\dfrac 1 T \int^T_0 E(u(t))\, \d t \searrow E_\infty \quad \mbox{ as } \ T \to \infty.
$$
Consequently, by \eqref{A}, one has
 \begin{corollary}\label{C:en-diss}
  In case $u_0 \in H^2(\Omega) \cap H^1_0(\Omega) \cap L^6(\Omega)$, for any $\vep > 0$, one can take $T_\vep > 0$ {\rm (}possibly depending on each solution $u${\rm )} such that, for all $T \geq T_\vep$,
\begin{equation}\label{diss3}
 \|\Delta u(T) - u^3(T) + \lam u(T)\|_2^2 \leq \left\|(\Delta u_0 - u_0^3 + \lam u_0)_-\right\|_2^2  + \vep.
\end{equation}
  In case $u_0 \in \overline{D_r}^{H^1_0\cap L^4}$, for any $\vep > 0$, one can take $T_\vep > 0$ such that, for all $T \geq T_\vep$,
\begin{equation*}%\label{diss3-ep}
 \|\Delta u(T) - u^3(T) + \lam u(T)\|_2^2 \leq r  + \vep.
\end{equation*}
\end{corollary}

\begin{remark}\label{R:diss}
 {\rm
 Due to the non-decreasing constraint on solutions, energy-dissipation cannot be observed in a usual way. Indeed, since $u(\cdot,t) \geq u_0$ a.e.~in $\Omega$ for all $t \geq 0$, it follows that
 $$
\|u(t)\|_p \geq \|u_0\|_p \quad \mbox{ for any } \ p \in [1,\infty],
 $$
 provided that $u_0 \geq 0$. Hence the $L^p$ norm of $u(t)$ never decays and no absorbing set in $L^p(\Omega)$ exists. Moreover, let $z$ be the positive solution of the classical elliptic Allen-Cahn equation \eqref{cAC-e}. Then any multiple $w = cz$ for $c \geq 1$ turns out to be an equilibrium for (P), since it holds that $\Delta w - w^3 + \lam w \leq 0$ a.e.~in $\Omega$. Therefore the set of equilibria is unbounded in any (linear) space including $z$. On the other hand, for any initial data $u_0 \in D_r$, one can observe \emph{partial energy-dissipation} in \eqref{diss}, \eqref{diss2} and \eqref{diss3}.  Indeed, the set $D_r$ excludes a part of the unbounded set of equilibria (still, we emphasize again that $D_r$ itself is unbounded).  Hence, there arises a question: whether or not one can construct an ``attractor'' for the DS generated by (P) over the set $D_r$. An answer to this question will be provided in \S \ref{S:ps} and \S \ref{S:At}.}
\end{remark}

\begin{comment}
\begin{remark}[Maximal regularity of $u$ and $\eta$]
{\rm
Recalling that $\eta =  -\left( \Delta u - u^3 + u \right)_- $, we formally observe that
$$
\partial_j \eta = \sgn \left( - \Delta u + u^3 - u\right)
 \left( \Delta u_j - 3 u^2 u_j + u_j \right),
 \quad u_j := \partial_j u
$$
for any spatial and temporal derivatives $\partial_j$ of the first order. However, if one formally takes the second derivative of $\eta$, then a $\delta$-function must appear as the derivative of the sign function. Hence, it seems difficult to ensure the $H^2$-regularity for $\eta$, even though one takes smooth data $u_0$. Moreover, even in a formal calculation, one should not differentiate $\eta$ (or the both sides of the equation \eqref{pde}) more than once. Furthermore, one may not expect the regularity of $u$ more than $H^3(\Omega)$ (see also Remark \ref{R:obstacle} for a further discussion on the regularity of solution).
}
\end{remark}
 \end{comment}

We close this section by giving a sketch of proof for the existence part of Theorem \ref{T:ex} and by exhibiting an idea to justify the formal arguments given so far.
 \begin{proof}[A sketch of proof for the existence part of Theorem \ref{T:ex}]
Set $H = L^2(\Omega)$ and define a functional $\psi : H \to [0,\infty]$ by
 \begin{equation}\label{varphi}
 \psi(u) := \begin{cases}
		\phi(u) &\mbox{ if } \ u \in H^1_0(\Omega) \cap L^4(\Omega),\\
		\infty &\mbox{ otherwise.}
	       \end{cases}
 \end{equation}
  Then the subdifferential $\partial \psi$ of $\psi$ has the representation, $\partial \psi(v) = - \Delta v + v^3$ for $v \in D(\partial \psi) = H^2(\Omega) \cap H^1_0(\Omega) \cap L^6(\Omega)$. Hence (P) is reduced to an abstract Cauchy problem in the Hilbert space $H$,
\begin{equation}\label{ee}
 u_t + \partial \iI(u_t) + \partial \psi(u) \ni \lam u \ \mbox{ in } H, \quad 0 < t < T, \quad u(0) = u_0,
\end{equation}
  whose solvability (i.e., existence of solutions) has been studied by~\cite{Barbu75} and~\cite{Arai} for $u_0 \in D(\partial \psi)$. Then (iii) can be proved by checking some structure conditions proposed in~\cite{Arai} (see Appendix \S \ref{S:Aex} for more details). For later use, let us briefly recall a strategy (similar to~\cite{Arai}) to construct a solution of \eqref{ee}: We construct approximate solutions for (P) and denote by $u_\lambda$ the (unique) solution of
\begin{equation}\label{P-approx}
 u_t + \partial \iI(u_t) + \partial \psi_\lambda(u) \ni \lam u \ \mbox{ in } H, \quad 0 < t < T, \quad u(0) = u_0 \in D(\partial \psi),
\end{equation}
 where $\partial \psi_\lambda$ is the subdifferential operator of the \emph{Moreau-Yosida regularization} $\psi_\lambda$ of $\psi$ (equivalently, the \emph{Yosida approximation} of $\partial \psi$) (see, e.g.,~\cite{HB1}). Here, one can write
  \begin{equation}\label{Jlam}
  \partial \psi_\lambda(v) = \partial \psi(J_\lambda v) = - \Delta (J_\lambda v) + (J_\lambda v)^3 \quad \mbox{ for all } \ v \in H, 
  \end{equation}
  where $J_\lambda : H \to D(\partial \psi) = H^2(\Omega) \cap H^1_0(\Omega) \cap L^6(\Omega)$ stands for the \emph{resolvent} of $\partial \psi$, i.e., $J_\lambda := (I + \lambda \partial \psi)^{-1}$ (see~\cite{HB1}). Indeed, Equation \eqref{P-approx} can be also rewritten as an evolution equation with a Lipschitz continuous operator in $H$, that is,
 $$
 u_t = \left(I + \partial \iI\right)^{-1} \left( - \partial \psi_\lambda(u) + \lam u \right) \ \mbox{ in } H, \quad 0 < t < T, \quad u(0) = u_0,
 $$
 since $\partial \psi_\lambda$ and $\left(I + \partial \iI\right)^{-1}$ are Lipschitz continuous in $H$. Therefore the solution $u_\lambda$ of \eqref{P-approx} is uniquely determined (by $u_0$) and $u_\lambda$ is of class  $C^{1,1}$  in time. Furthermore, the section $\eta_\lambda$ of $\partial \iI(\partial_t u_\lambda)$ as in \eqref{EQ} belongs to $ C^{0,1} ([0,T];H)$ by means of the relation $\eta_\lambda = \lam u_\lambda - \partial \psi_\lambda(u_\lambda) - \partial_t u_\lambda$. Moreover, \eqref{P-approx} is also equivalent to
 $$
 u_t = \left( - \partial \psi_\lambda(u) + \lam u \right)_+ \ \mbox{ in } H, \quad 0 < t < T.
 $$
  As in the formal computations given above, one can derive corresponding energy inequalities for $u_\lambda$ with $\eta_0$ replaced by $ -( \lam u_0 - \partial \psi_\lambda(u_0))_- $. Therefore passing to the limit as $\lambda \to 0_+$ (see~\cite{Arai}), one can construct an $L^2$-solution $u$ of (P) (for $u_0 \in D(\partial \psi)$) and reproduce all the energy inequalities obtained so far. Moreover, in order to prove smoothing effects (e.g., (ii)), we approximate initial data $u_0 \in \overline{D_r}^{H^1_0 \cap L^4}$ by $u_{0,n} \in D_r$ satisfying
 \begin{equation*}%\label{u0-approx}
  u_{0,n} \to u_0 \quad \mbox{ strongly in } H^1_0(\Omega) \cap L^4(\Omega) \quad \mbox{ as } \ n \to \infty
 \end{equation*}
  and particularly employ the fact $\| (\Delta u_{0,n} - u_{0,n}^3 + \lam u_{0,n})_-\|_2^2 \leq r$ (by $u_{0,n} \in D_r$) to reproduce the energy inequalities. For more precise arguments as well as a proof for (i), we refer the reader to Appendix \S \ref{S:Aex}.
\end{proof}

\section{Reformulation of (P) as an obstacle problem}\label{S:ref}

In this section we shall verify that \eqref{pde} is equivalently rewritten as a parabolic variational inequality of obstacle type. Such a reformulation will not only shed light on a characteristic behavior of solutions but also play an important role to reveal the long-time behavior of each solution (see \S \ref{S:conv}). Moreover, it will be also employed to discuss the uniqueness of solutions and a comparison principle for (P) (see \S \ref{S:cp}) as well as to investigate Lyapunov stability of equilibria in a forthcoming paper (see~\cite{AE2}).

Our result of this section is stated in the following:
 \begin{theorem}[Reformulation of (P) as an obstacle problem]\label{T:reform}
  For $u_0 \in \overline{D_r}^{L^2}$, the Cauchy-Dirichlet problem {\rm (P)} admits a solution $u = u(x,t)$ which also solves
\begin{alignat}{4}
 u_t + \partial \iII (u) - \Delta u + u^3 - \lam u &\ni 0 \quad &\mbox{ in }& \ \Omega \times (0,\infty),\label{pde2}\\
 u&= 0 \quad &\mbox{ on }& \ \partial \Omega \times (0,\infty),\label{bc2}\\
 u&= u_0 \quad &\mbox{ in }& \ \Omega,\label{ic2}
\end{alignat}
 where $\partial \iII$ is the subdifferential operator of the indicator function $\iII$ over $[u_0(x),\infty)$. Hence the section $\eta$ of $\partial \iI(u_t)$ as in \eqref{EQ} also belongs to $\partial \iII(u)$ for a.e.~in $\Omega \times (0,\infty)$.  Such a solution to {\rm (P)} is uniquely determined by the initial datum $u_0$.  Furthermore, {\rm (P)} is equivalently rewritten as {\rm \eqref{pde2}--\eqref{ic2}},  provided that the solution of {\rm (P)} is unique.% for $u_0 \in \overline{D_r}^{L^2}$.  
 \end{theorem}

\begin{remark}
 {\rm
 \begin{enumerate}
  \item In this paper, it is not proved that \emph{all} solutions to (P) solve \eqref{pde2}--\eqref{ic2}, unless solutions to (P) are uniquely determined by initial data. Since Theorem \ref{T:reform} will be proved through the approximation \eqref{P-approx} of \eqref{pde} in Appendix \S \ref{A:reform}, the equivalence of two problems will be ensured only for the solutions constructed by the approximation as in \S \ref{S:e} (and also in Appendix \S \ref{S:Aex}). We also refer the reader to the following formal arguments and Remark \ref{R:inaccuracy}.
  \item The theorem stated above also provides a \emph{selection principle} for (P). Indeed, for $u_0 \in \overline{D_r}^{L^2}$, the uniqueness of solutions is not generally ensured. However, according to Theorem \ref{T:reform}, (P) always possesses one and only one solution which also solves \eqref{pde2}--\eqref{ic2}. Moreover, as will be discussed in Appendix \S \ref{S:Aex}, selected solutions fulfill energy inequalities derived in \S \ref{S:e}. Such a selection principle will be used to consider the dynamical system generated by (P) and to prove the convergence of solutions as $t \to +\infty$.
  \item It is noteworthy that the \emph{fully nonlinear} problem \eqref{iAC} is now converted to a \emph{semilinear} obstacle problem \eqref{pde2}. However, such a semilinear problem would have another difficulty, since the obstacle function $u_0$ is supposed to lie on the $L^2$ closure of $D_r$ and the problem is posed on the $L^2$ (i.e., strong) framework. On the other hand, it is also known  (see~\cite{Ef}) that uniformly elliptic fully nonlinear equations of the form $f(D^2u) = 0$ can be reduced to a quasilinear one, provided that $f$ is smooth enough (e.g., of class $C^{3,\alpha}$). However, it is not applicable to \eqref{iAC}, for the corresponding $f$ is not so smooth and not uniformly elliptic. 
 \end{enumerate}
 }
\end{remark}

 \begin{remark}[Parabolic obstacle problem]\label{R:obstacle}
  {\rm
  Problem \eqref{pde2}--\eqref{ic2} can be equivalently rewritten as follows:
  \begin{align*}
   u \geq u_0, \quad u_t - \Delta u + u^3 - \lam u \geq 0 \quad \mbox{ in } \ \Omega \times (0,\infty),\\
   \left( u - u_0 \right) \left(u_t - \Delta u + u^3 - \lam u \right) = 0 \quad \mbox{ in } \ \Omega \times (0,\infty),\\
   u|_{\partial \Omega} = 0, \quad u|_{t = 0} = u_0,
  \end{align*}
  which is an obstacle problem of parabolic type and where the initial datum $u_0$ also plays a role of the obstacle function from below (see~\cite{LauSal09,CafFig13}). One may no longer expect classical regularity of solutions to (P). Indeed, let us consider a simpler elliptic obstacle problem, e.g.,
 $$
 - \Delta \phi (x) \geq f(x), \quad \phi(x) \geq g(x), \quad \left( -\Delta \phi(x) - f(x) \right) \left( \phi(x) - g(x) \right) = 0 \ \mbox{ in } \Omega
 $$
 along with the homogeneous Dirichlet condition. It is well known that the optimal regularity of solution is $C^{1,1}(\Omega)$ (unless the contact set is non-empty), even though the obstacle function $g$ is sufficiently smooth (e.g., $g \in C^\infty(\overline\Omega)$) (see~\cite{Cafferelli}).
  }
 \end{remark}

 In the rest of this section, we shall give only a formal argument to explain an idea of proof in an intuitive way. A rigorous proof for the theorem above will be provided in \S \ref{A:reform} of Appendix (see also Remark \ref{R:inaccuracy}). %Furthermore, we also refer the reader to Corollary \ref{A:C:ref} as a remark.
 Let us start with the following observation:
\begin{lemma}\label{L:eta-dec}
Let $u$ be a solution of {\rm (P)} {\rm (}which is constructed as in the proof of Theorem \ref{T:ex}{\rm )} and let $\eta$ be the section of $\partial \iI(u_t)$ as in \eqref{EQ}. Then $\eta(x,t)$ is non-decreasing in $t$ for a.e.~$x \in \Omega$. 
\end{lemma}

\begin{proof}[Formal proof of Lemma \ref{L:eta-dec}]
 Let $\zeta \in C^\infty_0(\Omega)$ be such that $\zeta \geq 0$ in $\Omega$ and \emph{formally} test \eqref{pde:v} by $\zeta \eta$. Then we find that
  \begin{align*}
  \dfrac{\d}{\d t} \int_\Omega \zeta \iI(u_t) \, \d x
  + \dfrac 1 2 \dfrac{\d}{\d t} \int_\Omega \zeta \eta^2 \, \d x
  + \left( - \Delta u_t ,  \zeta \eta \right)\\
  + 3 \int_\Omega u^2 u_t \eta \zeta \, \d x
  = \lam \int_\Omega u_t \eta \zeta \, \d x.
  \end{align*}
  Here we notice that $\zeta \eta$ also belongs to $\partial \iI(u_t)$ by $\zeta \geq 0$ and $\eta \in \partial \iI(u_t)$, and therefore,
  $$
 %  \nabla u_t \cdot \nabla (\zeta \eta)  = (\nabla u_t \cdot \nabla \eta) \zeta + \left( \nabla u_t \cdot \nabla \zeta \right) \eta \geq \nabla \iI(u_t) \cdot \nabla \zeta = 0.
 \left( - \Delta u_t ,  \zeta \eta \right) \geq 0.
 $$
% (see \eqref{Delxeta} for a rigorous argument).
  Then integrate both sides over $(s,t)$ and employ the facts that $u_t \eta = 0$ and $\iI(u_t) = 0$ a.e.~in $\Omega \times (0,\infty)$ to obtain
%  $$
 %  \int_\Omega \zeta \iI(u_t) \, \d x \Big|^t_0  +
% \dfrac 1 2 \int_\Omega \zeta \eta^2 \, \d x \Big|^t_s \leq 0.
%  $$
%  Therefore
  $$
  \int_\Omega \zeta \eta^2(t) \, \d x \leq \int_\Omega \zeta \eta^2(s) \, \d x
  \quad \mbox{ for } \ t \geq s \geq 0.
  $$
 Likewise, noting $|\eta|^{r-2}\eta \in \partial \iI(u_t)$ for any $r \in (1,\infty)$, one can also obtain
  \begin{equation}\label{eta}
   \int_\Omega \zeta |\eta(t)|^r \, \d x \leq \int_\Omega \zeta |\eta(s)|^r \, \d x
  \quad \mbox{ for } \ t \geq s \geq 0 \ \mbox{ and } \ 1 < r < \infty
  \end{equation}
 for any $\zeta \in C^\infty_0(\Omega)$ satisfying $\zeta \geq 0$.
 In particular, let $r > 1$ be less than 2 and let $q \in (1,\infty)$ and $\zeta \in L^q(\Omega)$ be such that $\zeta \geq 0$ and $1/q + r/2 = 1$. Then one can take $\zeta_n \in C^\infty_0(\Omega)$ such that $\zeta_n \geq 0$ and $\zeta_n \to \zeta$ strongly in $L^{q}(\Omega)$. Moreover, passing to the limit in \eqref{eta} with $1 < r < 2$ and $\zeta$ replaced by $\zeta_n$ as $n \to \infty$, we deduce that
 \begin{equation}\label{eta2}
 \int_\Omega \zeta |\eta(t)|^r \, \d x \leq \int_\Omega \zeta |\eta(s)|^r \, \d x
  \quad \mbox{ for } \ t \geq s \geq 0 \ \mbox{ and } \ 1 < r < 2
 \end{equation}
 for any $\zeta \in L^q(\Omega)$, $\zeta \geq 0$, $1/q + r/2 = 1$. Then by \eqref{eta2}, we assure that $|\eta(x,t)|$ is non-increasing in $t$ for a.e.~$x \in \Omega$. Indeed, suppose on the contrary that $|\eta(x,t)| > |\eta(x,s)|$ for all $x \in \Omega_0$ and for some $t > s$ and $\Omega_0 \subset \Omega$ satisfying $|\Omega_0| > 0$. Substitute $\zeta = \chi_{\Omega_0}$, which is the characteristics function over $\Omega_0$, into \eqref{eta2} to obtain
 $$
 \int_{\Omega_0} |\eta(t)|^r \, \d x \leq  \int_{\Omega_0} |\eta(s)|^r \, \d x. 
 $$
 However, this fact contradicts the assumption.
\end{proof}

\begin{proof}[Formal proof of Theorem \ref{T:reform}]
 Let $Q_0 \subset \Omega \times (0,T)$ be such that \eqref{pde} holds and the section $\eta \in \partial \iI(u_t)$ is non-decreasing (in time) in $Q_0$ (more precisely, $\eta(x,t) \geq \eta(x,s)$ for any $(x,t), (x,s) \in Q_0$ satisfying $t \geq s$) and \emph{suppose that} $Q_0$ is well-defined and the compliment of $Q_0$ in $\Omega \times (0,T)$ has Lebesgue measure zero. Let $(x_0,t_0) \in Q_0$ be fixed. In case $u(x_0,t_0) = u_0(x_0)$, it is obvious that $\eta(x_0,t_0)$ belongs to $\partial \iII(u(x_0,t_0)) = (-\infty, 0]$. Hence \eqref{pde2} is satisfied at $(x_0,t_0)$. In case $u(x_0,t_0) > u_0(x_0)$, \emph{suppose that} there is a set $I \subset (0,t_0)$ with $|I| > 0$ such that $u_t(x_0,t) > 0$ for all $t \in I$. Therefore $\eta(x_0,t) = 0$ for $t \in I$ satisfying $(x_0,t) \in Q_0$. From the non-decrease of $\eta$, we derive that $\eta(x_0,t_0) = 0$. Thus \eqref{pde2} holds true at $(x_0,t_0)$. Consequently, the solution $u$ of (P) solves \eqref{pde2}--\eqref{ic2}, and moreover, $\eta$ belongs to $\partial \iII(u)$.

 One can prove in a standard way that the solution of \eqref{pde2}--\eqref{ic2} is uniquely determined by the initial datum $u_0$ (cf.~Theorem \ref{T:comp}). Hence, it turns out that (P) and \eqref{pde2}--\eqref{ic2} are equivalent each other, provided that the solution to (P) is unique.
\end{proof}

\begin{remark}[Ambiguity of the formal arguments above]\label{R:inaccuracy}
 {\rm
 In the present paper, we are working on Lebesgue space settings, instead of continuous function spaces. In particular, the section $\eta$ of $\partial \iI(u_t)$ is not supposed to be continuous in space. It seems natural in view of obstacle problems, since $u(x,t)$ may loose classical regularity and $\eta = - (\Delta u - u^3 + \lam u)_-$ (see Remark \ref{R:obstacle}).

Let us also point out an obscure point of the formal arguments above. The conclusion of Lemma \ref{L:eta-dec} is the following: for each $t > s \geq 0$, one can take a subset $\Omega_{t,s}$ of $\Omega$ such that $\eta(x,t) \geq \eta(x,s)$ for all $x \in \Omega_{t,s}$ and $|\Omega \setminus \Omega_{t,s}| = 0$. Indeed, as $\eta(x,t)$ may not be continuous in $(x,t)$ and it is defined only for a.e.~$(x,t) \in \Omega \times (0,T)$, we have not proved a pointwise monotonicity of $\eta(x,t)$ at each $t$. This fact prevents us to apply the argument above to prove Theorem \ref{T:reform}. Indeed, it is not obvious whether the set $Q_0$ is well-defined. In Appendix \S \ref{A:reform}, we shall approximate $\eta$ by a differentiable (in $t$) function $\eta_\lambda \in C^{0,1} ([0,T];L^2(\Omega))$ and prove the pointwise decrease of $\eta_\lambda$ in $t$. Even so, we still face a delicate issue whether a set corresponding to $Q_0$ has full Lebesgue measure or not, for $\eta_\lambda(\cdot,t)$ is defined only for a.e.~$x \in \Omega$ at each $t \geq 0$. Moreover, when we fix a point $x_0 \in \Omega$, the existence of the set $I$ where $u_t(x_0,\cdot) > 0$ and whose measure is positive is also obscure. However, these difficulties will be overcome by carefully carrying out a measure theoretic argument in Appendix \ref{A:reform}.
 }
\end{remark}

\begin{comment}
\begin{remark}[Equivalence between two problems]
 {\rm
 Uniqueness of solutions for \eqref{pde2}--\eqref{ic2} can be proved in a standard manner (see, e.g.,~\cite{HB1}). On the other hand, for any $u_0$ belonging to $\overline{D_r}$, (P) admits a unique solution. %Therefore if $u$ solves \eqref{pde2}--\eqref{ic2} for some $u_0 \in \overline D$, then it must coincide with the unique solution for (P) with the same data.
 Hence both problems are equivalent to each other, provided that $u_0 \in \overline{D_r}$.
 }
\end{remark}
\end{comment}

 \section{Comparison principle}\label{S:cp}

 This section is devoted to proving a comparison principle for the obstacle problem \eqref{pde2} as well as the strongly irreversible Allen-Cahn equation \eqref{iAC} (or equivalently, \eqref{pde}). Let us begin with the definition of $L^2$ sub- and supersolutions of \eqref{pde2} (and \eqref{pde}).

  \begin{definition}
   Let $T \in (0,\infty)$ be fixed. A function $u \in C([0,T];L^2(\Omega))$ is said to be an \emph{$L^2$ subsolution} {\rm (}or sub $L^2$ solution{\rm )} of \eqref{pde2} on $Q_T = \Omega \times (0,T)$, if the following conditions are all satisfied\/{\rm :}
   \begin{enumerate}
    \item[\rm (i)] $u$ belongs to the same class as in {\rm (i)} of Definition \ref{D:sol},
    \item[\rm (ii)] there exists $\eta \in L^\infty(0,T;L^2(\Omega))$ such that
		 \begin{equation}\label{subsol}
		 u_t + \eta - \Delta u + u^3 - \lam u \leq 0, \quad \eta \in \partial \iII(u) \ \mbox{ for a.e. } (x,t) \in \Omega \times (0,T).
		 \end{equation}
   \end{enumerate}
   A function $u \in C([0,T];L^2(\Omega))$ is said to be an \emph{$L^2$ supersolution} {\rm (}super $L^2$ solution{\rm )} of \eqref{pde2}  on $Q_T = \Omega \times (0,T)$, if {\rm (i)} and {\rm (ii)} above are satisfied with the inverse inequality of \eqref{subsol}. Furthermore, a sub- and a super $L^2$ solution of \eqref{pde} are also defined as above by replacing the inclusion of \eqref{subsol} with $\eta \in \partial \iI(u_t)$.
  \end{definition}

 Our result reads,
 
\begin{theorem}[Comparison principle for \eqref{pde2}]\label{T:comp}
 Let $u$ and $v$ be a sub- and a super $L^2$ solution for \eqref{pde2} with the obstacle function replaced by $u_0 = u(0)$ and $v_0 = v(0)$, respectively, in $Q_T = \Omega \times (0,T)$ for some $T > 0$. Suppose that $u \leq v$ a.e.~on the parabolic boundary $\partial_p Q_T = (\Omega \times \{0\}) \cup (\partial \Omega \times [0,T))$. Then it holds that
 $$
 u \leq v \ \mbox{ a.e.~in } Q_T.
 $$
 In particular, the solution of \eqref{pde2}--\eqref{ic2} is unique.
\end{theorem}

\begin{proof}
By subtracting inequalities (see \eqref{subsol}) and by setting $w := u - v$, we see that
$$
w_t - \Delta w + u^3 - v^3 \leq \lam w + \nu - \mu \ \mbox{ in } Q_T,
$$
where $\mu$ and $\nu$ are sections of $\partial \iII(u)$ and $\partial \iIIv(v)$, respectively. Test both sides by $w_+$. Then we have:
$$
\dfrac 1 2 \dfrac{\d}{\d t} \|w_+\|_2^2 \leq \lam \|w_+\|_2^2 + \int_\Omega (\nu - \mu) w_+ \, \d x \quad \mbox{ for a.e. } 0 < t < T.
$$
Here we observe by $\nu \leq 0$ and $\mu \in \partial \iII(u)$ that
\begin{align*}
 \int_\Omega (\nu - \mu)w_+ \, \d x
 &= \int_\Omega \nu w_+ \, \d x - \int_\Omega \mu w_+ \, \d x\\
 &\leq % \int_{\{v = v_0\} \cap \{u \geq v\}} \nu w \, \d x
 - \int_{\{u = u_0\} \cap \{u \geq v\}} \mu w \, \d x.
 \end{align*}
Due to the fact that $v \geq v_0 \geq u_0$ a.e.~in $\Omega$, one of the following (i) and (ii) holds: (i) the set $\{u = u_0\} \cap \{u \geq v\}$ has Lebesgue measure zero; (ii) $w = 0$ for a.e.~$x \in \{u = u_0\} \cap \{u \geq v\}$. Hence it follows that
 $$
  \int_{\{u = u_0\} \cap \{u \geq v\}} \mu w \, \d x
 = 0.
  $$
%  Moreover, by the definition of subdifferential,
%  \begin{align*}
%  \int_{\{v = v_0\} \cap \{u \geq v\}} \nu w \, \d x
%  &= \int_{\{v = v_0\} \cap \{u \geq v\}} \nu (u - v_0) \, \d x\\
%  &\leq \int_{\{v = v_0\} \cap \{u \geq v\}} \left( \iIIv(u) - \iIIv(v_0) \right) \, \d x = 0
%  \end{align*}
%  (here one also notes that $u \geq v = v_0$ over the region of integral above).
Combining all these facts, we deduce that
$$
 \int_\Omega (\nu - \mu)w_+ \, \d x \leq 0.
 $$
 Therefore one obtains
 $$
 \dfrac 1 2 \dfrac{\d}{\d t} \|w_+\|_2^2 \leq \lam \|w_+\|_2^2 \quad \mbox{ for a.e. } 0 < t < T,
 $$
 and hence, applying Gronwall's inequality, we conclude that $w_+ \equiv 0$ a.e.~in $Q_T$, which completes the proof.
\end{proof}

Now, we exhibit a range-preserving property of solutions to (P) in the following:

\begin{corollary}\label{C:bdd}
Let $u$ be the \emph{unique} solution of {\rm (P)} such that $u$ also solves \eqref{pde2}--\eqref{ic} {\rm (}see Theorem \ref{T:reform}{\rm )}. Assume $u_0 \in L^\infty(\Omega)$. Then it holds that
 $$
 u_0(x) \leq u(x,t) \leq \max \left\{  \sqrt{\lam}  , \|u_0\|_{L^\infty(\Omega)} \right\} \ \mbox{ a.e.~in } \ \Omega \times (0,\infty),
 $$
and hence, $u \in L^\infty(\Omega \times (0,\infty))$.% In particular, if $0 \leq u_0(x) \leq \sqrt{\lam}$ for a.e.~$x \in \Omega$, then $0 \leq u(x,t) \leq 1$ for a.e.~$(x,t) \in \Omega \times (0,\infty)$.
\end{corollary}

 \begin{proof}
  Due to the non-decrease of $u(x,t)$, it follows immediately that
 $$
  u_0(x) \leq u(x,t) \ \mbox{ a.e.~in } \Omega \times (0,\infty).
  $$
 On the other hand, by assumption, $u$ is also a solution of \eqref{pde2}. Moreover, the constant function $U(x,t) \equiv \max \left\{ \sqrt{\lam} , \|u_0\|_{L^\infty(\Omega)} \right\} \geq \sqrt{\lam}$ turns out to be a supersolution of \eqref{pde2}, and furthermore, one can observe that
 $$
 u(x,t) \leq U(x,t) \ \mbox{ a.e.~on } \partial_p Q_T \ \mbox{ for any } T > 0.
 $$
 Thus by Theorem \ref{T:comp}, we deduce that $u(x,t) \leq \max \left\{ \sqrt{\lam} , \|u_0\|_{L^\infty(\Omega)} \right\}$ a.e.~in $Q_T$ for any $T > 0$.
 \end{proof}

As for \eqref{iAC} (or equivalently \eqref{pde}), we shall exhibit two comparison principles under different additional assumptions. The following theorem provides a comparison principle for classical solutions of \eqref{iAC}:

\begin{theorem}[Comparison principle for classical solutions to \eqref{iAC}]\label{T:comp0}
 Let $u$ and $v$ be a sub- and a super $C^{2,1}$-solution for \eqref{iAC} in $Q_T = \Omega \times (0,T)$ for some $T > 0$, respectively. Suppose that $u \leq v$ a.e.~on the parabolic boundary $\partial_p Q_T = \Omega \times \{0\} \cup \partial \Omega \times [0,T)$. Then it holds that
 $$
 u \leq v \ \mbox{ a.e.~in } Q_T.
 $$
\end{theorem}

\begin{proof}
Let $u$ and $v$ be a sub- and a supersolution for (P), respectively. Then, it holds that
\begin{align*}
 \partial_t (u-v) &\leq \left( \Delta u - u^3 + \lam u \right)_+ - \left( \Delta v - v^3 + \lam v \right)_+\\
 &\leq \left( \Delta w - u^3 + v^3 + \lam w \right)_+,
\end{align*}
where we set $w := u - v$. Let $\alpha > 0$ be fixed so that
$r \mapsto \lam r_+ - \alpha r$ is strictly decreasing (e.g., $\alpha > \lam$). Subtracting $\alpha w$ from both sides, one has
$$
 w_t - \alpha w \leq \left( \Delta w - u^3 + v^3 + \lam w \right)_+ - \alpha w.
$$
Multiply both sides by $e^{-\alpha t}$ and set $z := e^{-\alpha t}w$. It then follows that
$$
z_t \leq \left( \Delta z - e^{-\alpha t}(u^3 - v^3) + \lam z \right)_+ - \alpha z.
$$
We claim that
$$
z \leq 0 \ \mbox{ in } Q := \Omega \times (0,T],
$$
which also implies
$$
u \leq v \ \mbox{ in } Q.
$$
Indeed, assume on the contrary that
 $$
 z(x_0,t_0) > 0
 $$
 at some $(x_0,t_0) \in Q$. Then
 $$
 \sup_{(x,t) \in Q} z(x,t) > 0,
 $$
 where the supremum is achieved by some $(x_1, t_1) \in \Omega \times (0,T]$. Then by Taylor's expansion,
 $$
 z_t \geq 0, \quad \nabla z = 0, \quad \Delta z \leq 0
 \quad \mbox{ at } \ (x_1,t_1).
 $$
 Hence
 $$
 0 \leq z_t \leq \left( \Delta z - e^{-\alpha t}(u^3 - v^3) + \lam z \right)_+ - \alpha z \leq \lam z_+ - \alpha z < 0 \quad \mbox{ at } \ (x_1,t_1).
 $$
 This yields a contradiction. Thus $z \geq 0$ on $Q$.
\end{proof}

One can also prove a comparison principle for \emph{strictly increasing} $L^2$-subsolutions for \eqref{pde}.
\begin{proposition}
 Let $u$ be an $L^2$-subsolution of \eqref{pde} in $Q_T$ satisfying $u_t > 0$ and let $v$ be an $L^2$-supersolution of \eqref{pde}  in $Q_T$. Suppose that $u \leq v$ a.e.~on $\partial_p Q_T$. Then it holds that
 $$
 u \leq v \ \mbox{ a.e.~in } Q_T.
 $$
\end{proposition}

\begin{proof}
By assumption, we find that $\partial \iI(u_t) = \{0\}$, and therefore, \eqref{pde} holds with $\eta \equiv 0$. Subtract inequalities to see that
 $$
 w_t - \nu - \Delta w + u^3 - v^3 \leq \lam w \ \mbox{ in } Q_T,
 $$
 where $w := u - v$ and $\nu$ is a section of $\partial \iI(v_t)$. The multiplication of the both sides and $w_+$ yields
 $$
 \dfrac \d {\d t} \|w_+\|_2^2 - \int_\Omega \nu w_+ \, \d x
 + \|\nabla w_+\|_2^2 \leq \lam \|w_+\|_2^2.
 $$
 Here recall that $\nu \leq 0$, and therefore, by Gronwall's inequality,
 we conclude that $w_+ \equiv 0$, that is, $u \leq v$ in $Q_T$.
\end{proof}

 \begin{comment}
 If $u_0 \geq \delta \in (0,1)$, then $u(x,t) \geq 0$ a.e.~in $\Omega$ as $t \to \infty$. Indeed, let $\rho(t)$ be the solution for the ODE,
 $$
 \dot \rho + \rho^3 = \rho, \quad \rho(0) = \delta > 0.
 $$
 Then $\rho(t)$ is increasing and converges to $1$ as $t \to \infty$. Moreover, 
 \end{comment}

\section{Phase set, semigroup and compact absorbing set}\label{S:ps}

The following two sections are devoted to constructing a global attractor of the Dynamical System (DS for short) generated by (P) as well as \eqref{pde2}--\eqref{ic2}. We emphasize again that due to the strong irreversibility, global attractor does not exist in any $L^p$-spaces. So we need a peculiar setting to extract energy-dissipation structures of the equation and to construct a global attractor. %Throughout the following two sections, we always assume $N \leq 3$ (hence, the continuous dependence of solutions to (P) on initial data is guaranteed by Theorem \ref{T:ex}.  Moreover, $H^1_0(\Omega)$ is continuously embedded in $L^6(\Omega)$. However, $N \leq 3$ is used only for those and the latter one is not essentially needed. Therefore the following results can be extended to general $N$, if the uniqueness of solutions to (P) is ensured). 
We start with setting up a (nonlinear) phase set and a metric on it. 

Let $r > 0$ be arbitrarily fixed and set a phase set $D = D_r$ (see \S \ref{S:wp} for the definition of $D_r$). Thanks to Theorem \ref{T:ex}, we see that $D$ is invariant under the evolution of solutions to (P). Furthermore, let us define a metric $\d (\cdot,\cdot)$ over the set $D$ by
 \begin{equation*}
 \d(u,v) := \|u - v\|_{H^1_0(\Omega)} + \|u - v\|_{L^4(\Omega)}
 \quad \mbox{ for } \ u , v \in D.
\end{equation*}
 Moreover, denote by $S_t : D \to D$ the semigroup associated with (P) and \eqref{pde2}--\eqref{ic2}, that is,
 $$
 S_t u_0 := u(t) \quad \mbox{ for } \ t \geq 0, \ u_0 \in D,
 $$
 where $u$ is the (unique) solution of (P) which also solves \eqref{pde2}--\eqref{ic2} and whose initial datum is $u_0$. By Theorems \ref{T:ex} and \ref{T:reform}, one can assure that $S_t$ is a continuous semigroup.

 We next set a subset of $D$ by
 $$
 B_0 := \Big\{
 u \in D \colon \|\Delta u - u^3 + \lam u\|_2^2 \leq  \rr  + 1, \
 \phi(u) \leq C_r + 1
 \Big\}
 $$
 with  $\rr = 2 \kappa M_0 + r + C_r$ and $C_r := C_r/(2\kappa)$ (see \eqref{diss} and \eqref{diss2}). Then the partial energy-dissipation estimates \eqref{diss} and \eqref{diss2} immediately ensure the following:

 \begin{lemma}\label{L:absorbing}
  The set $B_0$ is \emph{$D$-absorbing}, that is, for any bounded subsets $B$ of $(D,\d)$, one can take $\tau_B \geq 0$ such that $S_t B \subset B_0$ for all $t \geq \tau_B$.
 \end{lemma}

 We next prove the compactness of $B_0$ in $(D,\d)$.

 \begin{lemma}\label{L:B0}
  The set $B_0$ is compact in $(D,\d)$.  
 \end{lemma}

 \begin{proof}
  To prove this lemma, let us define a functional $C : L^2(\Omega) \to [0,\infty)$ by
 $$
 C(f) := \int_\Omega \big( f(x) \big)_-^2 \, \d x
 \quad \mbox{ for } \ f \in L^2(\Omega).
 $$
 We then observe that $C(\cdot)$ is (strongly) continuous in $L^2(\Omega)$ and convex. Hence $C(\cdot)$ is also weakly lower semicontinuous in $L^2(\Omega)$ by the convexity.

  Let $(u_n)$ be a sequence in $B_0$. Then obviously, $(u_n)$ is bounded in $\HeinLvier$, and moreover,
  $$
  \|-\Delta u_n + u_n^3\|_2 \leq \sqrt{ \rr  +1} + \lam \|u_n\|_2 \leq C.
  $$
Noting that
  $$
  \|- \Delta v\|_2^2 + \|v^3\|_2^2
  \leq \|-\Delta v + v^3\|_2^2 \quad \mbox{ for all } \ v \in \HzweiLsechs,
  $$
  we deduce that $(u_n)$ is bounded in $\HzweiLsechs$. Then
  $$
  u_n \to u \quad \mbox{ weakly in } \HzweiLsechs
  $$
  for some $u \in \HzweiLsechs$.
  Moreover, by the compact embedding $\HzweiLsechs \hookrightarrow \HeinLvier$, one can take a subsequence of $(n)$ without relabeling such that 
 $$
  u_n \to u \quad \mbox{ in } \ (D,\d),
  \ \mbox{ i.e., strongly in } \HeinLvier,
 $$
 which also yields
 \begin{alignat*}{4} 
 \Delta u_n &\to \Delta u \quad &&\mbox{ strongly in } H^{-1}(\Omega) \mbox{ and weakly in } L^2(\Omega),\\
  u_n^3 &\to u^3 \quad &&\mbox{ strongly in } L^q(\Omega)
  \mbox{ and weakly in } L^2(\Omega),\ 1 \leq q <2.
 \end{alignat*}
Hence we see that
 $$
 \Delta u_n - u_n^3 + \lam u_n \to \Delta u - u^3 + \lam u \quad \mbox{ weakly in } L^2(\Omega).
 $$
 Since $C(\cdot)$ is weakly lower semicontinuous in $L^2(\Omega)$ and $u_n \in D$, it follows that
 $$
 C(\Delta u - u^3 + \lam u) \leq \liminf_{n \to \infty} C(\Delta u_n - u_n^3 + \lam u_n) \leq r,
 $$
 which implies $u \in D$. Moreover, the weak lower semicontinuity of $\|\cdot\|_2$ and $\phi$ leads us to obtain $u \in B_0$.
 \end{proof}

\begin{remark}[Set up of the phase set]
 {\rm
 \begin{enumerate}
  \item %As is pointed out, due to the strong irreversibility, one cannot construct any absorbing set in any $L^p$-space.
	The phase space $D$ assigned here is nonlinear and non-convex (cf.~see also Proposition \ref{P:Dconv} below). Furthermore, we stress that $D$ is unbounded. The metric $\d$ is chosen such that $B_0$ becomes compact in $(D,\d)$.%; indeed, in the proof of Lemma \ref{L:B0}, the compact embedding $H^2(\Omega) \cap L^6(\Omega) \hookrightarrow H^1_0(\Omega) \cap L^4(\Omega)$ plays an important role.
  \item One may replace the phase set $D$ by its closure in $\HeinLvier$. Then the compactness of $B_0$ follows in a simpler way; indeed, it suffices to prove the precompactness of $B_0$ in $(D,\d)$. However, we address ourselves to the phase set $D$ instead of its closure.
 \end{enumerate}
 }
\end{remark}
 
\section{Construction of a $(D,\d)$-global attractor}\label{S:At}

In this section, we shall construct a \emph{global attractor} defined in the following sense for the DS  generated by (P) on the phase set $(D,\d)$:

\begin{definition}[$(D,\d)$-global attractor]\label{D:at}
 A subset $\mathcal U$ of $D$ is called a \emph{$(D,\d)$-global attractor} associated with the DS $(S_t, (D,\d))$ if the following conditions hold true\/{\rm :}
 \begin{enumerate}
  \item $\mathcal U$ is compact in $(D,\d)$.
  \item $\mathcal U$ satisfies an \emph{attraction property} in $(D,\d)$, that is, let $B \subset D$ be a $\d$-bounded subset of $D$ {\rm (}i.e., the diameter $\mathrm{diam}(B) := \sup\{\d(u,v) \colon u,v \in B\}$ is finite{\rm )}. %it holds that $S_t B \to \mathcal U$ in $(D,\d)$ as $t \to \infty$, that is,
	Then for any neighborhood $\mathcal O$ of $\mathcal U$ in $(D,\d)$, there exists $\tau_{\mathcal O} \geq 0$ such that $S_t B \subset \mathcal O$ for all $t \geq \tau_{\mathcal O}$.
  \item $\mathcal U$ is strictly invariant, i.e., for any $t \geq 0$, it holds that $S_t \, \mathcal U = \mathcal U$.
 \end{enumerate}
\end{definition}

Our result reads,

 \begin{theorem}[Existence of $(D,\d)$-global attractor]\label{T:GA}
The DS $(S_t, (D,\d))$ admits the $(D,\d)$-global attractor $\mathcal U$, which is given by
  \begin{equation}\label{U}
\mathcal U := \bigcap_{\tau \geq \tau_0} \overline{F_\tau}, \quad
  F_\tau := \bigcup_{t \geq \tau} S_t B_0,
  \end{equation}
where $\tau_0$ is a positive constant and $\overline{F_\tau}$ stands for the closure of $F_\tau$ in $(D,\d)$. Moreover, $\mathcal U$ is the maximal bounded strictly invariant set, and therefore, the $(D,\d)$-global attractor is unique.
 \end{theorem}

  \begin{proof}
   The following proof is essentially based on a standard theory (see~\cite[Chap 2, \S 2]{BV} and also~\cite{EfGak}); however,  for a convenience of the reader,  we give a proof specific to our setting instead of applying a ready-made theorem, since there seem to be slight differences of basic settings (e.g., in~\cite[Chap 2, \S 2]{BV}, an absorbing set is supposed to attract all bounded sets in a \emph{Banach space}; on the other hand, in our setting, the absorbing set $B_0$ attracts any bounded sets in the \emph{metric space} $(D,\d)$). In what follows, we shall check three conditions (i)--(iii) of Definition \ref{D:at} for the set $\mathcal U$ given by \eqref{U}.
 
\noindent
\emph{Compactness in $(D,\d)$.} Here we note that $F_\tau \subset B_0$ for any $\tau \geq t_0$ and some $t_0 \geq 0$. Hence $\overline{F_\tau}$ is compact in $(D,\d)$ and included in $B_0$. Therefore $\mathcal U$ is included in $B_0$ and compact in $(D,\d)$.

\noindent
\emph{Attraction property in $(D,\d)$.} To prove this, suppose on the contrary that there exist a neighborhood $\mathcal O_0$ of $\mathcal U$ in $(D,\d)$ and a sequence $t_n \to \infty$ such that $S_{t_n} B \cap (D \setminus \mathcal O_0) \neq \emptyset$. Let us take $y_n \in S_{t_n} B \cap (D \setminus \mathcal O_0)$. Since $B_0$ is $D$-absorbing, one can take $\tau_B \geq 0$ such that $S_t B \subset B_0$ for all $t \geq \tau_B$. Hence for $n \gg 1$ satisfying $t_n \geq \tau_B$, one observes that $y_n \in S_{t_n} B \subset  B_0$. Therefore, up to a subsequence, $y_n$ converges to an element $y$ of $B_0$ in $(D,\d)$. Moreover, let $n_0 \in \mathbb N$ be such that $S_{t_{n_0}} B \subset B_0$. Then since $y_n \in S_{t_n - t_{n_0}} \circ S_{t_{n_0}} B \subset S_{t_n - t_{n_0}} B_0 \subset \overline{F_{t_n - t_{n_0}}}$ for all $n \gg 1$, the limit $y$ belongs to $\mathcal U$ (see Lemma \ref{L:apdx-01} below). On the other hand, by $y_n \in D \setminus \mathcal O_0$, the limit $y$ never belongs to $\mathcal U$. This is a contradiction. Therefore $\mathcal U$ enjoys the attraction property in $(D,\d)$.

   \begin{lemma}\label{L:apdx-01}
    Let $(X_n)$ be a sequence of closed subsets of a metric space $(D,\d)$.
    Let $y_n \in X_n$ be such that $y_n \to y$ in $(D,\d)$.
    In addition, suppose that $X_n \subset X_m$ if $n \geq m$.
    Then it holds that
    $$
    y \in \bigcap_{k \in \mathbb N} X_k.
    $$
   \end{lemma}

   \begin{proof}
    Let $k \in \mathbb N$ be arbitrarily fixed. For any $n \geq k$, we recall $y_n \in X_n \subset X_k$. Hence the closedness of $X_k$ implies $y \in X_k$. From the arbitrariness of $k$, we conclude that $y \in \cap_{k \in \mathbb N} X_k$.
   \end{proof}

\noindent
\emph{Strict invariance.} We claim that $S_t \,\mathcal U \subset \mathcal U$. Let $y \in S_t \,\mathcal U$. Then there exists $u \in \mathcal U$ such that $y = S_t u$. Moreover, $u$ belongs to $\overline{F_\tau}$ for any $\tau \geq \tau_0$. Hence in particular, by a diagonal argument, one can take a sequence $u_n \in F_n$ such that $u_n \to u$ in $(D,\d)$. Indeed, for each $m \in \mathbb N$, since $u$ belongs to $\overline{F_m}$, one can take a sequence $(u^{(m)}_n)_{n \in \mathbb N}$ in $F_m$ such that $u^{(m)}_n \to u$ in $(D,\d)$ as $n \to \infty$. Now let $u_n := u^{(n)}_n \in F_n$ and observe that $u_n \to u$ in $(D,\d)$.
 Furthermore, there exist sequences $t_n \geq n$ and $b_n \in B_0$ such that $u_n = S_{t_n} b_n$. %Since $B_0$ is comact in $(D,\d)$, there exists $b \in B_0$ such that $b_n \to b$ in $(D,\d)$.
Here one can suppose that $t_n$ is increasing without any loss of generality. Thus one can write
$$
y = S_t u = S_t ( \lim_{n \to \infty} u_n) = \lim_{n \to \infty} S_t u_n = \lim_{n \to \infty} S_t \circ S_{t_n} b_n = \lim_{n \to \infty} S_{t + t_n} b_n.
$$
Here we used the continuity of $S_t$ in $(D,\d)$, which follows from the continuous dependence of solutions for \eqref{pde}--\eqref{ic} on initial data (see Theorem \ref{T:ex}), to verify the third equality. Noting that $b_n \in B_0$, we deduce that $S_{t + t_n} b_n \in F_{t + t_n}$. Therefore $y$ belongs to $\mathcal U$ (see Lemma \ref{L:apdx-01}), which also implies the relation $S_t \, \mathcal U \subset \mathcal U$. We next show $\mathcal U \subset S_t \, \mathcal U$. Let $y \in \mathcal U$ be fixed. Then $y$ belongs to $\overline{F_\tau}$ for all $\tau \geq \tau_0$. Hence one can particularly take a sequence $t_n \nearrow \infty$ and $b_n \in B_0$ such that $y = \lim_{n \to \infty} S_{t_n} b_n$. Note that
$$
y = \lim_{n \to \infty} S_{t_n} b_n = \lim_{n \to \infty} S_t \circ S_{t_n - t} b_n = S_t \left( \lim_{n \to \infty} S_{t_n - t} b_n \right)
$$
for each $t > 0$. Here we used the continuity of $S_t$ again and further noticed that
$$
S_{t_n - t} b_n \in S_{t_n - t} B_0 \subset B_0
$$
for $n \gg 1$, since $B_0$ is $D$-absorbing, and therefore, the compactness of $B_0$ implies that $S_{t_n - t} b_n$ converges to an element $u_1 \in B_0$ in $(D,\d)$, up to a subsequence, as $n \to \infty$. Recall that $S_{t_n - t} b_n \in F_{t_n - t}$ for $n \gg 1$ to obtain
$$
u_1 = \lim_{n \to \infty} S_{t_n - t} b_n \in \mathcal U
   $$
   (see also Lemma \ref{L:apdx-01}).
   Thus $y$ belongs to $S_t \, \mathcal U$. Consequently, we conclude that $S_t \, \mathcal U = \mathcal U$.

   Finally, let us prove the maximality of $\mathcal U$ among bounded strictly invariant sets. Indeed, let $\mathcal V$ be a bounded strictly invariant set in $(D,\d)$. Then since $\mathcal V$ is a bounded set in $(D,\d)$, one can take $\tau \geq 0$ such that $S_t \,\mathcal V \subset B_0$ for all $t \geq \tau$. From \eqref{U} along with the strict invariance of $\mathcal V$, it follows that $\mathcal V \subset \mathcal U$. Thus $\mathcal U$ is maximal.
  \end{proof}

  Not surprisingly, we observe that
  \begin{proposition}\label{P:GA}
   Let $r > 0$ and let $\psi \in D_r$ be a solution of the inclusion,
   \begin{equation}\label{sta}
    \partial \iI(0) - \Delta \psi + \psi^3 - \lam \psi \ni 0 \ \mbox{ in } L^2(\Omega)
   \end{equation}
   {\rm (}hence $\psi$ is a supersolution to the elliptic Allen-Cahn equation \eqref{cAC-e}{\rm )}.%, $-\Delta w + w^3 - \lam w = 0$ in $\Omega$, $w|_{\partial \Omega} = 0${\rm )}.
   Then $\psi$ belongs to the global attractor $\mathcal U$ constructed in Theorem \ref{T:GA} under the phase set $D = D_r$.
  \end{proposition}

  \begin{proof}
   We note that \eqref{sta} corresponds to a stationary equation for (P). More precisely, $u(x,t) \equiv \psi(x)$ is a solution for (P) with $u_0 = \psi$. Hence $\psi$ must belong to the absorbing set $B_0$ (see Lemma \ref{L:absorbing}). Therefore by means of \eqref{U} along with the fact $S_\tau \psi = \psi$, one can conclude that $\psi \in \mathcal U$.
  \end{proof}
  
 The connectedness of the global attractor $\mathcal U$ (with $D = D_r$) is not proved due to the peculiar setting of the phase set $D_r$. However, we can prove it by assigning the following set $D_r^+$ to the phase set $D$ instead of $D_r$: 
$$
D_r^+ := \Big\{
% u \in H^2(\Omega) \cap H^1_0(\Omega) \cap L^6(\Omega) 
u \in D_r
  \colon
 u \geq 0 \ \mbox{ a.e.~in } \Omega%, \quad
% \|(\Delta u - u^3 + \lam u)_-\|_2^2 \leq r
 \Big\}.
 $$
 Here we remark that $D_r^+$ is still non-compact in $H^1_0(\Omega)$ and unbounded in $H^2(\Omega)$ (cf.~see (i) of Remark \ref{R:se}). Then the preceding argument still runs as before. Indeed, the nonnegativity of initial data is inherited to solutions of (P).  
 \begin{proposition}\label{P:Dconv}
It holds that
  \begin{enumerate}
   \item $D_r^+$ is convex,
   \item $\mathcal U$ is connected if $D = D_r^+$.
  \end{enumerate}
 \end{proposition}

 \begin{proof}
  We first prove (i). Let $u,v \in D_r^+$ and $\theta \in (0,1)$ and set $u_\theta := (1-\theta)u+\theta v$. Note that
  \begin{align*}
  \Delta u_\theta + \lam u_\theta - u_\theta^3
   &= (1-\theta) \left[ \Delta u + \lam u \right] + \theta \left[ \Delta v + \lam v \right] - \left( (1-\theta)u+\theta v \right)^3\\
   &\geq (1-\theta) \left[ \Delta u - u^3 + \lam u \right] + \theta \left[ \Delta v - v^3 + \lam v \right]
  \end{align*}
  by the convexity of the cubic function $x^3$ on $[0,\infty)$. Hence the decrease as well as the convexity of the function $x \mapsto (x)_-^2$ lead us to observe that
  \begin{align*}
   \left( \Delta u_\theta - u_\theta^3 + \lam u_\theta \right)_-^2
   &\leq \left( (1-\theta)\left[\Delta u - u^3 + \lam u \right] + \theta\left[\Delta v - v^3 + \lam v\right] \right)_-^2\\
   &\leq (1-\theta) \left( \Delta u - u^3 + \lam u \right)_-^2 + \theta \left( \Delta v - v^3 + \lam v \right)_-^2.
  \end{align*}
  Thus integrating both sides over $\Omega$ and recalling the fact that $u,v \in D_r^+$, we obtain
  $$
  C(\Delta u_\theta - u_\theta^3 + \lam u_\theta) \leq r,
  $$
  which implies $u_\theta \in D_r^+$. Therefore $D_r^+$ is convex.

  We next prove (ii).  In the proof of Lemma \ref{L:B0}, we have shown that $B_0$ is bounded in $\HzweiLsechs$. Hence one can take $R > 0$ such that
  $$
  B_0 \subset B_1 := \left\{ u \in D \colon \|u\|_{H^2(\Omega)} \leq R \right\}.
  $$
  Then since $D = D_r^+$ is convex, so is $B_1$, and hence, $B_1$ is connected in $(X,\d)$. Moreover, we can verify that $B_1$ is compact in $(D,\d)$. Furthermore, by Lemma \ref{L:absorbing}, we can take $t_0 > 0$ such that $S_{t} B_1 \subset B_0$ for all $t \geq t_0$. Hence $S_{t} B_1 \subset B_0 \subset B_1$ for all $t \geq t_0$, and therefore, it holds that
  $$
  \mathcal U = \bigcap_{\tau \geq \tau_0} \overline{E_\tau}, \quad E_\tau := \bigcup_{t \geq \tau} S_t B_1.
  $$
 Moreover, due to the continuity of $S_t$ in $(D,\d)$, the set $S_t B_1$ is also connected for each $t \geq 0$. Furthermore, since the family $\{S_t B_1\}_{t \geq 0}$ has a nonempty intersection (indeed, every stationary point in $B_1$ (e.g., $0 \in B_1$) belongs to the intersection), the union $E_\tau = \cup_{t \geq \tau} S_t B_1$ is connected as well. Therefore the closure $\overline{E_\tau}$ is also connected. Finally, Lemma \ref{L:connected} below ensures the connectedness of $\mathcal U = \cap_{\tau \geq \tau_0} \overline{E_\tau}$, since $\overline{E_\tau}$ is included in the compact set $B_1$ for $\tau \geq t_0$.
 \end{proof}

\begin{lemma}[see e.g. {\cite[p.437]{Engelking}}]\label{L:connected}
 Let $X$ be a compact Hausdorff space. Let $\mathcal P$ be a family of nonempty, closed and connected subsets of $X$ such that either $A \subset B$ or $B \subset A$ holds true for any $A,B \in \mathcal P$. Then the intersection
 $$
 \bigcap \,\mathcal P := \bigcap_{A \in \mathcal P} A 
 $$
 is also connected.
\end{lemma}

\section{Convergence to equilibria}\label{S:conv}

We next discuss the convergence of each solution $u = u(x,t)$ for (P) as $t$ goes to $\infty$. We shall prove the $\omega$-limit set is non-empty and a singleton. Moreover, the limit is characterized as a solution of an elliptic variational inequality of obstacle type.

  \begin{theorem}\label{T:conv}
   Let $u_0 \in \overline{D_r}^{L^2}$ with an arbitrary $r > 0$ and let $u$ be the solution of {\rm (P)} as well as \eqref{pde2}--\eqref{ic2} {\rm (}see Theorem \ref{T:reform}{\rm )}. Then it holds that
   \begin{align*}
    u(t) \to \phi \quad &\mbox{ strongly in } H^1_0(\Omega) \cap L^4(\Omega),\\
    &\mbox{ weakly in } H^2(\Omega) \cap L^6(\Omega) \ \mbox{ as } \ t \to \infty
\end{align*}
   for some $\phi \in H^2(\Omega) \cap H^1_0(\Omega) \cap L^6(\Omega)$. Hence the $\omega$-limit set $\omega(u)$ of $u$ is non-empty and a singleton. Moreover, the limit $\phi$ is a solution of the following elliptic variational inequality of obstacle type\/{\rm :}
   \begin{equation}\label{st}
    \partial \iII(\phi) - \Delta \phi + \phi^3 \ni \lam \phi \ \mbox{ in } L^2(\Omega),  \quad \phi \in H^1_0(\Omega), 
   \end{equation}
   which is rewritten as
    \begin{align*}
  \phi \geq u_0, \quad - \Delta \phi + \phi^3 - \lam \phi \geq 0 \quad \mbox{ in } \ \Omega,\\
  \left( \phi - u_0 \right)\left( - \Delta \phi + \phi^3 - \lam \phi \right) = 0 \quad \mbox{ in } \ \Omega,\\
  \phi|_{\partial \Omega} = 0.
 \end{align*}
  \end{theorem}

%  
%\begin{remark}[Elliptic obstacle problem]
% {\rm
% Problem \eqref{st} is rewritten as the following obstacle problem of elliptic type:
% }
%\end{remark}
%  
 
   \begin{proof}
    Even though the uniqueness of solutions to (P) is not guaranteed, the solution of (P) as well as of \eqref{pde2}--\eqref{ic2} is unique (see Theorem \ref{T:reform}). Hence all energy inequalities are valid (see Appendix \S \ref{S:Aex} and \S \ref{A:reform}). The following proof is based on a strategy used in~\cite{AK09}.
 Recall \eqref{e1} and the boundedness of $E(\cdot)$ from below. Then $E(u(t))$ decreasingly converges to a number $E_\infty$ as $t \to \infty$. Moreover, by \eqref{e1-1}, there is a sequence $\tau_n \in [n,n+1]$ such that
 $$
 u_t(\tau_n) \to 0 \quad \mbox{ strongly in } L^2(\Omega).
 $$
 Furthermore, since $u(t)$ is bounded in $H^1_0(\Omega) \cap L^4(\Omega)$ for $t \geq 0$, up to a (not relabeled) subsequence, there exists $\phi \in H^1_0(\Omega) \cap L^4(\Omega)$ such that
 \begin{alignat*}{4}
  u(\tau_n) &\to \phi \quad  &&\mbox{ weakly in } H^1_0(\Omega) \cap L^4(\Omega) \mbox{ and strongly in } L^2(\Omega).
 \end{alignat*}
   We also further derive from \eqref{e3} (with $s$ and $\|\eta(s)\|_2^2$ replaced by $0$ and $r$, respectively) and \eqref{e6-1} along with \eqref{A} that
 \begin{alignat*}{4}
  \eta(\tau_n) &\to \eta_\infty \quad &&\mbox{ weakly in } L^2(\Omega),\\
  - \Delta u(\tau_n) + u(\tau_n)^3 &\to - \Delta \phi + \phi^3 \quad &&\mbox{ weakly in } L^2(\Omega),
 \end{alignat*}
 which along with the demiclosedness of $\partial \iI$ gives $\eta_\infty \in \partial \iI(0)$. Therefore we assure that
 \begin{equation}\label{st-0}
  \eta_\infty - \Delta \phi + \phi^3 - \lam \phi = 0, \quad \eta_\infty \in \partial \iI(0),
 \end{equation}
   which is a necessary condition for \eqref{st}. Noting that
 \begin{align*}
  \limsup_{n \to \infty} \|\nabla u(\tau_n)\|_2^2
  &= \lam  \lim_{n \to \infty}  \|u(\tau_n)\|_2^2 - \liminf_{n \to \infty} \|u(\tau_n)\|_4^4 - \lim_{n \to \infty} (\eta(\tau_n), u(\tau_n))\\
  &\quad - \lim_{n \to \infty} (u_t(\tau_n),u(\tau_n))\\
  &\leq \lam \|\phi\|_2^2 - \|\phi\|_4^4 - (\eta_\infty,\phi) \stackrel{\eqref{st-0}}{=} \|\nabla \phi\|_2^2
 \end{align*}
   and also deriving in a similar way that
   $$
   \limsup_{n \to \infty} \|u(\tau_n)\|_4^4
   \leq \|\phi\|_4^4,
   $$
 we deduce by the uniform convexity of $H^1_0(\Omega)$ and $L^4(\Omega)$ that
 $$
 u(\tau_n) \to \phi \quad \mbox{ strongly in } H^1_0(\Omega) \cap L^4(\Omega).
 $$
 Thus one can also identify the limit $E_\infty = E(\phi)$.

 It follows from \eqref{e1-1} that
 \begin{align}
  \|u(t) - u(s)\|_2 &\leq \left( \int^t_s \|u_\tau(\tau)\|_2^2 \, \d \tau \right)^{1/2} \sqrt{t-s}\nonumber\\
  &\leq \left(\int^\infty_s \|u_\tau (\tau)\|_2^2 \, \d \tau \right)^{1/2} \sqrt{t-s} \quad \mbox{ for } \ 0 \leq s \leq t < \infty.\label{ut-s}
 \end{align}
 Let $t_n \to \infty$ be an increasing sequence. Then one can take a (not relabeled) subsequence of $(\tau_n)$ such that $0 \leq t_n - \tau_n \leq 1$. Therefore, putting $t = t_n$ and $s = \tau_n$ to \eqref{ut-s},
 $$
 u(t_n) \to \phi \quad \mbox{ strongly in } L^2(\Omega),
 $$
 which along with \eqref{e1-1} implies, up to a (not relabeled) subsequence,
 $$
 u(t_n) \to \phi \quad \mbox{ weakly in } H^1_0(\Omega) \cap L^4(\Omega).
 $$
 
 Therefore combining all these facts, we observe that
 \begin{align*}
  \dfrac 1 2 \limsup_{n \to \infty} \|\nabla u(t_n)\|_2^2
  &=  \lim_{n \to \infty} E(u(t_n)) - \dfrac 1 4  \liminf_{n \to \infty} \|u(t_n)\|_4^4 + \dfrac 1  2  \lam  \lim_{n \to \infty} \|u(t_n)\|_2^2\\
  &\leq E(\phi) - \dfrac 1 4 \|\phi\|_4^4 + \dfrac 1 2 \lam \|\phi\|_2^2 = \dfrac 1 2 \|\nabla \phi\|_2^2,
 \end{align*}
 which together with the uniform convexity of $H^1_0(\Omega)$ ensures that
 $$
 u(t_n) \to \phi \quad \mbox{ strongly in } H^1_0(\Omega).
 $$
Similarly, one can prove that $u(t_n) \to \phi$ strongly in $L^4(\Omega)$. Consequently, $\phi$ is an element of the $\omega$-limit set $\omega(u)$ of $u$.    Furthermore, from the non-decrease of $t \mapsto u(x,t)$ for a.e.~$x \in \Omega$, we conclude that
 $$
 u(x,t) \nearrow \phi(x) \quad \mbox{ for a.e. } \ x \in \Omega
 \ \mbox{ as } \ t \to \infty.
 $$
 Hence $\omega(u) = \{\phi\}$. 

   Now, recall that $u$ also solves \eqref{pde2}--\eqref{ic2} and $\eta$ belongs to $\partial \iII(u)$. By the demiclosedness of $\partial \iII$ in $L^2(\Omega) \times L^2(\Omega)$, we conclude that $\eta_\infty$ is a section of $\partial \iII(\phi)$ a.e.~in $\Omega$. Thus $\phi$ turns out to be a solution of \eqref{st}. This completes the proof.
   \end{proof}

\begin{remark}
 {\rm
 \begin{enumerate}
  \item To prove that the $\omega$-limit set is a singleton, \L ojasiewicz-Simon type inequalities are often used. However, it seems difficult to apply them to (P), since \eqref{pde} is not a gradient flow but a generalized one, which can be written in the form,
 $$
 u_t + \partial \iI(u_t) \ni - E'(u), 
 $$
	where $E'$ stands for a functional derivative of $E$ (i.e., Fr\'echet derivative). %Therefore, we should handle the subdifferential term in a proper way. On the other hand, we emphasize that in our setting,
	On the other hand,  this point was proved more easily since solutions of (P) are non-decreasing in time.

  \item As for the parabolic obstacle problem \eqref{pde2}--\eqref{ic2}, it also seems difficult to apply a \L ojasiewicz-Simon type inequality due to the presence of the nonsmooth potential $\iII$; however, by reducing the obstacle problem to (P), one can prove that the $\omega$-limit set of each solution for the obstacle problem is a singleton and consists of a single solution to \eqref{st}.

% \item If $\phi$ solves \eqref{st}, then $\phi$ always satisfies \eqref{st-0}.
% \item Equation \eqref{st} may have multiple solutions. On the other hand, the $\omega$-limit set of each evolutionary solution is a singleton and solves \eqref{st}. Hence $u(t)$ converges to one of solutions for \eqref{st} and it is still unclear which solution of \eqref{st} is the $\omega$-limit of $u$.
\end{enumerate}
   }
  \end{remark}

\begin{comment}
  Not surprisingly, we also have:
  \begin{corollary}\label{C:conv-GA}
   Under the same assumptions as in Theorem \ref{T:conv}, let $\phi$ be the limit of the solution $u = u(x,t)$ for {\rm (P)} with $u_0 \in D_r$. Then $\phi$ belongs to the global attractor $\mathcal U$ constructed in Theorem \ref{T:GA} in the phase space $D = D_r$ with the metric $\d(\cdot,\cdot)$ given by \eqref{d}.
  \end{corollary}

  \begin{proof}
   Indeed, for a sufficiently large $t_0 > 0$, $u(t_0)$ belongs to $B_0$. Hence for each $n \in \mathbb N$, $u(t_0 + n) = S_n u(t_0)$ belongs to $F_n$ given by \eqref{U}. By Theorem \ref{T:conv}, $u(t_0+n)$ converges to $\phi$ strongly in $H^1_0(\Omega) \cap L^4(\Omega)$ (i.e. in $(D_r,d(\cdot,\cdot))$) as $n \to \infty$. Hence due to Lemma \ref{L:apdx-01}, we conclude that $\phi \in \mathcal U$.
  \end{proof}
\end{comment}  
  
  In Theorem \ref{T:conv}, the rate of convergence is not estimated. Under a suitable assumption on initial data, by employing \eqref{e2}, one can verify an exponential convergence of $u(t)$ as $t \to \infty$.
  \begin{corollary}\label{C:exp-conv}
   In addition to the same assumptions as in Theorem \ref{T:conv}, suppose that
   \begin{equation}\label{a:exp-conv}
   u_0 \geq 0 \quad \mbox{ and } \quad \lambda_\Omega (3 u_0^2) > \lam.
   \end{equation}
   Set $\sigma := \lambda_\Omega(3u_0^2) - \lam  > 0$ and $C = \|(\Delta u_0 - u_0^3 + \lam u_0)_+\|_2$. Then it holds that
   $$
   \|u(t) - \phi\|_2 \leq \dfrac{C}{\sigma} e^{-\sigma t}  \quad \mbox{ for all } \ t \geq 0.
   $$
  \end{corollary}

  \begin{proof}
   By Theorem \ref{T:conv}, it is already known that $u(t)$ converges to some equilibrium $\phi$ strongly in $H^1_0(\Omega) \cap L^4(\Omega)$ as $t \to \infty$. Moreover, setting  $\sigma := \lambda_\Omega(3u_0^2) - \lam > 0$  and letting $s_0 > 0$, we observe that
   \begin{align*}
    \|u(t) - u(s)\|_2 &\leq \int^t_s \|\partial_\tau u(\tau)\|_2 \, \d \tau\\
    &\stackrel{\eqref{e2}}{\leq} C \int^t_s e^{- \sigma \tau} \, \d \tau
    \leq \dfrac C \sigma \left( e^{-\sigma s} - e^{-\sigma t} \right)
    \quad \mbox{ for } \ s_0 \leq s \leq t < \infty
   \end{align*}
   for some constant $C \geq 0$. Letting $t \to \infty$, we deduce that
   $$
   \|\phi - u(s)\|_2 \leq \dfrac C \sigma e^{-\sigma s} \quad \mbox{ for all } \ s \geq s_0.
   $$
   This completes the proof.
  \end{proof}

\begin{remark}[On assumption \eqref{a:exp-conv}]\label{R:exp-conv}
 {\rm
 Note that $\lambda_\Omega(3u_0^2) > \mu(\Omega) > 0$ by $u_0 \not\equiv 0$, where $\mu(\Omega)$ stands for the first eigenvalue of the Dirichlet Laplacian $-\Delta$ posed in $\Omega$. Hence, the second inequality of \eqref{a:exp-conv} holds true if $\mu(\Omega) \geq \lam$ (e.g., the diameter of $\Omega$ is small enough). On the other hand, even if $\mu(\Omega) < \lam $, the second condition of \eqref{a:exp-conv} is also satisfied under an appropriate assumption on the initial datum $u_0$, for instance,
 $$
 3u_0^2 \geq U_\lambda \quad \mbox{ a.e.~in } \Omega,
 $$
 where $U_\lambda = U_\lambda(x)$ is the (unique) positive solution of the elliptic equation for any $\lambda > \lam$,
 \begin{equation}\label{eAC-2}
 - \Delta U_\lambda + U_\lambda^2 = \lambda U_\lambda, \ U_\lambda > 0 \  \mbox{ in } \Omega, \quad U_\lambda = 0 \ \mbox{ on } \partial \Omega.
 \end{equation}
 Indeed, for each $\lambda > \lam$ (hence $\lambda > \mu(\Omega)$), \eqref{eAC-2} admits the unique positive solution $U_\lambda \in C^2(\Omega) \cap C(\overline \Omega)$ such that $0 < U_\lambda \leq \lambda$ in $\Omega$, and moreover, $(\lambda , U_\lambda)$ turns out to be a principal eigenpair of the Schr\"odinger operator $v \mapsto -\Delta v + U_\lambda v$. Hence if $3u_0^2 \geq U_\lambda$ a.e.~in $\Omega$, then $\lambda_\Omega(3u_0^2) \geq \lambda_\Omega(U_\lambda) = \lambda > \lam$.
 }
\end{remark}

\appendix

\section{Proof of the existence part of Theorem {\ref{T:ex}} and derivation of energy inequalities}\label{S:Aex}

In this section, we give a proof for the existence part of Theorem \ref{T:ex} and a rigorous derivation of energy inequalities, which are derived in \S \ref{S:e} by formal computations. More precisely, we shall prove
\begin{theorem}\label{A:T:ex}
 Let $r > 0$ be arbitrarily fixed.
 \begin{enumerate}
  \item[(i)] Let $u_0$ belong to the closure $\overline{D_r}^{L^2}$ of $D_r$ in $L^2(\Omega)$. Then {\rm (P)} admits the unique solution  $u = u(x,t)$  satisfying  all the regularity conditions as in {\rm (i)} of Theorem \ref{T:ex} 
\begin{comment}
\begin{align*}
  u \in C_w([\delta,T];H^2(\Omega) \cap L^6(\Omega)), \quad u_t \in L^2(\delta,T;H^1_0(\Omega)),\\
  \eta \in L^\infty(0,T;L^2(\Omega))
\end{align*}
 for any $0 < \delta < T < \infty$
\end{comment}
 such that \eqref{e1}, \eqref{e1-2}, \eqref{e4}--\eqref{e4-1}, \eqref{e5}, \eqref{e5-2} and \eqref{e6-3} hold true with $\|\eta_0\|_2^2$ replaced by $r$. Moreover, it is also satisfied that
\begin{equation}\label{eta-ini}
\|\eta(t)\|_2^2 \leq r \quad \mbox{ for a.e. } \ t > 0.
\end{equation}
  \item[(ii)] If $u_0$ also belongs to the closure $\overline{D_r}^{H^1_0 \cap L^4}$ of $D_r$ in $H^1_0(\Omega) \cap L^4(\Omega)$, then the solution  $u = u(x,t)$ also satisfies all the regularity conditions as in {\rm (ii)} of Theorem \ref{T:ex}.  Moreover, \eqref{e1}--\eqref{e1-2}, \eqref{e4}--\eqref{e5-2}, \eqref{e6-1}--\eqref{diss2} and \eqref{eta-ini} are satisfied with $\|\eta_0\|_2^2$ replaced by $r$.
  \item[(iii)] If $u_0 \in H^2(\Omega) \cap L^6(\Omega)$, then  the solution $u = u(x,t)$ also fulfills all the regularity conditions as in {\rm (iii)} of Theorem \ref{T:ex}.  Moreover, \eqref{e1}--\eqref{e1-2}, \eqref{e4}--\eqref{e5-2}, \eqref{e6-int}, \eqref{e6-1}--\eqref{diss2} are satisfied with $\|\eta_0\|_2^2$ replaced by $\|(\Delta u_0 - u_0^3 + \lam u_0)_-\|_2^2$. Furthermore, it holds that
\begin{equation}\label{eta-ini2}
\|\eta(t)\|_2^2 \leq \|(\Delta u_0 - u_0^3 + \lam u_0)_-\|_2^2 \quad \mbox{ for a.e. } \ t > 0.
\end{equation}
  \item[(iv)] If $u = u(x,t)$ also solves the obstacle problem \eqref{pde2}--\eqref{ic2} {\rm (}see Theorem \ref{T:reform}{\rm )}, then \eqref{e3}, \eqref{e6} and \eqref{e6-0} hold. In addition, if $u_0 \in H^2(\Omega) \cap H^1_0(\Omega) \cap L^6(\Omega)$, then \eqref{e2} is also satisfied with $\|v_0\|_2^2$ replaced by $\|(\Delta u_0 - u_0^3 + \lam u_0)_+\|_2^2$.
 \end{enumerate}
\end{theorem}

Now, we give a proof of Theorem \ref{A:T:ex} below.

\subsection{Reduction to an abstract Cauchy problem}

Let $T > 0$ be arbitrarily fixed. Set $H = L^2(\Omega)$, $V = H^1_0(\Omega) \cap L^4(\Omega)$ and define a functional $\psi$ on $H$ as in \eqref{varphi}. Moreover, set $\varphi : H \to [0,\infty]$ by
$$
\varphi(u) = \dfrac 1 2 \|u\|_2^2 + \iI(u) \quad \mbox{ for } \ u \in H,
$$
which is homogeneous of degree $p = 2$. Then as in \S \ref{S:e}, (P) is reduced to the abstract Cauchy problem \eqref{ee}, which is also equivalent to
\begin{equation}\label{ee-2}
\partial \varphi(u_t) + \partial \psi(u) \ni \lam u \ \mbox{ in } H, \quad 0 < t < T, \quad u(0) = u_0.
\end{equation}
In order to prove the existence of solutions to \eqref{ee-2}, it suffices to check assumptions (A.1), (A.2), ${\rm (A.3)}'$, (A.4), (A.5) of~\cite{Arai} (see also~\cite{Barbu75}). Since (A.1)--(A.4) follow immediately from the setting of $\psi$ and $\varphi$, we only give a proof for checking (A.5) (i.e., the \emph{$\partial \varphi$-monotonicity} of $\partial \psi$) below.
\begin{lemma}\label{L:A5}
Let $J_\lambda$ be the resolvent of $\partial \psi$, that is, $J_\lambda := (I + \lambda \partial \psi)^{-1}$. Then it holds that
$$
\iI(J_\lambda u - J_\lambda v) \leq \iI(u - v), \quad \|J_\lambda u - J_\lambda v\|_2^2 \leq \|u-v\|_2^2
$$
 for $u, v \in H$. In particular,
 $$
 \varphi(J_\lambda u - J_\lambda v) \leq \varphi(u-v) \quad \mbox{ for all } \ u,v \in H.
 $$
\end{lemma}

\begin{proof}
The second inequality follows from a well-known fact that resolvents of maximal monotone operators are non-expansive, i.e., $\|J_\lambda u - J_\lambda v\|_H \leq \|u-v\|_H$ for $u,v \in H$ (see e.g.~\cite{HB1}). So it remains to prove the first inequality. In case $\iI(u-v)=\infty$, we have nothing to prove. In case $\iI(u-v) = 0$, that is, $u \geq v$ a.e.~in $\Omega$, by the definition of $J_\lambda$, we see that
\begin{equation}\label{diff-Jlam}
J_\lambda u - J_\lambda v + \lambda \left[ \partial \psi(J_\lambda u) - \partial \psi(J_\lambda v) \right] = u - v.
\end{equation}
Test both sides by $ -(J_\lambda u - J_\lambda v)_- \leq 0$ to get
$$
\int_\Omega \left( J_\lambda u - J_\lambda v \right)_-^2 \, \d x
\leq  -\int_\Omega (u-v) \left( J_\lambda u - J_\lambda v \right)_-  \, \d x \leq 0,
$$
 which implies $J_\lambda u \geq J_\lambda v$ a.e.~in $\Omega$. Here we used the fact that
 \begin{align*}
  \left( \partial \psi(J_\lambda u) - \partial \psi(J_\lambda u) ,  -(J_\lambda u - J_\lambda v)_-  \right)
  = \left( - \Delta (J_\lambda u - J_\lambda v) ,  -(J_\lambda u - J_\lambda v)_-  \right)\\
  + \left( |J_\lambda u|^2 J_\lambda u - |J_\lambda v|^2 J_\lambda v,  -(J_\lambda u - J_\lambda v)_-  \right) \geq 0
 \end{align*}
 by monotonicity. Thus $\iI(J_\lambda u - J_\lambda v) = 0$.
\end{proof}

\subsection{Proof of (iii)}

Let us prove (iii). To this end, suppose that
\begin{equation}\label{u0-iii}
 u_0 \in D(\partial \psi) = H^2(\Omega) \cap H^1_0(\Omega) \cap L^6(\Omega).
\end{equation}
Then thanks to Arai~\cite[Theorem 3.3]{Arai} (see also Barbu~\cite{Barbu75}), we assure that \eqref{ee-2} admits a solution $u \in W^{1,\infty}(0,T;H) \cap L^\infty(0,T;V)$ such that the function $t \mapsto \varphi(u'(t))$ belongs to $L^\infty(0,T)$ and the function $t \mapsto \psi(u(t))$ is absolutely continuous on $[0,T]$. Concerning energy inequalities, one can rigorously derive \eqref{e1}--\eqref{e1-2} as in \S \ref{S:e} under the frame of Definition \ref{D:sol}. So we shall verify the other energy inequalities. To this end, the rest of this subsection is devoted to preparing auxiliary steps.

Recall approximate problems \eqref{P-approx} for (P) and denote by $u_\lambda$ the unique solution. Furthermore, let $\eta_\lambda$ be the section of $\partial \iI(\partial_t u_\lambda)$ satisfying
\begin{equation}\label{P-approx:2}
\partial_t u_\lambda + \eta_\lambda + \partial \psi_\lambda(u_\lambda) = \lam u_\lambda, \quad u_\lambda(0) = u_0.
\end{equation}
Set $p_\lambda := \partial_t u_\lambda + \eta_\lambda$. Then $p_\lambda$ is a section of $\partial \varphi(\partial_t u_\lambda)$.

\begin{remark}[Approximate equations in~\cite{Arai}]
 {\rm
 Approximate problems used in~\cite{Arai} seem slightly different from \eqref{P-approx:2}; indeed, they involve a liner relaxation term such as
 $$
 \lambda u_t + \partial \varphi(u_t) + \partial \psi_\lambda(u) \ni \lam u,
 $$
since the quadratic coercivity of $\varphi$ is not assumed. However, concerning \eqref{pde}, one can reproduce the same arguments as in~\cite{Arai} for \eqref{P-approx:2}, since the original equation \eqref{pde} already includes the linear relaxation term. On the other hand, the following arguments also work well for approximate equations with the additional relaxation term as in~\cite{Arai}.
 }
\end{remark}

As mentioned in \S \ref{S:e}, we assure that  $\partial_t u_\lambda, \eta_\lambda \in C^{0,1}([0,T];H)$  and $\eta_\lambda =  -( \lam u_\lambda - \partial \psi_\lambda(u_\lambda) )_- $ in $H$ for each $t \in [0,T]$. In particular, one finds that
\begin{equation*}%\label{eta0}
\eta_\lambda(0) := \lim_{t \to 0_+} \eta_\lambda(t) =  -\big( \lam u_0 - \partial \psi_\lambda(u_0) \big)_- .
\end{equation*}
Moreover, every assertion obtained by~\cite{Arai} for $u_\lambda$ is valid (see proofs of Theorems 3.1 and 3.3 in~\cite{Arai} for details). In particular, let us recall that, up to a (not relabeled) subsequence $\lambda \to 0$,
\begin{alignat*}{4}
 J_\lambda u_\lambda &\to u \quad &&\mbox{ strongly in } C([0,T];H),\\
% & &&\mbox{ weakly star in } L^\infty(0,T;H^1_0(\Omega) \cap L^4(\Omega)),\\ 
 u_\lambda &\to u \quad &&\mbox{ strongly in } C([0,T];H),\\
 \partial_t u_\lambda &\to u_t \quad &&\mbox{ weakly star in } L^\infty(0,T;H),\\
 \partial \psi_\lambda(u_\lambda) &\to \partial \psi(u) \quad &&\mbox{ weakly star in } L^\infty(0,T;H),\\
 p_\lambda &\to p \quad &&\mbox{ weakly star in } L^\infty(0,T;H),
\end{alignat*}
and moreover, $t \mapsto \psi(u(t))$ is (absolutely) continuous on $[0,T]$ (hence, $u \in C([0,T];H^1_0(\Omega) \cap L^4(\Omega))$) and $p \in \partial \varphi(u_t)$. Since $\partial \psi(u) \in L^\infty(0,T;H)$ and $u(t) \in D(\partial \psi) = H^2(\Omega) \cap H^1_0(\Omega) \cap L^6(\Omega)$ for a.e.~$t \in (0,T)$, it follows that $u \in L^\infty(0,T;H^2(\Omega) \cap L^6(\Omega))$ from the fact that $\|\Delta w\|_2^2 + \|w\|_6^6 \leq \|\partial \psi(w)\|_2^2$ for all $w \in D(\partial \psi)$ along with the elliptic estimate $\|w\|_{H^2(\Omega)} \leq C(\|\Delta w\|_2 + \|w\|_2)$ for $w \in H^2(\Omega)$.  Thus $u$ solves (P). Here we further observe that
$$
 J_\lambda u_\lambda \to u \quad \mbox{ weakly star in } L^\infty(0,T;H^1_0(\Omega) \cap L^4(\Omega))
 $$
 and (see~\cite{HB1},~\cite{BCP} and~\cite{B})
 $$
 \lim_{\lambda \to 0} \int^T_0 \left( \partial \psi_\lambda (u_\lambda), J_\lambda u_\lambda \right) \, \d t \to
 \int^T_0 \left( \partial \psi (u), u \right) \, \d t.
 $$
One can also verify that
\begin{align*}
\limsup_{\lambda \to 0} \int^T_0 \|\nabla J_\lambda u_\lambda(t)\|_2^2 \, \d t
 &= \limsup_{\lambda \to 0} \int^T_0 \left( - \Delta J_\lambda u_\lambda , J_\lambda u_\lambda \right) \, \d t\\
&\leq \limsup_{\lambda \to 0} \int^T_0 \left( \partial \psi_\lambda(u_\lambda) - (J_\lambda u_\lambda)^3, J_\lambda u_\lambda \right) \, \d t\\
&\leq \int^T_0 \left( - \Delta u , u \right) \, \d t
= \int^T_0 \|\nabla u(t)\|_2^2 \, \d t,
\end{align*}
which implies
$$
J_\lambda u_\lambda \to u \quad \mbox{ strongly in } L^2(0,T;H^1_0(\Omega)).
$$
Similarly,
$$
J_\lambda u_\lambda \to u \quad \mbox{ strongly in } L^4(0,T;L^4(\Omega)).
$$
Hence
\begin{equation}\label{phi-conv}
\int^T_0 \phi(J_\lambda u_\lambda(t)) \, \d t \to \int^T_0 \phi(u(t)) \, \d t.
\end{equation}
Moreover, by $u \in C([0,T];H^1_0(\Omega) \cap L^4(\Omega)) \cap L^\infty(0,T;H^2(\Omega) \cap L^6(\Omega))$, we deduce that $u \in C_w([0,T];H^2(\Omega) \cap L^6(\Omega))$ (see~\cite{LM}). It follows that
\begin{equation}\label{c:u-H2}
J_\lambda u_\lambda(t) \to u(t) \quad \mbox{ weakly in } H^2(\Omega) \cap L^6(\Omega)
\quad \mbox{ for any } \ t \in [0,T].
\end{equation}
On the other hand, there exists $\eta \in L^\infty(0,T;H)$ such that
$$
\eta_\lambda \to \eta \quad \mbox{ weakly star in } L^\infty(0,T;H)
$$
and $\eta = p - u_t \in \partial \iI(u_t)$. From the equivalence between \eqref{pde} and \eqref{iAC}, we also remark that
\begin{equation}\label{eta-repre}
\eta =  -\left( \lam u - \partial \psi(u) \right)_- = - \left( \Delta u - u^3 + \lam u\right)_-  \ \mbox{ a.e.~in } \ \Omega \times (0,T).
\end{equation}

We next justify formal arguments in \S \ref{S:e} to derive energy inequalities (except Energy Inequality 1 in \S \ref{S:e}). To this end, we claim that
\begin{equation}\label{c1:Ju-reg}
J_\lambda u, \ |J_\lambda u_\lambda|J_\lambda u_\lambda \in W^{1,2}(0,T;H^1_0(\Omega)).
\end{equation}
 Indeed, recalling \eqref{diff-Jlam} with $u$ and $v$ replaced by $u_\lambda(t+h)$ and $u_\lambda(t)$, respectively, and multiplying it by $J_\lambda u_\lambda(t+h) - J_\lambda u_\lambda(t)$, one can derive that
  \begin{align*}
   \lefteqn{
 \dfrac 1 2 \|J_\lambda u_\lambda(t + h) - J_\lambda u_\lambda(t)\|_2^2
  + \lambda \left\|\nabla \left( J_\lambda u_\lambda(t+h) - J_\lambda u_\lambda(t) \right)\right\|_2^2}\\
&+ \dfrac 3 4 \Big\| (|J_\lambda u_\lambda|J_\lambda u_\lambda)(t+h) - (|J_\lambda u_\lambda|J_\lambda u_\lambda) (t) \Big\|_2^2
  %++\lambda \left( (J_\lambda u_\lambda)^3(t+h) - (J_\lambda u_\lambda)^3(t), J_\lambda u_\lambda(t+h) - J_\lambda u_\lambda(t) \right)\\
 \leq \dfrac 1 2 \|u_\lambda(t+h)-u_\lambda(t)\|_2^2
  \end{align*}
  for a.e.~$t \in (0,T)$ and $h \in \mathbb R$ satisfying $t + h \in [0,T]$. Here we also used the fundamental inequality,
  \begin{equation}\label{fi}
  \dfrac 3 4 \Big| |a|a - |b|b \Big|^2 \leq (a^3 - b^3)(a-b) \quad \mbox{ for all } \ a,b \in \mathbb R.
  \end{equation}
   From the arbitrariness of $h$, we deduce that $J_\lambda u_\lambda \in W^{1,2}(0,T;H^1_0(\Omega))$ by $u_\lambda \in  C^{1,1}  ([0,T];L^2(\Omega)) \subset W^{1,2}(0,T;L^2(\Omega))$.

By~\cite[Lemma 3.10]{Arai} and the monotonicity of $\partial \psi_\lambda$ along with Lemma \ref{L:A5}, we have
\begin{lemma}\label{L:phi-mono} For $u \in C^1([0,T];H)$ satisfying $u_t \geq 0$ a.e.~in $\Omega \times (0,T)$, it holds that
\begin{enumerate}
 \item $\iI( (J_\lambda u)_t ) \leq \iI(u_t)$ for a.e.~$t \in (0,T)$, in particular, $(J_\lambda u)_t \geq 0$ a.e.~in $\Omega \times (0,T)$,\\
 \item for any $\eta \in \partial \iI(u_t)$, one has
$$
\left( \eta(t), \dfrac{\d}{\d t} \partial \psi_\lambda(u(t)) \right) \geq 0 \quad \mbox{ for a.e. } t \in (0,T).
$$
\end{enumerate}
\end{lemma}

\subsection{Derivation of Energy Inequalities under \eqref{u0-iii}}

We next derive energy inequalities.

\smallskip
\noindent
{\bf Energy Inequalities 3.} Differentiate both sides of \eqref{P-approx:2} in $t$ (indeed, it is rigorously possible, since both sides of \eqref{P-approx:2} are smooth (in $t$) enough by approximation) and put $v_\lambda := \partial_t u_\lambda \in C^{0,1}([0,T];H) \subset W^{1,\infty}(0,T;H)$. Then
\begin{equation}\label{pde:va}
 \partial_t v_\lambda + \partial_t \eta_\lambda + \dfrac{\d}{\d t} \partial \psi_\lambda(u_\lambda) = \lam v_\lambda.
\end{equation}
Multiplying both sides by $\eta_\lambda$ and employing (ii) of Lemma \ref{L:phi-mono}, we deduce that
$$
\dfrac{\d}{\d t} \iI(v_\lambda) + \dfrac 1 2 \dfrac{\d}{\d t} \|\eta_\lambda\|_2^2 \leq \lam \int_\Omega v_\lambda \eta_\lambda \, \d x = 0,
$$
which leads us to get
\begin{equation}\label{e3-1a}
\|\eta_\lambda(t)\|_2^2 \leq \|\eta_\lambda(0)\|_2^2 = \left\| (\lam u_0 - \partial \psi_\lambda(u_0))_- \right\|_2^2 \quad \mbox{ for all } \ t \in [0,T].
\end{equation}
Since $\partial \psi_\lambda(u_0) \to \partial \psi(u_0)$ strongly in $H$ as $\lambda \to 0$ by $u_0 \in D(\partial \psi)$ (see~\cite{HB1}), one has
$$
\|\eta(t)\|_2^2 \leq \|\eta\|_{L^\infty(0,T;H)}^2 \leq \liminf_{\lambda \to 0} \|\eta_\lambda\|_{L^\infty(0,T;H)}^2 \leq \|(\lam u_0 - \partial \psi(u_0))_-\|_2^2
$$
for a.e.~$t \in (0,T)$. Hence \eqref{eta-ini2} follows.

\smallskip
\noindent
{\bf Energy Inequalities 4--6.}
Thanks to \eqref{eta-ini2}, as in \S \ref{S:e}, one can derive \eqref{e4}--\eqref{e5-2} by replacing $\|\eta_0\|_2$ by $\|(\Delta u_0 - u_0^3 + \lam u_0)_-\|_2$.
As for \emph{Energy Inequality 6}, test \eqref{P-approx:2} by $(\partial \psi_\lambda(u_\lambda) - \lam  u_\lambda)_t$, which is well-defined due to the smoothness of $u_\lambda$ and $\partial \psi_\lambda(u_\lambda)$ in $t$. Then it follows that
\begin{align*}
 \left( \partial_t u_\lambda + \eta_\lambda, \left( \partial \psi_\lambda(u_\lambda) - \lam u_\lambda \right)_t \right)
 + \dfrac 1 2 \dfrac \d {\d t} \left\| \partial \psi_\lambda(u_\lambda) - \lam u_\lambda \right\|_2^2 = 0.
\end{align*}
Here we also observe by (ii) of Lemma \ref{L:phi-mono} that
\begin{align*}
 \left( \eta_\lambda, \left( \partial \psi_\lambda(u_\lambda) - \lam u_\lambda \right)_t \right) = \left( \eta_\lambda, (\partial \psi_\lambda(u_\lambda))_t \right)
 \geq 0,
\end{align*}
and moreover, %by \eqref{Jlam} and \eqref{c1:Ju-reg},
\begin{align}
 \left( \partial_t u_\lambda ,\dfrac \d {\d t} \partial \psi_\lambda(u_\lambda) \right)
 &= \left( (J_\lambda u_\lambda)_t + \lambda \dfrac \d {\d t} \partial \psi_\lambda(u_\lambda) , \dfrac \d {\d t} \partial \psi_\lambda(u_\lambda) \right)\nonumber\\
 &\geq \left( (J_\lambda u_\lambda)_t , \dfrac \d {\d t} \partial \psi_\lambda(u_\lambda) \right)\nonumber\\
 &\stackrel{\eqref{Jlam}}\geq \|\nabla (J_\lambda u_\lambda)_t\|_2^2  + \dfrac 3 4 \left\| \dfrac \d {\d t} \left( |J_\lambda u_\lambda| J_\lambda u_\lambda \right)\right\|_2^2.
 \label{ut*dpsit}
\end{align}
 Here we also used the fact that
 \begin{align*}
  \lefteqn{
  \left( J_\lambda u_\lambda (t+h) - J_\lambda u_\lambda (t), \partial \psi_\lambda (u_\lambda (t+h)) - \partial \psi_\lambda (u_\lambda (t)) \right)
  }\\
  &\geq \Big\| \nabla \left( J_\lambda u_\lambda(t+h) - J_\lambda u_\lambda(t) \right) \Big\|_2^2 + \dfrac 3 4 \Big\| (|J_\lambda u_\lambda| J_\lambda u_\lambda)(t+h) - (|J_\lambda u_\lambda|J_\lambda u_\lambda)(t) \Big\|_2^2 
 \end{align*}
 by \eqref{fi}.  By combining all these facts,
\begin{align*}
 \|\nabla (J_\lambda u_\lambda)_t\|_2^2 + \dfrac 3 4 \left\| \dfrac \d {\d t} \left( |J_\lambda u_\lambda| J_\lambda u_\lambda \right)\right\|_2^2 + \dfrac 1 2 \dfrac \d {\d t} \left\| \partial \psi_\lambda(u_\lambda) - \lam u_\lambda \right\|_2^2\\
\leq \lam \|\partial_t u_\lambda\|_2^2 = - \lam \dfrac{\d}{\d t} E_\lambda(u_\lambda(t)),
\end{align*}
where $E_\lambda(w) := \psi_\lambda(w) - (\lam/2) \|w\|_2^2$.
Integrate both sides over $(0,t)$ to see that
\begin{align}
\lefteqn{
 \int^t_0 \left( \|\nabla (J_\lambda u_\lambda)_\tau\|_2^2 + \dfrac 3 4 \left\| \dfrac \d {\d t} \left( |J_\lambda u_\lambda| J_\lambda u_\lambda \right)\right\|_2^2 \right) \, \d \tau
 }\nonumber\\
& + \dfrac 1 2 \left\| \partial \psi_\lambda(u_\lambda(t)) - \lam u_\lambda(t) \right\|_2^2 + \lam E_\lambda(u_\lambda(t))
\leq \dfrac 1 2 \left\| \partial \psi_\lambda(u_0) - \lam u_0 \right\|_2^2 + \lam E_\lambda(u_0).\label{e6a}
\end{align}
%Here by \eqref{c1:Ju-reg} one can also rewrite
%$$
%\int_\Omega (J_\lambda u_\lambda)^2 \left[ (J_\lambda u_\lambda)_t \right]^2 \, \d x = \dfrac 1 4 \dfrac \d {\d t} \left\| \partial_t (|J_\lambda u_\lambda| J_\lambda u_\lambda) \right\|_2^2.
%$$
Thus
\begin{alignat}{4}
(J_\lambda u_\lambda)_t &\to u_t \quad &&\mbox{ weakly in } L^2(0,T;H^1_0(\Omega)),\label{sharp2-1}\\
\partial_t \left( |J_\lambda u_\lambda| J_\lambda u_\lambda \right) &\to \partial_t (|u|u) \quad &&\mbox{ weakly in } L^2(0,T;L^2(\Omega)).\label{sharp2-2}
\end{alignat}
Passing to the limit in \eqref{e6a} as $\lambda \to 0$ and recalling that $u \in C_w([0,T];H^2(\Omega)\cap L^6(\Omega))$, we have
\begin{align*}
\lefteqn{
 \int^t_0 \left( \|\nabla u_\tau\|_2^2 + \dfrac 3 4 \left\| \dfrac \d {\d t} \left( |J_\lambda u_\lambda| J_\lambda u_\lambda \right)\right\|_2^2 \right) \, \d \tau
 }\\
& + \dfrac 1 2 \left\| \partial \psi(u(t)) - \lam u(t) \right\|_2^2 + \lam E(u(t))
\leq \dfrac 1 2 \left\| \partial \psi(u_0) - \lam u_0 \right\|_2^2 + \lam E(u_0)
\end{align*}
for all $t \in (0,T)$. Furthermore, one can also derive \eqref{e6-1} (with $\|\eta_0\|_2$ replaced by $\|(\Delta u_0 - u_0^3 + \lam u_0)_-\|_2$). % by exploiting \eqref{phi-conv} and \eqref{c:u-H2}.
  Then \eqref{diss2} also follows immediately from \eqref{e6-1} as in \S \ref{S:e}.

\subsection{Proof of (ii)}\label{Ss:ii}

We next prove (ii). Take an approximate sequence $(u_{0,n})$ satisfying
\begin{equation}\label{u0-ii}
 u_{0,n} \in D_r, \quad
 u_{0,n} \to u_0 \quad \mbox{ strongly in } H^1_0(\Omega) \cap L^4(\Omega).
\end{equation}
Since $u_{0,n}$ fulfills \eqref{u0-iii}, the solution $u_n$ of (P) with $u_0$ replaced by $u_{0,n}$ and the section $\eta_n \in \partial \iI(\partial_t u_n)$ as in \eqref{EQ} satisfy all energy inequalities that have been justified in the proof of (iii). Here we mainly use \eqref{e1}--\eqref{e1-2} and \eqref{e5-1} and note by \eqref{eta-ini2} and \eqref{u0-ii} that
$$
E(u_{0,n}) \to E(u_0), \quad
\|\eta_n(t)\|_2^2 \leq \|\left(\Delta u_{0,n} - u_{0,n}^3 + \lam u_{0,n}\right)_-\|_2^2 \leq r \ \mbox{ for a.e. } t > 0.
$$
Hence, by a priori estimates \eqref{e1}--\eqref{e1-2} and \eqref{e5-1} for $u_n$, one can obtain, up to a (not relabeled) subsequence $n \to \infty$,
\begin{alignat*}{4}
 u_n &\to u \quad &&\mbox{ weakly in } W^{1,2}(0,T;H),\\
 & &&\mbox{ weakly star in } L^\infty(0,T;H^1_0(\Omega) \cap L^4(\Omega)),\\
 & &&\mbox{ strongly in } C([0,T];H),\\
 - \Delta u_n + u_n^3 &\to - \Delta u + u^3 \quad &&\mbox{ weakly in } L^2(0,T;H),\\
 \eta_n &\to \eta \quad &&\mbox{ weakly star in } L^\infty(0,T;H),
\end{alignat*}
which also implies $u(t) \in D(\partial \psi)$ for a.e.~$t \in (0,T)$ and $u_t + \eta - \Delta u + u^3 = \lam u$ a.e.~in $\Omega \times (0,T)$. Moreover, as in the proof of \eqref{phi-conv}, one finds that
$$
 \int^T_0 E(u_n(t)) \, \d t \to \int^T_0 E(u(t)) \, \d t.
$$
We next identify the limit $\eta$. We see that
\begin{align*}
 \int^T_0 \left( \eta_n, \partial_t u_n \right) \, \d t
 &= - \int^T_0 \|\partial_t u_n\|_2^2 - E(u_n(T)) + E(u_{0,n}),
\end{align*}
which implies
$$
\limsup_{n \to \infty}  \int^T_0 \left( \eta_n, \partial_t u_n \right) \, \d t
\leq - \int^T_0 \|u_t\|_2^2 - E(u(T)) + E(u_0)
= \int^T_0 (\eta, u_t) \, \d t.
$$
Hence by Minty's trick, we conclude that $u_t \geq 0$ and $\eta \in \partial \iI(u_t)$ a.e.~in $\Omega \times (0,T)$. Since the function $t \mapsto u(t)$ is weakly continuous on $[0,T]$ with values in $H^1_0(\Omega) \cap L^4(\Omega)$ (see~\cite{LM}) and the function $t \mapsto \phi(u(t))$ is (absolutely) continuous on $[0,T]$ (by $u_t \in L^2(0,T;H)$ and $-\Delta u + u^3 \in L^2(0,T;H)$), we also assure by the uniform convexity of $H^1_0(\Omega) \cap L^4(\Omega)$ that
$$
u \in C([0,T];H^1_0(\Omega) \cap L^4(\Omega)).
$$

Concerning energy inequalities, \eqref{e1}--\eqref{e1-2} are (rigorously) derived as in \S \ref{S:e}. Moreover, \eqref{eta-ini} is proved as in the proof of (iii). Hence \eqref{e4}--\eqref{e5-2} can be also rigorously derived with $\|\eta_0\|_2^2$ replaced by $r$.  Moreover, combining \eqref{e6-1} with \eqref{e5-1} for $u_n$, one can verify
\begin{alignat*}{4}
 t^{1/2} \partial_t u_n &\to t^{1/2} u_t \quad &&\mbox{ weakly in } L^2(0,T;H^1_0(\Omega)),\\
 t^{1/2} \partial_t (|u_n|u_n) &\to t^{1/2} \partial_t (|u|u) \quad &&\mbox{ weakly in } L^2(0,T;L^2(\Omega)),\\
 t^{1/2} \Delta u_n &\to t^{1/2} \Delta u \quad &&\mbox{ weakly star in } L^\infty(0,T;L^2(\Omega)),\\
 t^{1/2} u_n^3 &\to t^{1/2} u^3 \quad &&\mbox{ weakly star in } L^\infty(0,T;L^2(\Omega)),
\end{alignat*}
which also yields \eqref{e6-1}--\eqref{diss2}. Thus (ii) has been proved.

\subsection{Proof of (i)}

Finally, let us prove (i). To this end, take $u_{0,n}$ satisfying
\begin{equation}\label{u0-i}
 u_{0,n} \in D_r, \quad
 u_{0,n} \to u_0 \quad \mbox{ strongly in } L^2(\Omega).
\end{equation}
The solution $u_n$ of (P) with $u_0$ replaced by $u_{0,n}$ and the section $\eta_n$ of $\partial \iI(\partial_t u_n)$ satisfy all the energy inequalities that are justified in (iii). Here we mainly use \eqref{eta-ini2}, \eqref{e4}--\eqref{e4-2}, \eqref{e5-2} and \eqref{e6-3} (with $\|\eta_0\|_2$ replaced by $\|(\Delta u_{0,n} - u_{0,n}^3 + \lam u_{0,n})_-\|_2$) for $u_n$ along with the fact that
$$
\|\eta_n(t)\|_2^2 \leq \|\left(\Delta u_{0,n} - u_{0,n}^3 + \lam u_{0,n}\right)_-\|_2^2 \leq r \quad \mbox{ for a.e. } t > 0.
$$
Moreover, \eqref{e4-2} yields
\begin{align*}
 \int^t_0 \tau \|\partial_\tau u_n\|_2^2 \, \d \tau + t E(u_n(t))
 \leq \dfrac{C_1}2 t \left( 1 + \|\eta_0\|_2^{4/3} \right) + \dfrac 1 4 \|u_{0,n}\|_2^2,
\end{align*}
which implies
\begin{equation}\label{Eu-ep}
E(u_n(t)) \leq \dfrac{C_1}2 \left( 1 + r^{2/3} \right) + \dfrac 1 {4t} \|u_{0,n}\|_2^2 \quad \mbox{ for any } \ t > 0.
\end{equation}
Due to the lack of the convergence $E(u_{0,n}) \to E(u_0)$, we need an extra argument. One can obtain the following estimate for solutions $u$ of (P) in the dual space $V^* = H^{-1}(\Omega) + L^{4/3}(\Omega)$ of $V = H^1_0(\Omega) \cap L^4(\Omega)$:
\begin{align*}
 \int^T_0 \|u_t\|_{V^*}^{4/3} \, \d t
& \leq C \int^T_0 \left( \|\eta\|_2^{4/3} + \|\Delta u\|_{V^*}^{4/3} + \|u^3\|_{V^{*}}^{4/3} + \|u\|_2^{4/3} \right) \, \d t \nonumber\\
& \leq C \int^T_0 \left( \|\eta\|_2^2 + \|\nabla u\|_2^2 + \|u\|_{L^4(\Omega)}^4 + \|u\|_2^2 + 1 \right) \, \d t.
%\label{e:utH1*}
 \end{align*}
 By Aubin-Lions-Simon's compactness lemma along with the compact embeddings $V \hookrightarrow L^2(\Omega) \equiv (L^2(\Omega))^* \hookrightarrow V^*$, it follows that
  \begin{alignat*}{4}
   u_n &\to u \quad    &&\mbox{ weakly star in } L^\infty(0,T;L^2(\Omega)),\\
   & &&\mbox{ weakly in } W^{1,4/3}(0,T;V^*) \cap L^2(0,T;H^1_0(\Omega)) \cap L^4(0,T; L^4(\Omega)),\\
   & &&\mbox{ strongly in } L^2(0,T;L^2(\Omega)) \cap C([0,T];V^*),\\
   \eta_n &\to \eta \quad &&\mbox{ weakly star in } L^\infty(0,T;L^2(\Omega)).
  \end{alignat*}
Moreover, $u \in C_w([0,T];L^2(\Omega))$ and $u(0) = u_0$. Let $\delta \in (0,T)$ be arbitrarily fixed. Then it follows from \eqref{e4-2}, \eqref{e5-2} and \eqref{e6-3} for $u_n$ that
  \begin{alignat*}{4}
   u_n &\to u \quad &&\mbox{ strongly in } C([\delta,T];L^2(\Omega)),\\
   t^{1/2} \partial_t u_n &\to t^{1/2}u_t &&\mbox{ weakly in } L^2(0,T;L^2(\Omega)),\\
   t^{1/2}u_n &\to t^{1/2}u &&\mbox{ weakly star in } L^\infty(0,T;H^1_0(\Omega) ),\\
   t^{1/4} u_n &\to t^{1/4} u &&\mbox{ weakly in } L^\infty(0,T;L^4(\Omega)),\\
   t\left(- \Delta u_n + u_n^3\right) &\to t \left(- \Delta u + u^3\right) \quad &&\mbox{ weakly star in } L^\infty(0,T;L^2(\Omega)),
  \end{alignat*}
  and hence, $u_t + \eta - \Delta u + u^3 = \lam u$ a.e.~in $\Omega \times (0,T)$. Here we used the demiclosedness of maximal monotone operators to identify the limit. Moreover, from the arbitrariness of $\delta > 0$, we see that $u \in C((0,T];L^2(\Omega))$. We claim that $u(t) \to u_0$ strongly in $L^2(\Omega)$ as $t \to 0_+$, which also implies $u \in C([0,T];L^2(\Omega))$. Indeed, since $u(t) \to u_0$ weakly in $L^2(\Omega)$ as $t \to 0_+$, by \eqref{e4-0} and \eqref{u0-i},
  $$
  \|u_0\|_2 \leq \liminf_{t \searrow 0} \|u(t)\|_2 \leq \limsup_{t \searrow 0} \|u(t)\|_2 \leq \|u_0\|_2,
  $$
  which concludes that $u(t) \to u_0$ strongly in $L^2(\Omega)$ as $t \searrow0$. Thus we obtain $u \in C([0,T];L^2(\Omega))$.

  Now, it remains to identify the limit $\eta$ of $\eta_n \in \partial \iI(\partial_t u_n)$. To this end, let $\vep \in (0,T)$ be a constant, and observe that 
  \begin{align*}
   \limsup_{n \to \infty} \int^T_\vep \left( \eta_n, \partial_t u_n \right) \, \d t
   &\stackrel{\eqref{pde}}\leq - \liminf_{n \to \infty} \int^T_\vep \|\partial_t u_n\|_2^2 \, \d t
   - \liminf_{n \to \infty} E(u_n(T))\\
   & \quad + \limsup_{n \to \infty} E(u_n(\vep)).
  \end{align*}
  By Aubin-Lions-Simon's compactness lemma along with the compact embedding $H^2(\Omega) \cap L^6(\Omega) \hookrightarrow H^1(\Omega) \cap L^4(\Omega)$, for any $\delta > 0$, we see that
  $$
  u_n \to u \quad \mbox{ strongly in } C([\delta,T];H^1_0(\Omega) \cap L^4(\Omega)),
  $$
  which particularly implies
  $$
  u_n(t) \to u(t) \quad \mbox{ strongly in } H^1_0(\Omega) \cap L^4(\Omega)
  $$
  for $t \in (0,\infty)$. Therefore for any $\vep > 0$, we conclude that
  $$
  E(u_n(\vep)) \to E(u(\vep)).
  $$
  Here we also remark that due to \eqref{Eu-ep}, $E(u(\vep))$ is estimated by
  $$
  E(u(\vep)) \leq \dfrac{C_1}2 \left( 1 + r^{2/3} \right) + \dfrac 1 {4\vep} \|u_0\|_2^2 \quad \mbox{ for any } \ \vep > 0.
  $$
  It follows that
\begin{align*}
   \limsup_{n \to \infty} \int^T_\vep \left( \eta_n, \partial_t u_n \right) \, \d t
 &\leq - \int^T_\vep \|\partial_t u\|_2^2 \, \d t
 - E(u(T)) + E(u(\vep))\\
 &= \int^T_\vep \left( \eta , u_t \right) \, \d t,
\end{align*}
  and therefore, due to Minty's trick (see~\cite{HB1}), we conclude that $\eta \in \partial \iI(u_t)$ a.e.~in $\Omega \times (\vep,T)$ (see \S \ref{Ss:ii}). Since one can also take $\vep > 0$ arbitrarily close to zero, the desired conclusion is obtained. %Finally, by \eqref{e6} and \eqref{e6-1} with $\|\eta_0\|_2^2$ and $u_0$ replaced by $r$ and $u(\vep) \in H^2(\Omega) \cap H^1_0(\Omega) \cap L^6(\Omega)$, respectively, we also ensure that
  %$$
  %u \in L^\infty(\vep,T;H^2(\Omega) \cap L^6(\Omega)), \quad u_t \in L^2(\vep,T;H^1_0(\Omega)).
  %$$
  As for the energy inequalities, the idea of derivation is basically same as the proof of (ii).

 \subsection{Proof of (iv)}
 Let $u = u(x,t)$ be a solution to (P) which also solves \eqref{pde2}--\eqref{ic2}. By Theorem \ref{T:reform}, it is uniquely determined by $u_0$ (and actually exists). Therefore by the proofs of (i)--(iii) of Theorem \ref{A:T:ex}, $u(x,t)$ satisfies the energy inequalities which have already been verified in the preceding subsections and is also obtained as a limit of unique solutions $u_\lambda$ to \eqref{P-approx:2} as $\lambda \to 0$.

  \smallskip
  \noindent
  {\bf Energy Inequality \eqref{e3}.} By \eqref{eta-repre}, for each $s \in [0,T)$ at which $\eta(s)$ satisfies \eqref{pde}, we can construct a solution to (P) with the initial datum $u(s)$ as above and deduce by \eqref{eta-ini2} and the uniqueness of solutions that
\begin{equation}\label{e3pre}
 \|\eta(t)\|_2^2 \leq \|\eta(s)\|_2^2 \quad \mbox{ for a.e. } \ t \in (s,T).
\end{equation}
We remark that the set of $t \in (s,T)$ at which \eqref{e3pre} is satisfied may depend on the choice of $s$. We further claim that
\begin{equation}\label{e3again}
 \|\eta(t)\|_2^2 \leq \|\eta(s)\|_2^2 \quad \mbox{ for a.e. } (s,t) \in \{(\sigma, \tau) \in [0,T]^2 \colon \sigma \leq \tau\}
\end{equation}
(hereafter, we also simply write \eqref{e3} instead of \eqref{e3again}). Indeed, the subset $I = \{(\sigma, \tau) \in [0,T]^2 \colon \sigma \leq \tau, \ \|\eta(\tau)\|_2 > \|\eta(\sigma)\|_2\}$ is (Lebesgue) measurable due to the measurability of $t \mapsto \|\eta(t)\|_2$. Hence since $I_\sigma := \{ \tau \in [\sigma, T] \colon (\sigma,\tau) \in I\}$ has Lebesgue measure zero, so is $I$ by Fubini-Tonelli's lemma. Thus \eqref{e3again} follows.

  \smallskip
  \noindent
  {\bf Energy Inequalities \eqref{e6} and \eqref{e6-0}.} Similarly, we can also prove by uniqueness that
 \begin{align*}
  \lefteqn{
  \int^t_s \left( \|\nabla u_\tau\|_2^2 + \dfrac 3 4 \left\| \dfrac \d {\d t} \left( |u|u \right)\right\|_2^2 \right) \, \d \tau
  }\\
&+ \dfrac 1 2 \left\| \partial \psi(u(t)) - \lam u(t) \right\|_2^2 + \lam E(u(t))
\leq \dfrac 1 2 \left\| \partial \psi(u(s)) - \lam u(s) \right\|_2^2 + \lam E(u(s))\\
  &\quad \quad \mbox{ for a.e. } \ 0 < s < t < T.
 \end{align*}
  In particular, the function $t \mapsto (1/2) \left\| \partial \psi(u(t)) - \lam u(t) \right\|_2^2 + \lam E(u(t))$ is non-increasing, and hence, it is differentiable a.e.~in $(0,T)$. Dividing both sides by $t - s$ and taking a limit as $s \nearrow t$, we obtain \eqref{e6}. Furthermore \eqref{e6-0} also follows in a similar way.

 \smallskip
 \noindent
 {\bf Energy Inequality 2.} Here we suppose that $u_0$ satisfies \eqref{u0-iii}. Multiplying \eqref{pde:va} by $v_\lambda$, we have
$$
\dfrac 1 2 \dfrac \d {\d t} \|v_\lambda\|_2^2 + \int_\Omega (\partial_t \eta_\lambda) v_\lambda \, \d x + \left( \dfrac \d {\d t} \partial \psi_\lambda(u_\lambda), v_\lambda \right) = \lam \|v_\lambda\|_2^2.
$$
Here we remark that
$$
\int_\Omega (\partial_t \eta_\lambda) v_\lambda \, \d x = \dfrac \d {\d t} \iI^*(\eta_\lambda) = 0
$$
by $v_\lambda \in \partial \iI^*(\eta_\lambda)$. Therefore we find by \eqref{ut*dpsit} that
\begin{equation*}%\label{e2a}
\dfrac 1 2 \dfrac \d {\d t} \|v_\lambda\|_2^2 + \|\nabla (J_\lambda u_\lambda)_t\|_2^2 + \dfrac 3 4 \left\| \dfrac \d {\d t} \left( |J_\lambda u_\lambda| J_\lambda u_\lambda \right)\right\|_2^2 \, \d x \leq \lam \|v_\lambda\|_2^2.
\end{equation*}
To apply the convergence obtained so far (e.g.~\eqref{sharp2-1} and \eqref{sharp2-2}) and employ the weak lower semicontinuity of norms,
$$
\dfrac 1 2 \dfrac \d {\d t} \left( e^{-2\lam t} \|v_\lambda\|_2^2 \right)
+ e^{-2\lam t} \|\nabla (J_\lambda u_\lambda)_t\|_2^2 %+ 3 e^{-2\lam t} \int_\Omega (J_\lambda u_\lambda)^2 \left[ (J_\lambda u_\lambda)_t \right]^2 \, \d x
\leq 0.
$$
Integrate both sides over $(s,t)$, pass to the limit as $\lambda \to 0$, divide both sides of the resulting inequality by $t - s$ and take the limit as $s \nearrow t$. Then one can obtain
$$
\dfrac 1 2 \dfrac \d {\d t} \left( e^{-2\lam t} \|u_t\|_2^2 \right)
+ e^{-2\lam t} \|\nabla u_t \|_2^2 %+ 3 e^{-2\lam t} \int_\Omega u^2 u_t^2 \, \d x
\leq 0
\quad \mbox{ for a.e. } \ 0 < t < T,
$$
which implies \eqref{e2}. Thus all energy inequalities (for $u_0$ satisfying \eqref{u0-iii}) obtained in \S \ref{S:e} along with (iii) of Theorem \ref{T:ex} have been rigorously reproduced. This completes the proof.

 \section{Proof of Theorem \ref{T:reform}}\label{A:reform}

As mentioned in Remark \ref{R:inaccuracy}, the arguments in \S \ref{S:ref} include formal computations. In order to justify them, we recall again the approximate problems \eqref{P-approx:2} whose solutions are sufficiently smooth in time and reproduce the arguments in a rigorous fashion. Throughout this section, let $((0,T),\M_t,\mu_t)$, $(\Omega,\M_x,\mu_x)$ and $(Q,\M_{x,t},\mu_{x,t})$ be the measure spaces of Lebesgue measures with respect to $t$, $x$ and $(x,t)$, respectively. Moreover, for any $A \in \M_{x,t}$, we write
\begin{align*}
 A_x &:= \{ t \in (0,T) \colon (x,t) \in Q \} \ \mbox{ for each } x \in \Omega,\\
 A_t &:= \{ x \in \Omega \colon (x,t) \in Q \} \ \mbox{ for each } t \in (0,T).
\end{align*}
Then $A_x \in \M_t$ for $\mu_x$-a.e.~$x \in \Omega$ and $A_t \in \M_x$ and for $\mu_t$-a.e.~$t \in (0,T)$, by Fubini-Tonelli's lemma.

Let $u_\lambda$ be the solution of \eqref{P-approx:2} for $u_0 \in D(\partial \psi) = H^2(\Omega) \cap H^1_0(\Omega) \cap L^6(\Omega)$ and let $\eta_\lambda$ be the section of $\partial \iI(\partial_t u_\lambda)$. We recall that $u_\lambda \in C^{1,1}([0,T];L^2(\Omega))$ and $\eta_\lambda, \partial \psi_\lambda(u_\lambda) \in C^{0,1}([0,T];L^2(\Omega))$. Hence $\eta_\lambda$ and $\partial \psi_\lambda(u_\lambda)$ are differentiable $\mu_t$-a.e.~in $(0,T)$ with values in $L^2(\Omega)$. Moreover, (by taking a continuous representation of $\eta_\lambda$) it holds that
$$
u_\lambda(t) = u_\lambda(s) + \int^t_s \partial_\tau u_\lambda(\tau) \, \d \tau,
\quad \eta_\lambda(t) = \eta_\lambda(s) + \int^t_s \partial_\tau \eta_\lambda(\tau) \, \d \tau \ \mbox{ in } L^2(\Omega)
$$
for any $t,s \in [0,T]$. Moreover, recall that $u_\lambda$ and $\eta_\lambda$ satisfy \eqref{P-approx:2} in $L^2(\Omega)$ for all $t \in [0,T]$.

Since both sides of \eqref{P-approx:2} are differentiable $\mu_t$-a.e.~in $(0,T)$, as in the proof of Lemma \ref{L:eta-dec}, for any $r \in (1,2)$ and $\zeta \in L^q(\Omega)$, $\zeta \geq 0$ with $q \in (1,\infty)$ satisfying $1/q + r/2 = 1$, one observes that
$$
\int_\Omega \zeta(x) |\eta_\lambda(x,t)|^r \, \d x \leq \int_\Omega \zeta(x) |\eta_\lambda(x,s)|^r \, \d x \quad \mbox{ if } \ t \geq s
$$
for all $t,s \in [0,T]$. Since $\eta_\lambda \in C^{0,1}([0,T];L^2(\Omega))$ and $\zeta|\eta_\lambda|^{r-2}\eta_\lambda \in C([0,T];L^2(\Omega))$, we deduce that the function $t \mapsto \int_\Omega \zeta(x)|\eta_\lambda(x,t)|^r \ \d x$ belongs to $C^{0,1}([0,T])$, and hence,
$$
\dfrac \d {\d t} \int_\Omega \zeta(x) |\eta_\lambda(x,t)|^r \, \d x \leq 0 \quad \mbox{ for $\mu_t$-a.e. } t \in (0,T).
$$
From the arbitrariness and nonnegativity of $\zeta$, it follows that
$$
r |\eta_\lambda(\cdot,t)|^{r-2}\eta_\lambda(\cdot,t) \partial_t \eta_\lambda(\cdot,t) = \partial_t |\eta_\lambda(\cdot,t)|^r \leq 0 \ \mbox{ $\mu_x$-a.e.~in } \Omega
$$
for $\mu_t$-a.e.~$t \in (0,T)$. Hence $\partial_t\eta_\lambda(t) \geq 0$ $\mu_x$-a.e.~in $\Omega$ for $\mu_t$-a.e.~$t \in (0,T)$.

For each $t \in (0,T)$, define the set $\Omega_t \in \M_x$ of $x \in \Omega$ satisfying 
\begin{enumerate}
 \item[(i)] $u_\lambda$ and $\eta_\lambda$ satisfy \eqref{P-approx:2} at $(x,t)$,
 \item[(ii)] $u_\lambda$, $\eta_\lambda$ and $\partial \psi_\lambda(u_\lambda)$ are partially differentiable in $t$ at $(x,t)$,
 \item[(iii)] the following identities hold at $(x,t)$:
       \begin{align*}
	u_\lambda(x,t) &= u_\lambda(x,0) + \int^t_0 \partial_\tau u_\lambda(x,\tau) \, \d \tau,\\
	\eta_\lambda(x,t) &= \eta_\lambda(x,0) + \int^t_0 \partial_\tau \eta_\lambda(x,\tau) \, \d \tau,
       \end{align*}
 \item[(iv)] $\partial_t \eta_\lambda(x,t) \geq 0$ at $(x,t)$.
\end{enumerate}
Then $\mu_x(\Omega \setminus \Omega_t) = 0$ for $\mu_t$-a.e.~$t \in (0,T)$. Define the set $Q_1 \in \M_{x,t}$ of $(x,t) \in Q$ satisfying (i)--(iv) above. Then noting that $\Omega_t = (Q_1)_t$, we find by Fubini-Tonelli's lemma that $Q_1$ has full measure, i.e., $\mu_{x,t}(Q \setminus Q_1) = 0$.
\begin{comment}
Indeed, suppose on the contrary that $\mu_{x,t}(Q \setminus Q_1) > 0$. Then $\mu_x((Q \setminus Q_1)_t) > 0$ for all $t \in I$ with some subset $I \subset (0,T)$ of positive measure. However, there arises a contradiction to the fact that $(Q \setminus Q_1)_t \subset \Omega \setminus \Omega_t$.
\end{comment}

Now, set
$$
\Omega_1 := \left\{ x \in \Omega \colon \mu_t ((0,T) \setminus (Q_1)_x) = 0\right\}.
$$
First, we claim that $\Omega_1 \in \M_x$. Indeed, by Fubini-Tonelli's lemma, we see that $(Q_1)_x \in \M_t$ for $\mu_x$-a.e.~$x \in \Omega$, and moreover, the function $x \mapsto \mu_t( (Q_1)_x )$ is $\M_x$-measurable. Since the function
$$
x \mapsto \mu_t ((0,T) \setminus (Q_1)_x) = T - \mu_t((Q_1)_x)
$$
is also $\M_x$-measurable, the level set $\Omega_1$ of the $\M_x$-measurable function above also belongs to $\M_x$.

Next, we claim that $\mu_x(\Omega \setminus \Omega_1) = 0$. Indeed, note that
\begin{align*}
N_1 &:= \left\{
(x,t) \in Q \colon x \in \Omega\setminus \Omega_1, \ t \in (0,T) \setminus (Q_1)_x
\right\} \subset Q \setminus Q_1,\\
N_2 &:= \left\{
(x,t) \in Q \colon x \in \Omega_1, \ t \in (0,T) \setminus (Q_1)_x
\right\} \subset Q \setminus Q_1.
\end{align*}
Since the measure space $(Q, \M_{x,t}, \mu_{x,t})$ is complete, the sets $N_1$ and $N_2$ also belong to $\M_{x,t}$. In particular, we obtain $\mu_{x,t}(N_1) = \mu_{x,t}(N_2) = 0$. By Fubini-Tonelli's lemma,
$$
\int_{\Omega \setminus \Omega_1} \mu_t ((0,T) \setminus (Q_1)_x) \, \d x = \mu_{x,t}(N_1) = 0,
$$
% Ito, Seizo, Chap. II, Th 7.3, p.36.
which implies $\mu_x(\Omega \setminus \Omega_1) = 0$ by $\mu_t ((0,T) \setminus (Q_1)_x) > 0$ for a.e.~$x \in \Omega \setminus \Omega_1$.

Furthermore, the set
$$
Q_2 := \{(x,t) \in Q \colon x \in \Omega_1, \ t \in (Q_1)_{x}\} = \left( \Omega_1 \times (0,T) \right) \setminus N_2
$$
is $\M_{x,t}$-measurable and has full measure, i.e., $\mu_{x,t}(Q \setminus Q_2) = 0$; indeed, applying Fubini-Tonelli's lemma and combining all the facts obtained so far, we conclude that
$$
0 \leq \mu_{x,t}(Q \setminus Q_2) \leq \mu_{x,t}(N_2) + \mu_x(\Omega \setminus \Omega_1) T = 0.
$$
Moreover, $Q_2$ is a subset of $Q_1$. Now, we are ready to prove Theorem \ref{T:reform}. Let $(x_0,t_0) \in Q_2$ be fixed. In case $u_\lambda(x_0,t_0) = u_0(x_0)$, by $\partial \iII(u_\lambda(x_0,t_0)) = (-\infty,0]$, the relation
\begin{equation}
 \partial_t u_\lambda + \partial \iII (u_\lambda) + \partial \psi_\lambda(u_\lambda) - \lam u_\lambda \ni 0\label{pde2-lam}
\end{equation}
holds true at $(x_0,t_0)$. In case $u_\lambda(x_0,t_0) > u_0(x_0)$ (then $\partial \iIIo(u_\lambda(x_0,t_0)) = \{0\}$), since $(x_0,t_0), (x_0,t) \in Q_1$ for $\mu_t$-a.e.~$t \in (0,T)$, there exists $t_1 \in (0,t_0) \cap (Q_1)_{x_0}$ such that $\partial_t u_\lambda(x_0,t_1) > 0$, which implies $\eta_\lambda(x_0,t_1) = 0$. Moreover, it follows that
\begin{align*}
0 \geq \eta_\lambda (x_0,t_0)
&= \int^{t_0}_{t_1} \partial_\tau \eta_\lambda(x_0,\tau) \, \d \tau + \eta_\lambda(x_0,t_1)\\
&= \int_{(t_1,t_0) \cap (Q_1)_{x_0}} \partial_\tau \eta_\lambda(x_0,\tau) \, \d \tau + \eta_\lambda(x_0,t_1)
\geq 0,
\end{align*}
which implies $\eta_\lambda(x_0,t_0) = 0$. Thus \eqref{pde2-lam} is satisfied at $(x_0,t_0)$. In particular, the section $\eta_\lambda$ of $\partial \iI(\partial_t u_\lambda)$ also belongs to the set $\partial \iIIo(u_\lambda)$ for $\mu_{x,t}$-a.e.~in $Q$.

Recalling the convergence as $\lambda \to 0$ of solutions $u_\lambda$ for \eqref{P-approx:2} obtained in \S \ref{S:Aex}, one can deduce that
$$
\eta_\lambda \to \eta \in \partial \iII(u) \quad \mbox{ weakly star in } L^\infty(0,T;L^2(\Omega))
$$
by the demiclosedness of maximal monotone operators. Hence the limit $u$ of $u_\lambda$ also solves \eqref{pde2}--\eqref{ic2}.

We next consider the case that $u_0 \in \overline{D_r}^{L^2}$. Then let us take $u_{0,n} \in D_r$ such that
\begin{equation}\label{u0n-conv}
u_{0,n} \to u_0 \quad \mbox{ strongly in } L^2(\Omega).
\end{equation}
Let $u_n$ be the solution of (P) with the initial data $u_{0,n}$ such that $u_n$ also solves \eqref{pde2}--\eqref{ic2} with $u_0$ replaced by $u_{0,n}$. In particular, the section $\eta_n$ of $\partial \iI(\partial_t u_n)$ as in \eqref{EQ} also belongs to the set $\partial \iIIn(u_n)$. On the other hand, by \eqref{u0n-conv}, it holds that $\iIIn \to \iII$ on $L^2(\Omega)$ in the sense of Mosco (see Lemma \ref{L:Mosco} below and~\cite{Attouch}). Therefore from the convergence of $u_n$ obtained in \S \ref{Ss:ii}, we deduce that the limit $\eta$ of $\eta_n$ fulfills
$$
\eta \in \partial \iII(u) \quad \mbox{ for a.e. } t \in (0,T).
$$

\begin{lemma}\label{L:Mosco}
Let $u_{0,n}, u_0 \in L^2(\Omega)$ be such that $u_{0,n} \to u_0$ strongly in $L^2(\Omega)$. Then $\iIIn \to \iII$ on $L^2(\Omega)$ in the sense of Mosco.
\end{lemma}

\begin{proof}
 \emph{Existence of recovery sequences}: For each $w \in D(\iII)$, define a recovery sequence $w_n := w - u_0 + u_{0,n} \in L^2(\Omega)$. Then $w_n \geq u_{0,n}$, which gives $w_n \in D(\iIIn)$. Moreover, $w_n \to w$ strongly in $L^2(\Omega)$ by assumption.

 \noindent
 \emph{Weak liminf convergence}: Let $w_n, w \in L^2(\Omega)$ be such that $w_n \to w$ weakly in $L^2(\Omega)$. We shall check that
 $$
 \liminf_{n \to \infty} \iIIn(w_n) \geq \iII(w).
 $$
 In case $\liminf_{n \to \infty} \iIIn(w_n) = \infty$, the assertion follows immediately. In case $\liminf_{n \to \infty} \iIIn(w_n) < \infty$, up to a (not relabeled) subsequence, $\iIIn(w_n)$ is bounded. Hence $w_n \geq u_{0,n}$ a.e.~in $\Omega$. For each $z \in C^\infty_0(\Omega)$ satisfying $z \geq 0$, it follows that
 $$
 \int_\Omega w_n z \, \d x \geq \int_\Omega u_{0,n} z \, \d x.
 $$
 Letting $n \to \infty$ and using the arbitrariness of $z$, we conclude that $w \geq u_0$ a.e.~in $\Omega$. Thus $\iII(w) = 0$, and hence, the assertion follows. Consequently, $\iIIn \to \iII$ on $L^2(\Omega)$ in the sense of Mosco.
\end{proof}

As in \S \ref{S:ref}, one can prove the uniqueness of $u$ which solves both \eqref{pde2}--\eqref{ic2} and (P) as well as the equivalence of two problems. This completes the proof of Theorem \ref{T:reform}.

\end{document}